\documentclass[11pt]{amsart}

\usepackage[english]{babel}     
\usepackage[utf8]{inputenc}    
\usepackage[sort]{cite}           
\usepackage{fancyhdr}          
\usepackage[a4paper, left=1.2in]{geometry} 
\usepackage{xspace}            
\usepackage{dsfont}
\usepackage{amsxtra}
\usepackage{amsmath}
\usepackage{amscd}
\usepackage{amssymb}
\usepackage{amsfonts}
\usepackage[all]{xy}
\usepackage{mathrsfs}
\usepackage{amsthm} 
\usepackage{color}
\usepackage{upgreek}
\usepackage{enumerate}
\usepackage{latexsym}
\usepackage{todonotes}
\usepackage{hyperref}
\usepackage{stmaryrd}
\usepackage{nicefrac}
\usepackage{euscript}
\usepackage{mathabx}
\usepackage{enumitem}
\usepackage[foot]{amsaddr}
\usepackage{relsize,exscale}
\usepackage{tikz-cd}     
\usepackage[titletoc]{appendix}

\newcommand{\arxiv}[1]{\href{http://arxiv.org/abs/#1}{\tt
    arXiv:\nolinkurl{#1}}}

\newtheorem{theorem}{Theorem}[subsection]
\newtheorem{lemma}[theorem]{Lemma}
\newtheorem{corollary}[theorem]{Corollary}
\newtheorem{proposition}[theorem]{Proposition}
\newtheorem{claim}[theorem]{Claim}
\newtheorem{conjecture}[theorem]{Conjecture}

\newtheorem*{theorem-A}{Theorem A}
\newtheorem*{theorem-B}{Theorem B}
\newtheorem*{theorem-C}{Theorem C}
\newtheorem*{theorem-D}{Theorem D}
\newtheorem*{conjecture-A}{Conjecture A}
\newtheorem*{conjecture-B}{Conjecture B}
\newtheorem*{conjecture-C}{Conjecture C}
\newtheorem*{conjecture-D}{Conjecture D}

\newtheorem{Theorem}{Theorem}[section]
\newtheorem{Lemma}[Theorem]{Lemma}

\newtheorem{Conjecture}[Theorem]{Conjecture}

\theoremstyle{remark} 
\newtheorem{remark}[theorem]{Remark}
\newtheorem{example}[theorem]{Example}

\def\A{\mathrm{A}}

\def\G{\mathrm{G}}

\def\K{\mathrm{K}}

\def\P{\mathrm{P}}
\def\Q{\mathrm{Q}}
\def\R{\mathrm{R}}
\def\SS{\mathrm{S}}
\def\T{\mathrm{T}}

\def\X{\mathrm{X}}

\def\Z{\mathrm{Z}}

\def\bbA{\mathbb{A}}

\def\bbC{\mathbb{C}}

\def\bbN{\mathbb{N}}

\def\bbQ{\mathbb{Q}}

\def\bbZ{\mathbb{Z}}

\def\frakg{\mathfrak{G}}

\def\frakL{\mathfrak{L}}
\def\frakM{\mathfrak{M}}
\def\frakN{\mathfrak{N}}

\def\calA{\mathcal{A}}
\def\calB{\mathcal{B}}
\def\calC{\mathcal{C}}

\def\calE{\mathcal{E}}

\def\calH{\mathcal{H}}
\def\calI{\mathcal{I}}

\def\calM{\mathcal{M}}

\def\calO{\mathcal{O}}
\def\calP{\mathcal{P}}

\def\calS{\mathcal{S}}
\def\calT{\mathcal{T}}
\def\calU{\mathcal{U}}

\def\calW{\mathcal{W}}
\def\calX{\mathcal{X}}
\def\calY{\mathcal{Y}}
\def\calZ{\mathcal{Z}}

\def\frakg{\mathfrak{g}}
\def\frakh{\mathfrak{h}}

\def\frakl{\mathfrak{l}}

\def\frakn{\mathfrak{n}}

\def\frakp{\mathfrak{p}}

\def\fraks{\mathfrak{s}}

\def\fraku{\mathfrak{u}}

\def\bfA{\mathbf{c}}

\def\bfs{\mathbf{s}}

\def\bfA{\mathbf{A}}

\def\bfH{\mathbf{H}}

\def\bfK{\mathbf{K}}

\def\bfN{\mathbf{N}}

\def\bfQ{Q}

\def\bfU{\mathbf{U}}

\def\bfY{\mathbf{Y}}

\def\r{\rangle}

\def\ind{\operatorname{ind}\nolimits}

\def\NKHA{\operatorname{\bfN\bfK}\nolimits}
\def\KHA{\operatorname{\bfK}\nolimits}
\def\SH{\operatorname{SH}\nolimits}

\def\Coh{{\calC oh}}

\def\sg{{\operatorname{sg}\nolimits}}

\def\b{{\operatorname{b}\nolimits}}
\def\c{{\operatorname{c}\nolimits}}
\def\e{{\operatorname{e}\nolimits}}

\def\top{{\operatorname{top}\nolimits}}
\def\Tot{{\operatorname{Tot}\nolimits}}
\def\perf{{\operatorname{perf}\nolimits}}

\def\rk{{\operatorname{rk}\nolimits}}
\def\gr{{\operatorname{gr}\nolimits}}

\def\-{{\operatorname{-}\!}}

\def\Im{\operatorname{Im}\nolimits}
\def\Ker{\operatorname{Ker}\nolimits}

\def\Spec{\operatorname{Spec}\nolimits}

\def\Tr{\operatorname{Tr}\nolimits}

\def\Sym{\operatorname{Sym}\nolimits}
\def\Tr{\operatorname{Tr}\nolimits}
\def\wt{\operatorname{wt}\nolimits}

\def\res{\mathrm{res}}

\numberwithin{itemcounter}{subsection}
\numberwithin{equation}{section}
\title[K-theoretic Hall algebras and super quantum groups]{K-theoretic Hall algebras, quantum groups 
and super quantum groups}

\author{M. Varagnolo$^1$} 
\address{\scriptsize{$^1$~CY Cergy Paris Universit\'e,  95302 Cergy-Pontoise, France,
UMR8088 (CNRS), ANR-18-CE40-0024 (ANR).}}
\author{E. Vasserot$^2$} 
\address{\scriptsize{$^2$~Universit\'e de Paris, 75013 Paris, France, UMR7586 (CNRS), 
ANR-18-CE40-0024 (ANR), Institut Universitaire de France (IUF).
}}

\begin{document}

\begin{abstract}
We prove that the K-theoretic Hall algebra of a preprojective algebra of affine type is isomorphic
to the positive half of a quantum toroidal quantum group. 
We also compare super toroidal quantum groups of type A
with some K-theoretic Hall algebras of quivers with potential introduced recently by Padurariu 
using categories of singularities of some Landau-Ginzburg models.
The proof uses both a deformation of the K-theoretic Hall algebra and
a dimensional reduction which allows us to compare the KHA of preprojective algebras to the
K-theoretic Hall algebras of a quiver with potential.
\end{abstract}

\maketitle


\section{Introduction and notation}
\subsection{Introduction}
Let $Q$ be a Kac-Moody quiver.
Let $X$ be the variety of all finite dimensional representations of $\bfQ$ and $\calX$ its moduli stack.
Lusztig has introduced a graded additive monoidal subcategory of the bounded constructible derived category 
$D^\b_c(\calX)$ of the stack $\calX$. Its split Grothendieck group is isomorphic 
to the integral form of the quantum unipotent enveloping algebra of type $\bfQ$.
This construction can be viewed as a categorification of Ringel-Hall algebras.
There are no known analogues of Ringel-Hall algebras for super quantum groups.
K-theoretic Hall algebras yield another geometric construction of quantum groups.
In this theory quantum groups are categorified by
the derived category $D^\b(T^*\calX)$ of coherent sheaves on the cotangent dg-stack of $\calX$.
The dg-stack $T^*\calX$ is identified with the 
derived moduli stack of finite dimensional representations of the preprojective algebra $\Pi_Q$, 
which is a Calabi Yau algebra of dimension 2 whenever $Q$ is not of Dynkin type.
In an informal way, the relation between Ringel-Hall algebras and K-theoretic Hall algebras is given by
characteristic cycles of constructible sheaves.
Our goal is to prove that K-theoretic Hall algebras permit to categorify both quantum groups and super 
quantum groups.

\smallskip

We first consider the K-theoretic Hall algebra of the 
preprojective algebra $\Pi_Q$ of the quiver $Q$
following Schiffmann-Vasserot's approach.
For each dimension vector $\beta$, 
let $X_\beta\subset X$ be the variety of all representations of dimension $\beta$.
A linear group $G_\beta\times\bbC^\times$ acts on the cotangent $T^*X_\beta$.
The  $\bbC^\times$-action is prescribed by a weight function.
The $G_\beta$-action  is Hamiltonian.
The moment map $\mu_\beta:T^*X_\beta\to \frakg_\beta$ is $G_\beta\times\bbC^\times$-equivariant,
with $\bbC^\times$ acting with the weight $2$ on $\frakg_\beta$. 
The moduli stack $\calX_\beta$ 
is the quotient of $X_\beta$ by $G_\beta$.
The cotangent dg-stack of $\calX_\beta$ is the derived fiber product 
$$T^*\calX_\beta=[T^*X_\beta\times^R_{\frakg_\beta}\{0\}\,/\,G_\beta\times\bbC^\times].$$
We define $T^*\calX$ to be the sum of all dg-stacks $T^*\calX_\beta$.
The category $D^\b(T^*\calX)$ is a graded triangulated category
with a graded triangulated monoidal structure.
The internal grading is given by the $\bbC^\times$-action.
Thus the rational Grothendieck group 
$\KHA(\Pi_Q)=\G_0(T^*\calX)$ of  $D^\b(T^*\calX)$ has an algebra structure over $\bbQ[q,q^{-1}]$.
We call it the K-theoretic Hall algebra  of $\Pi_Q$.
We'll mostly consider another similar algebra $\NKHA(\Pi_Q)$, which is the Grothendieck group of
the derived category of all coherent sheaves of 
dg-modules over the dg-stack $T^*\calX$ which are supported on a closed 
substack of nilpotent elements. The role of the algebra $\NKHA(\Pi_Q)$ is explained later in the text.
 Note however that $\NKHA(\Pi_Q)$ and $\KHA(\Pi_Q)$ are
two integral forms of the same  algebra over $\bbQ(q)$. More precisely, the $\bbQ[q,q^{-1}]$-module 
$\NKHA(\Pi_Q)$ is free and we have 
$$\NKHA(\Pi_Q)\otimes_{\bbQ[q,q^{-1}]}\bbQ(q)=\KHA(\Pi_Q)\otimes_{\bbQ[q,q^{-1}]}\bbQ(q).$$

\smallskip

Let $\bfU^+$ be the Lusztig integral form of the quantum enveloping algebra 
of the loop algebra of the positive part of the Kac-Moody algebra of type $\bfQ$, i.e., 
a Drinfeld half of a toroidal quantum group in the terminology of\cite{GKV}.
Our first goal is to compare $\bfU^+$ and $\NKHA(\Pi_Q).$
This work is motivated by \cite{SVV}, where
the graded triangulated category $D^\b(T^*\calX)$ for a quiver $Q$ of type $A_1$ is compared
with a derived category of  graded
modules over the quiver-Hecke algebra of affine type $A^{(1)}_1$.
We'll prove the following theorem which is a more general version of \cite[thm~7.1]{SVV}.

\smallskip

\begin{theorem-A}\label{thm-A}
Let $Q$ be a Kac-Moody quiver with
a normal weight function of $\bar Q$.
\begin{itemize}[leftmargin=8mm]
\item[$\mathrm{(a)}$] 
There is a surjective $\bbN^{Q_0}$-graded algebra homomorphism 
$\phi:\bfU^+\to\NKHA(\Pi_Q).$
\item[$\mathrm{(b)}$] 
If $\bfQ$ is of finite or affine type but not of type $A_1^{(1)}$, then the map $\phi$ is injective.
\qed
\end{itemize}
\end{theorem-A}

\smallskip

The method of the proof goes back to some previous work of Schiffmann-Vasserot
which considers the case of Borel-Moore homology.
There,  it was observed that an essential step to prove a similar isomorphism in Borel-Moore homology was 
to deform the algebra $\NKHA(\Pi_Q)$ by allowing a big torus action,
and to prove torsion freeness of this deformation relatively to a commutative symmetric algebra
 generated by the classes of some universal bundles.
We prove a similar torsion freeness in K-theory, in Proposition \ref{prop:TFT}.
One of the main ingredient of the proof in \cite{SVV} was Davison's dimensional reduction theorem for critical 
cohomology in \cite{D17a}. 
Our proof relies on a K-theoretic version of the dimensional reduction theorem
which is proved by Isik in \cite{I12}, see also  \cite{HL20}, \cite{H17}, in the context of categories of singularities.

\smallskip

It is essential to twist the multiplication in $\NKHA(\Pi_Q)$ to compare it with $\bfU^+$.
In the case of Borel-Moore homology and Yangians the analogue of this twist is just a normalization of signs
in the definition of the product and is often omitted.

\smallskip

The restriction $Q\neq A_1^{(1)}$ is not essential, and is due to the fact that the work \cite{E03}, 
that we use in our proof, is done under the same restriction. It is possible to overpass it. To do so we must
modify the Drinfeld relations of the toroidal quantum group. We do not do it here in order that the paper remains 
of a reasonable length. 

\smallskip

If $Q$ is a Kac-Moody quiver which is not of affine type we do not know
how to modify claim (b). It is conjectured that, in this case, the toroidal Lie algebra should be replaced by the
Lie algebra defined by Maulik and Okounkov using stable envelopes.

\smallskip

We also discuss briefly the case of quivers which are not of Kac-Moody type in \S\ref{sec:GQ}.
Note that, an important step in the proof is the injectivity of an algebra homomorphism
$\rho$ from a deformed algebra $\NKHA(\Pi_Q)_\T$ to a shuffle algebra under some 
conditions on the deforming torus $\T$. More precisely 
$\T=\bbC^\times\times\A$ with an homogeneous action as in \eqref{homogeneous}, \eqref{actionT}
such that \eqref{large} holds, or $\T$ contains a one-parameter subgroup satisfying \eqref{COND}
if the quiver has no edge loops.

\smallskip

We'll also consider the deformed 
K-theoretic Hall algebra $\bfK(Q,W)_\T$ of a quiver with potential.
We'll use the approach of Padurariu, which is inspired by
the cohomological Hall algebras of Kontsevich-Soibelmann. 
In this setting, the algebra $\bfK(Q,W)_\T$ is defined as the Grothendieck group of the
category of singularities of some Landau-Ginzburg model attached to the pair $(Q,W)$.
This category is Calabi Yau of dimension 3.
The relation with the preprojective K-theoretic Hall algebra is given by dimensional reduction which is recalled later,
see also \cite{I12}, \cite{HL20}, \cite{H17}.
As explained in \S\ref{sec:DHAW}, there is a canonical algebra homomorphism
$i:\KHA(Q,W)_\T\to\KHA(Q,0)_\T$ from a K-theoretic Hall algebra of a quiver $Q$ with potential $W$ to the 
K-theoretic Hall algebra of $Q$ with the zero potential.
The injectivity of the map $\rho$ above is equivalent to the injectivity of the map $i$ in 
the triple quiver case.

\smallskip

In the second part of this paper we consider the Drinfeld half $\pmb\calU^+_\T$ of the
super toroidal quantum groups of type $A$ introduced in \cite{BM19}. 
In this case, the relevant geometric algebra is the deformed 
K-theoretic Hall algebra $\bfK(Q,W)_\T$ of a quiver with potential
which are described in
\S\ref{sec:super}. 
We'll prove the following.

\begin{theorem-B}
There is an $\bbN^{Q_0}$-graded algebra homomorphism
$\psi:\pmb\calU^+_\T\to \KHA(Q,W)_\T$.
\qed
\end{theorem-B}

We do not have a torsion freeness statement similar to the one of toroidal quantum groups.

\begin{conjecture-C}
The map $\psi$ is injective.
\qed
\end{conjecture-C}

\smallskip

Both theorems have analogues in Borel-Moore homology.
The analogue of Theorem A is already known, it is
recalled in Theorem \ref{thm:C} below and its proof uses
the geometric results in \cite{SV18}.
The analogue of Theorem B is stated in Theorem \ref{thm:D} below
and is new.

\smallskip

\subsection{Relations with recent works}
After the first version of our paper was posted, several related works appeared that we would like to mention briefly.
A revised version of \cite{P19} in \cite{P21} contains more details on the construction and basic properties of 
K-theoretic Hall algebras of quivers with potentials. Using our embedding $\rho$ in Corollary 
\ref{cor:loops}, Negut gives a system of generators of
the generic deformed K-theoretic Hall algebra $\KHA(\Pi_Q)_\T\otimes_{\R_\T}\K_\T$ in \cite{N21}.
Finally, \cite{RR21} contains a new geometric construction of Yangians of
Lie superalgebras which we expect to be related
to our critical cohomology approach in \S\ref{sec:super} and \S\ref{sec:SY}.

\smallskip

\subsection{Notation}
All schemes or Artin stacks are taken over the field $\bbC$.
A \emph{quotient stack} is an Artin stack equivalent to the stack $[X\,/\,G]$ 
where $X$ is an algebraic space
and $G$ a linear group acting on $X$. The stack $[\Spec(\bbC)\,/\,G]$ is the classifying stack of $G$, which we 
shall denote by $BG$. Unless specified otherwise, all the quotient stacks we'll consider here are of the form
$\calX=[X\,/\,G]$ with $X$ a $G$-quasi-projective scheme $X$. 
Such a stack satisfies the resolution property by the work of Thomason, i.e., 
every coherent sheaf on $\calX$ is a quotient of a vector bundle.

\smallskip

Given a small Abelian or exact category $\calC$, let $D^\b(\calC)$ be its bounded derived category.
For any Artin stack $\calX$ we abbreviate $D^\b(\calX)=D^\b(Coh(\calX))$, where $Coh(\calX)$
is the category of coherent sheaves over $\calX$.
Let $D^\perf(\calX)$ be the triangulated category of perfect complexes in $D^\b(\calX)$.
The category of singularities of $\calX$ is the Verdier quotient 
$$D^\sg(\calX)=D^\b(\calX)\,/\,D^\perf(\calX).$$

\smallskip

Let $\K_0(\calT)$ be the Grothendieck group of a triangulated category $\calT$ and
$\K_0(\calT)_\bbQ=\K_0(\calT)\otimes\bbQ$ its rational form. 
We have $\K_0(D^\b(\calC))=\K_0(\calC)$.
We define 
$$\K_0(\calX)=\K_0(D^\perf(\calX))_\bbQ
\quad,\quad
\G_0(\calX)=\K_0(D^\b(\calX))_\bbQ
\quad,\quad
\G_0^\sg(\calX)=\K_0(D^\sg(\calX))_\bbQ.
$$
For any integer $i\geqslant 0$ 
we'll also use the $i$th Grothendieck group $\G_i(\calX)$ of the category $\Coh(\calX)$.
Its topological version $\G_i(\calX)^\top$ is defined as in \cite{T88} (by inverting the Bott element).

\smallskip

For any torus $\T$ and any linear group $G$ we abbreviate $G_\T=G\times\T$. 
If $\T=\bbC^\times$ we write $G_c=G\times\bbC^\times$. 
Let $\R_G=\G_0(BG)_\bbQ$ and 
$\K_G$ be its fraction field.
Let $X_*(G)$ and $\X^*(G)$ be the set of characters and cocharacters of $G$.
We abbreviate
$$\R=\R_{\bbC^\times}=\bbQ[q,q^{-1}]
\quad,\quad
\K=\K_{\bbC^\times}=\bbQ(q).$$

\medskip

\emph{Acknowledgments.} 
Initial stages of this work were partly inspired by discussions with R. Rouquier and P. Shan.
We would like to thank them for these discussions.
We are also grateful to T. Padurariu for useful discussions relative to the material in
\S \ref{sec:DHAW}.

\medskip

\section{K-theoretic Hall algebras and quantum loop algebras}
\subsection{Coherent sheaves on derived moduli stacks}
\subsubsection{Reminder on dg-stacks}
Let first begin by a down-to-earth introduction to dg-stacks and their derived categories.
By a dg-scheme we mean a pair $(X\,,\,\calA)$ where $(X\,,\,\calO_X)$ is a (separated of finite type) scheme and 
$\calA$ is a quasi-coherent  connective graded-commutative $\calO_X$-dg-algebra on $X$,
where connective means non-positively graded.
When the sheaf $\calA$ is clear from the context we may abbreviate
$X=(X\,,\,\calA)$.
By a dg-stack we mean a pair $(\calX\,,\,\calA)$ consisting of
an Artin stack $\calX$ with a quasi-coherent  connective graded-commutative
$\calO_\calX$-dg-algebra. 
\smallskip

Given a dg-scheme $(X\,,\,\calA)$ with an action of a linear group $G$ on $X$ 
such that $\calA$, its multiplication and its
differential are $G$-equivariant, we get a dg-stack over the quotient stack $\calX=[X\,/\,G]$.
The corresponding sheaf of dg-algebras is denoted by the same symbol $\calA$.
The truncation of $(\calX\,,\,\calA)$ is the closed substack $\pi_0(\calX,\calA)=\Spec\calH^0(\calA)$.
The truncation is the left adjoint functor to the canonical inclusion 
of the category of stacks into the category of dg-stacks.
If $\calX=[X\,/\,G]$ is a quotient stack, we'll always
 assume that  $X$ is a $G$-scheme which is $G$-equivariantly quasi-projective.
Then, the resolution property insures that every $G$-equivariant 
quasi-coherent sheaf on $X$ is a quotient of a $G$-equivariant quasi-coherent sheaf which is flat over $\calO_X$.
We further impose that $\calA^{\,0}=\calO_{\calX}$ and that the dg-stack $(\calX\,,\,\calA)$
is bounded in the following sense : the $\calO_{\calX}$-module $\calA$
is locally free of finite rank. Under this hypothesis, we'll say that the dg-stack $(\calX\,,\,\calA)$ is 
\emph{almost smooth} if the scheme $X$ is smooth.

\smallskip

A quasi-coherent sheaf over the dg-stack
$(\calX\,,\,\calA)$ 
is a quasi-coherent $\calO_\calX$-module with an action of the dg-algebra 
$\calA$.
It is coherent if the cohomology sheaf $\calH(\calE)$
is coherent over $\calH(\calA)$.
Since we assumed that the dg-stack $(\calX\,,\,\calA)$ is bounded, 
a quasi-coherent sheaf $\calE$ is coherent
if and only if $\calH(\calE)$ is coherent over $\calH^0(\calO_\calX)$.
Let $\Coh(\calX\,,\,\calA)$ and $Q\Coh(\calX\,,\,\calA)$
be the categories of all coherent and quasi-coherent sheaves over a dg-stack $(\calX\,,\,\calA)$.
Let $D^\b(\calX\,,\,\calA)$ and $D(\calX\,,\,\calA)$ be their 
derived categories, i.e.,
the triangulated category obtained by inverting all quasi-isomorphisms 
in the homotopy categories of $\Coh(\calX\,,\,\calA)$  and $Q\Coh(\calX\,,\,\calA)$.
These triangulated categories do not 
not depend on the choice of the representative of  the dg-stack $(\calX\,,\,\calA)$ in its quasi-isomorphism class.
The triangulated category $D^\b(\calX\,,\,\calA)$ admits the \emph{standard $t$-structure}, 
whose heart is equivalent to
the category of coherent sheaves on the truncation $\pi_0(\calX,\calA)$. 
Let $D^\perf(\calX\,,\,\calA)$ be the smallest full triangulated subcategory of 
$D^\b(\calX\,,\,\calA)$ containing the object $\calA$ which is closed under direct summands.
The construction of the category of singularities
of an Artin stack carries over to the case of dg-stacks.
The category of singularities of the dg-stack $(\calX,\calA)$ is the Verdier quotient 
$$D^\sg(\calX\,,\,\calA)=D^\b(\calX\,,\,\calA)\,/\,D^\perf(\calX\,,\,\calA).$$
We may also define it as a dg-quotient, but we'll not need this in this paper.

\smallskip

A morphism of dg-stacks $f:(\calX\,,\,\calA)\to(\calY\,,\,\calB)$ is the data of a morphism of stacks
$f_\flat:\calX\to\calY$ and a morphism of sheaves of dg-algebras $(f_\flat)^*\calB\to\calA$.
Under the assumptions above, the left and right derived functors
$Lf^*:D(\calY\,,\,\calB)\to D(\calX\,,\,\calA)$
and
$Rf_*:D(\calX\,,\,\calA)\to D(\calY\,,\,\calB)$ exist and satisfy some standard properties such as the base change,
see, e.g., \cite{BR12}.
We'll say that the morphism $f$ is \emph{almost smooth} if $f_\flat$
is smooth and $\calA$ is flat as an $(f_\flat)^*\calB$-module. It is proper if $f_\flat$ is proper.
We'll also use the notion of \emph{quasi-smooth} morphism.
The morphism $f$ is quasi-smooth if its relative cotangent complex
is perfect of Tor-amplitude $[-1,0]$.
If $f$ is almost-smooth or quasi-smooth then the functor 
$Lf^*$ yields a functor $D^\b(\calY\,,\,\calB)\to D^\b(\calX\,,\,\calA)$.
If $f$ is proper then $Rf_*$ yields a functor $D^\b(\calX\,,\,\calA)\to D^\b(\calY\,,\,\calB)$.
We'll use the terminology \emph{l.c.i.~closed immersion} to mean a quasi-smooth closed immersion of dg-schemes
or dg-stacks, i.e., a quasi-smooth map whose truncation is a closed immersion.
By \cite[prop~2.1.10]{AG}, a morphism $f:(X,\calA)\to (Y,\calB)$ of 
dg-schemes is an l.c.i.~closed immersion if and only if
Zariski-locally on $Y$ there is a Cartesian square
$$\xymatrix{(X,\calA)\ar[d]\ar[r]^f& (Y,\calB)\ar[d]\\\{0\}\ar[r]&\bbC^n}.$$

\smallskip

Most of the dg-schemes we'll consider here are affine. 
In this case, the notions recalled above can be expressed in terms of dg-algebras.
Let $\bfA$ be a bounded finitely generated connective graded-commutative 
dg-algebra over $\bbC$, with an action
of a linear algebraic group $G$ by dg-algebra automorphisms.
Let $R\Spec(\bfA)=(X\,,\,\calA)$ be the locally dg-ring space associated with $\bfA$.
It is a $G$-equivariant dg-scheme with structural sheaf
$\calA=\bfA\otimes_{\bfA^0}\calO_X$ such that $X=\Spec(\bfA^0)$.
We consider the quotient stack $\calX=[X\,/\,G]$ with  the dg-stack
$(\calX\,,\,\calA)=[R\Spec(\bfA)\,/\,G]$.
Then the triangulated category $D^\b(\calX\,,\,\calA)$  is the derived category $D^\b(\bfA\rtimes G)$ 
of all dg-modules over the crossed product dg-algebra $\bfA\rtimes G$ 
whose cohomology is finitely generated over the graded algebra 
$H^0(\bfA).$
The morphisms are residue classes of morphisms modulo null-homotopic 
ones with all quasi-isomorphisms inverted.
There is a standard t-structure on the triangulated category $D^\b(\bfA\rtimes G)$ whose heart 
consists of the $G$-equivariant 
dg-modules $M$ over $\bfA$
such that $H^0(M)$ is finitely generated over $H^0(\bfA)$ and $H^i(M)=\{0\}$ for $i\neq 0$.
The assignment $M\mapsto H^0(M)$ defines an equivalence of Abelian categories from this heart to the category
of finitely generated modules over the crossed product $H^0(\bfA)\rtimes G$.

\smallskip

We'll also use \emph{graded} versions of the above, 
corresponding to affine dg-stacks with a $\bbC^\times$-action.
Let $G_c=G\times\bbC^\times$ with $G$ a linear algebraic group.
A $G_c$-equivariant dg-algebra is the same as a graded $G$-equivariant dg-algebra, i.e.,
a  $G$-equivariant  
bigraded algebra with a differential of bidegree
$(1,0)$ satisfying the Leibniz rule. 
For each bidegree $(i,j)$, we call $i$ the \emph{cohomological degree} and
$j$ the \emph{internal degree}. 
A $G_c$-equivariant dg-module $M$ is the same as a bigraded $G$-equivariant module.
Let $M^i_j$ be the component with cohomological degree $i$ and internal degree $j$.
The grading shift functors
$[\text{-}]$ and $\langle\text{-}\rangle$ are given by
$$(M[a]\langle b\rangle)^i_j=M^{i+a}_{j+b}
\quad,\quad
d_{M[a]\langle b\rangle}=(-1)^ad_M.$$
We may also write $$q M=M\langle -1\rangle.$$
Note that the graded dg-algebra $\bfA$ is graded-commutative relatively to the cohomological degree,
i.e., for any homogeneous elements $x$, $y$ of cohomological degrees $|x|$, $|y|$, we have
\begin{align}\label{grcom}x y=(-1)^{|x|\cdot |y|}y x.\end{align}

\smallskip

\begin{example}
For any graded vector space $V$ the
graded-symmetric algebra $\SS(V)$ is
the quotient of the tensor algebra of $V$ by the relations \eqref{grcom}.
For any scheme $X$ and any graded vector bundle $E$ on $X$ with section $s$, let
$$K(X,E,s)=\big(\SS(E^*[1])\,,\,d_s\big)$$
be the Koszul cochain complex whose differential is the contraction along $s$.
It is a graded dg-algebra over $\calO_X$.
The \emph{derived zero locus} of $s$ is the graded dg-scheme 
$$R\Z(X,E,s)=R\Spec K(X,E,s).$$
Its truncation is the scheme theoretic zero locus $\{s=0\}$ in $X$.
Given a linear group $G$ acting on $X$, $E$ such that $s$ is $G$-invariant,
we can form the vector bundle $\calE=[E\,/\,G\,]$ over the stack 
$\calX=[X\,/\,G\,]$. The derived zero locus of $s$ is 
the  graded dg-stack 
$$R\Z(\calX,\calE,s)=[R\Z(X,E,s)\,/\,G\,].$$
\end{example}

\smallskip

Finally, let us recall the \emph{dimensional reduction} following Isik \cite{I12}, see also  \cite{HL20}, \cite{H17}.
Let $G$ be a linear group acting on a scheme $X$
with quotient stack $\calX=[X\,/\,G]$.
Let $E\to X$ be a $G$-equivariant vector bundle over $X$.
Let $s$ be a $G$-invariant section of $E$.
Let $\calY=R\Z(\calX,E,s)$ be the derived zero locus of $s$.
Let $\Tot(E^*)$ be the total space of the dual vector bundle and
$w:\Tot(E^*)\to\bbC$ the $G$-invariant function given by $s$. 
Let the group $\bbC^\times$ acts on $\Tot(E^*)$ by dilatation of weight 2 along the fibers
of the projection $\pi:\Tot(E^*)\to X$.
Thus, the function $w$ is linear along the fibers $\pi$, and it is homogeneous of weight 2 relatively 
to the $\bbC^\times$-action.
The function $w$ is equivariant under the $\bbC^\times$-action by dilatation on $E^*$.
Let $G_c=G\times\bbC^\times$.
Let $\calZ_c$ be the derived zero locus of $w$, i.e.,
$$\calZ_c=R\Z([\Tot(E^*)/G_c],\calO_{[\Tot(E^*)/G_c]},w).$$ 
By \cite[thm.~4.6, rem.~4.7]{I12}, see also  \cite{HL20}, \cite{H17} for the equivariant case,
we have a triangulated equivalence
\begin{align}\label{dimred0}D^\b(\calY)=D^\sg(\calZ_c).\end{align}
The triangulated category $D^\sg(\calZ_c)$ is the \emph{graded} singularity category
of the derived stack 
$$\calZ=R\Z([\Tot(E^*)/G],\calO_{[\Tot(E^*)/G]},w).$$ 
The usual singularity category is $D^\sg(\calZ)$.
We may also write $D^{\gr,\sg}(\calZ)=D^\sg(\calZ_c)$.

\smallskip

\subsubsection{K-theory}\label{sec:K-theory}
Next, we consider the Grothendieck groups.
We abbreviate
\begin{align*}\G_0(\calX\,,\,\calA)=\K_0(D^\b(\calX\,,\,\calA))
\quad,\quad
\K_0(\calX\,,\,\calA)=\K_0(D^\perf(\calX\,,\,\calA)).
\end{align*}
Given a closed substack $\calZ$ of $\calX$, let $D^\b(\calX\,,\,\calA)_\calZ$ be the full triangulated subcategory of
$D^\b(\calX\,,\,\calA)$ consisting of all modules which are acyclic over the open subset $\calX\,\setminus\,\calZ$.
Set
$$\G_0(\calX\,,\,\calA)_\calZ=\K_0(D^\b(\calX\,,\,\calA)_\calZ).$$
By devissage, the obvious map yields the following isomorphism, see, e.g., 
\cite[cor~5.11]{J12}
\begin{align}\label{devissage}\G_0(\calZ\,,\,\calA|_\calZ)=\G_0(\calX\,,\,\calA)_\calZ.\end{align}
Since the heart of the standard t-structure on $D^\b(\calX\,,\,\calA)$ is equivalent to
the category of coherent sheaves on the truncation $\pi_0(\calX,\calA)$, see, e.g.,
\cite[\S 4.1]{KY}, we have a group isomorphism
\begin{align}\label{truncation}\G_0(\calX\,,\,\calA)=\G_0(\pi_0(\calX\,,\,\calA)).\end{align}

\smallskip

In the affine case, assume that
$\bfA$ is a bounded finitely generated connective graded-commutative 
dg-algebra over $\bbC$, with an action
of a linear algebraic group $G$ by dg-algebra automorphisms.
Let $R\Spec(\bfA)=(X\,,\,\calA)$ and the stack $\calX$ and the dg-stack
$(\calX\,,\,\calA)$ be as above.
Then the Grothendieck group
$\K_0(D^\b(\bfA\rtimes G))$ is the Grothendieck group 
of the category of all finitely generated $H^0(\bfA)\rtimes G$-modules.
We'll abbreviate $$\G_0(\bfA\rtimes G)=\K_0(D^\b(\bfA\rtimes G)).$$

\smallskip

Finally, we consider the $K$-theoretic analogue of the dimensional reduction.
Let the vector bundle $E$, the derived
stacks $\calZ_c$, $\calY$, and the functions $w_1,\dots,w_n$ be as in the section above. 
Assume that $E$ is a trivial bundle, i.e., we have $\Tot(E^*)=X\times\bbC^n$ and
$w=\sum_{i=1}^n w_ix_i$ where $x_1,\dots,x_n$ are the coordinates on $\bbC^n$
and $w_1,\dots,w_n$ are $G$-invariant regular functions on $X$. 
Let $\{w_1=\dots=w_n=0\}$ be the zero set in $\calX$ of the functions $w_1,\dots,w_n$.
From \cite[cor.~3.13]{T21} and the equivalence between (graded) matrix factorizations and
(graded) singularity categories we deduce that the Grothendieck groups
$\G_0^\sg(\calZ)$ and $\G_0^\sg(\calZ_c)$ are canonically identified.
Thus, from \eqref{dimred0} we get
the following isomorphism of Grothendieck groups
\begin{align}\label{Kdimred0}\G_0(\{w_1=\dots=w_n=0\})=\G_0^\sg(\calZ).\end{align}

\smallskip

\subsubsection{Quivers and weight functions}\label{sec:TQ1}
Let $Q=(Q_0\,,\,Q_1)$ be any quiver.
We'll always assume that $Q_0$ and $Q_1$ are finite sets.
Let $Q'$ be the quiver opposite to $Q$ and $\bar\bfQ$ its double.
We have $Q'=(Q_0\,,\,Q'_1)$ and $\bar Q=(Q_0\,,\,\bar Q_1)$ with $\bar Q_1=Q_1\sqcup Q'_1$.
Let $h\mapsto h'$ denote the bijection $Q_1\to Q'_1$ which reverses all arrows. 
The triple quiver of $Q$
is the quiver $\widetilde Q=(Q_0\,,\,\widetilde Q_1)$ such that $\widetilde Q_1=\bar Q_1\sqcup \omega(Q_0)$,
where $\omega:Q_0\to\widetilde Q_1$ is the injection 
which associates to each vertex $i$ a loop $\omega_i:i\to i$.
The source of an arrow $h$ is denoted by $h_s$, its target by $h_t$.
To simplify the notation, we'll use the following convention :
\emph{unless specified otherwise we abbreviate $i=h_s$ and $j=h_t$, i.e., the arrow is $h:i\to j$.}
Given a total ordering of the vertices of $Q$,
we set
$$Q_1^<=\{h\in Q_1\,;\,i<j\}.$$

\smallskip

Let $\T$ be a complex torus.
A \emph{weight function} of the quiver $Q$ in the abelian group
$\X^*(\T)$ is a map $Q_1\to\X^*(\T)$ such that  $h\mapsto q_h$.
Assume first that $\T=\bbC^\times$ and let $q$ be the weight one character. 
A weight function of the double quiver $\bar Q$
is \emph{homogeneous} if we have $q_hq_{h'}=q^2$ for each arrow $h\in\bar Q_1$. 
If the quiver $Q$ is of Kac-Moody type 
we'll say that the weight function of $\bar Q$ is \emph{normal} if it
is homogeneous and the following additional assumptions hold :
\begin{align}\label{normal}
\begin{split}
\bar Q_1^<&=Q_1,\\
\{q_h\,;\,h:i\to j\,,\,h\in \bar Q_1\}&=\{q^{2+a_{i,j}},q^{4+a_{i,j}},\dots,q^{-a_{i,j}}\}
\text{ with $a_{i,j}\in-\bbN$.}
\end{split}
\end{align}
We'll set $a_{i,i}=2$ for all $i$.
Now, consider a torus of the form $\T=\bbC^\times\times\A$, where $\A$ is a subtorus, and
let $q$ be the first projection $\T\to\bbC^\times$.
A weight function of $Q$ in $\X^*(\T)$ is the same as a pair of weight functions
$h\mapsto q_h$ in $\X^*(\bbC^\times)$ and  $h\mapsto t_h$ in $\X^*(\A)$.
A weight function of the double quiver $\bar Q$ in $\X^*(\T)$
is homogeneous if we have
\begin{align}\label{homogeneous}
q_hq_{h'}=q^2\quad,\quad t_ht_{h'}=1\quad,\quad h\in\bar Q_1.
\end{align}
If $Q$ is of Kac-Moody type this weight function  is normal if \eqref{normal} 
and \eqref{homogeneous} hold.

\smallskip

\subsubsection{Moduli stacks of representations of quivers}\label{sec:TQ2}

Let $\bbC Q$, $\Pi_Q$ be the path algebra and the preprojective algebra of the quiver $Q$.
Let $\beta=\sum_{i\in Q_0}b_i\,\alpha_i$ be a dimension vector in $\bbN^{Q_0}$,
where $\alpha_i$ is the Dirac function at the vertex $i$. 
Let $X_\beta$ be the space of all $\beta$-dimensional 
representations of the path algebra.
A point of $X_\beta$ is a tuple of matrices $x=(x_h\,;\,h\in Q_1)$.  
The group 
$G_\beta=\prod_{i\in Q_0}GL_{b_i}$ acts on $X_\beta$ by conjugation.
Let $\bar X_\beta$ be the space of $\beta$-dimensional representations of 
the double quiver $\bar\bfQ$, and
$\widetilde X_\beta$ be the space of $\beta$-dimensional representations of the triple 
quiver $\widetilde\bfQ$.
We have $\bar X_\beta=T^*X_\beta=X_\beta\times X'_\beta$ 
and $\widetilde X_\beta=\bar X_\beta\times\frakg_\beta$ as graded representations of $G_\beta$.
We'll identify $\frakg_\beta$ with its dual via the trace.
The $G_\beta$-action on the cotangent bundle $T^*X_\beta$ is Hamiltonian.
The moment map 
$\mu_\beta$ is the quadratic map
$x\mapsto\sum_{h\in Q_1}[x_h\,,\,x_{h'}]$.
A weight function of  $Q$ in $\X^*(\bbC^\times)$ yields a 
$\bbC^\times$-action on  $X_\beta$ such that 
\begin{align}\label{action}u\cdot x_h=q_h(u)x_h
\quad,\quad u\in\bbC^\times, x\in X_\beta, h\in Q_1.\end{align}
So, the $G_\beta$-action on $X_\beta$ extends to   
a $G_{\beta,c}$-action and yields the quotient stack 
$$\calX_\beta=[\,X_\beta\,/\,G_{\beta,c}\,].$$
We define the stacks $\bar\calX_\beta$, $\widetilde\calX_\beta$ associated with the quivers
$\bar Q$, $\widetilde Q$ in a similar way.

\smallskip

Now, assume that the assignment 
$h\mapsto q_h$ is an homogeneous weight function of $\bar Q$.
Let $\frakg^*_\beta\langle-2\rangle$ denote the coadjoint representation with
the $\bbC^\times$-action of weight 2.
Then, the moment map $\mu_\beta:\bar X_\beta\to\frakg_\beta^*\langle -2\rangle$
is $G_{\beta,c}$-equivariant. 
We can view it as a section of the vector bundle 
$$E_\beta=[\bar X_\beta\times\frakg^*_\beta\langle -2\rangle\,/\,G_{\beta,c}]\to\bar\calX_\beta.$$
We define the following dg-stack 
\begin{align*}
T^*\calX_\beta=[\,\bar X_\beta\times^R_{\frakg_\beta^*}\{0\}\,/\,G_{\beta,c}\,]
=R\Z(\bar\calX_\beta\,,\,E_\beta\,,\,\mu_\beta).
\end{align*}
The truncation of $T^*\calX_\beta$ is
$\calM_\beta=\pi_0(T^*\calX_\beta).$
It is the closed substack of $\bar\calX_\beta$ given by
\begin{align*}\calM_\beta=[\,M_\beta\,/\,G_{\beta,c}\,]
\quad,\quad M_\beta=\mu_\beta^{-1}(0).\end{align*}
Taking the sum over all dimension vectors, we abbreviate
$$T^*\calX=\bigsqcup_\beta T^*\calX_\beta\quad,\quad
\calM=\bigsqcup_\beta \calM_\beta.$$
We call $T^*\calX$ the \emph{dg-moduli stack of representations} of the preprojective algebra $\Pi_Q$.
We have
$$D^\b(T^*\calX)=\bigoplus_\beta D^\b(T^*\calX_\beta)
\quad,\quad
\G_0(T^*\calX)=\bigoplus_\beta\G_0(T^*\calX_\beta).$$
The heart of the 
standard t-structure of $D^\b(T^*\calX)$ is equivalent to the category of coherent sheaves
$\Coh(\calM)$ as a graded Abelian category.

We now define a $G_{\beta,c}$-equivariant connective graded-commutative
graded dg-algebra $\bfA_\beta$ such that
$T^*\calX_\beta=[R\Spec(\bfA_\beta)\,/\,G_{\beta,c}\,].$
As a $G_\beta$-equivariant bigraded graded-commutative algebra we set
\begin{align}\label{A}
\bfA_\beta=\SS(\bar X_\beta^*)\otimes\SS(\frakg_\beta[1]\langle 2\rangle).
\end{align}
The  differential of $\bfA_\beta$ is the contraction with the element $\mu_\beta$ of 
$\SS^2(\bar X_\beta^*)\otimes\frakg_\beta^*\langle -2\rangle.$
We have
$$\bar\calX_\beta=[\Spec(\bfA_\beta^0)\,/\,G_{\beta,c}\,]
\quad,\quad
\calM_\beta=[\Spec(H^0(\bfA_\beta))\,/\,G_{\beta,c}\,]
\quad,\quad
D^\b(T^*\calX_\beta)=D^\b(\bfA_\beta\rtimes G_{\beta,c}).$$
The direct image by the closed embedding $\calM_\beta\subset T^*\calX_\beta$
yields a group isomorphism 
\begin{align}\label{isom0}
\G_0(\calM_\beta)=\G_0(T^*\calX_\beta)=
\G_0(\bfA_\beta\rtimes G_{\beta,c}).
\end{align}

\smallskip

\subsubsection{Nilpotent representations}\label{sec:nilp}
For each flag of $Q_0$-graded vector spaces
\begin{align}\label{flag}F=(\{0\}=F_0\subset F_1\subset\cdots\subset F_r=V)\end{align}
let $\bar X(F)$ be the set of all representations
$x=(x_h\,;\,h\in Q_1)$ such that
\begin{itemize}[leftmargin=8mm]
\item $x_h(F_p)\subset F_p$ for all $h\in Q_1$ and $p\in[1,r]$,
\item $x_h(F_p)\subset F_{p-1}$ for all $h\in Q'_1$ and $p\in[1,r]$.
\end{itemize}
The elements in the set $\bigcup_F\bar X(F)$, where the union runs over the flags
$F$ such that $F_p\,/\,F_{p-1}$ is concentrated on a single vertex for each $p$,
are called \emph{strongly seminilpotent representations} in \cite[def.~3.4]{SV18}.
If the quiver $Q$ is of Kac-Moody type, then the strongly seminilpotent representations are the same as the
nilpotent ones in Lusztig sense \cite{L91}.
We'll abbreviate nilpotent for strongly seminilpotent.
We define
$$
N_\beta=M_\beta\cap\bigcup_F\bar X(F)
\quad,\quad
\Lambda_\beta=[N_\beta\,/\,G_{\beta,c}]
\quad,\quad
\Lambda=\bigsqcup_\beta\Lambda_\beta
\quad,\quad
N=\bigsqcup_\beta N_\beta
.$$
Thus $\Lambda_\beta$ is a closed substack of $\calM_\beta$.

\smallskip

Let  $D^\b(\bfA_\beta\rtimes G_{\beta,c})_\Lambda$ be the derived category
of all $G_{\beta,c}$-equivariant dg-modules over $\bfA_\beta$ 
whose cohomology is finitely generated over the graded algebra 
$H^0(\bfA_\beta)$ and is supported on $N_\beta$.
The Grothendieck group $\G_0(\bfA_\beta\rtimes G_{\beta,c})_\Lambda$ 
of $D^\b(\bfA_\beta\rtimes G_{\beta,c})_\Lambda$ is an $\R$-module such that 
$q[M]=[M\langle -1\rangle]$ for each $\bfA_\beta\rtimes G_{\beta,c}$-module $M$.
Let $D^\b(T^*\calX)_\Lambda$ be the graded triangulated category
of all coherent dg-modules on $T^*\calX$ whose cohomology is 
supported on $\Lambda$, and $\G_0(T^*\calX_\beta)_\Lambda$ its Grothendieck group.
By \eqref{devissage}, \eqref{truncation} and \eqref{isom0} there is a canonical group isomorphism
\begin{align}\label{isom2}
\G_0(T^*\calX_\beta)_\Lambda=\G_0(\Lambda_\beta)=\G_0(\bfA_\beta\rtimes G_{\beta,c})_\Lambda.
\end{align}




\smallskip

\subsection{Derived Hall algebras}
\label{sec:DHAPQ}
In this section we define a monoidal structure on the category $D^\b(T^*\calX)$.
We use the description of $T^*\calX$ as a stacky quotient of a derived spectrum mentioned above
to get an elementary formulation for this product in terms of induction/restriction of dg-modules.

\subsubsection{Induction and restriction}
Fix dimension vectors $\alpha$,  $\gamma$ with $\beta=\alpha+\gamma$.
Set 
\begin{align*}
\frakg=\frakg_\beta
\quad,\quad
G=G_\beta
\quad,\quad
X_G=X_\beta
\quad,\quad
\bfA_G=\bfA_\beta
\quad,\quad
\mu_G=\mu_\beta,
\end{align*}
\begin{align*}
\frakl=\frakg_\alpha\times\frakg_\gamma
\quad,\quad
L=G_\alpha\times G_\gamma
\quad,\quad
X_L=X_\alpha\times X_\gamma
\quad,\quad
\bfA_L=\bfA_\alpha\otimes\bfA_\gamma.
\end{align*} 
We fix a flag $(F\subset V)$ of $Q_0$-graded vector spaces of dimensions 
$$\dim F=\gamma\quad,\quad\dim V=\beta.$$
Let $P$ be the stabilizer of $F$ in $G$.
Thus $L$ is a Levi subgroup of $P$.
Let $U$ be the unipotent radical of $P$. 
Let $X_P$ be the space of all representations in $X_G$ which preserve the subspace $F$,
and $X_U$ the subspace of all representations $x$ such $x(V)\subset F$ and $x(F)=\{0\}$.
Using the quivers $Q'$ or $\bar Q$ instead of $\bfQ$, we define similarly the spaces $X'_H$, 
$\bar X_H$ 
for the group $H=L$, $U$, $P$, $G$. 
Let $\frakh$ be the Lie algebra of $H$.
For $H=L,P,G$ we consider the dg-stack
$$T^*\calX_H=[\,\bar X_H\times^R_{\frakh}\{0\}\,/\,H_c\,].$$
There is an $H_c$-equivariant graded-commutative dg-algebra $\bfA_H$ 
such that
$$T^*\calX_H
=[R\Spec(\bfA_H)\,/\,H_c\,].$$
Hence, we have
\begin{align}\label{identification}
D^\b(T^*\calX_H)=D^\b(\bfA_H\rtimes H_c).
\end{align}
It is enough to describe $\bfA_P$.
As a $P_c$-equivariant graded-commutative algebra, we set
$$\bfA_P=\SS((\bar X_P)^*\oplus\frakp[1]\langle2\rangle).$$
The dual of the obvious inclusions $\bar X_P\subset\bar X_G$ and $\frakp^*\subset\frakg^*$
gives the surjective map
$$(\bar X_G)^*\oplus\frakg[1]\langle2\rangle
\to
(\bar X_P)^*\oplus\frakp[1]\langle2\rangle.$$
Taking the symmetric tensor, we get
a surjective bigraded algebra homomorphism 
$\bfA_G\to\bfA_P.$
The restriction $\mu_G|_{\bar X_P}$ is viewed as an element $\mu_P$ in
$\SS^2((\bar X_P)^*)\otimes\frakp^*.$
The differential of $\bfA_P$ is the contraction with this element.

\smallskip

The algebra homomorphism 
$\bfA_G\to\bfA_P$ gives a $P_c$-equivariant surjective graded dg-algebra homomorphism $\bfA_G\to\bfA_P$.
Taking the derived spectra, we get a proper dg-stack homomorphism
$p:T^*\calX_P\to T^*\calX_G$
and a direct image functor
\begin{align}\label{pdi}
Rp_*:D^\b(T^*\calX_P)\to D^\b(T^*\calX_G).
\end{align}
Under the identification \eqref{identification}, this functor may be viewed as a functor
$$\ind_P^G:D^\b(\bfA_P\rtimes P_c)\to D^\b(\bfA_G\rtimes G_c).$$
It is the induction of rational representations from $P$ to $G$, 
compare to the non derived case in Remark \ref{rem:induction} below.
The dual of the obvious projections $\bar X_P\to\bar X_L$ and $\frakp^*\to\frakl^*$ give the inclusion
$$(\bar X_L)^*\oplus\frakl[1]\langle2\rangle
\to
(\bar X_P)^*\oplus\frakp[1]\langle2\rangle,$$
hence a bigraded algebra homomorphism
$\bfA_L\to\bfA_P.$
The image of $\mu_P$ by the projection
$$\SS^2((\bar X_P)^*)\otimes\frakp^*
\to\SS^2((\bar X_P)^*)\otimes\frakl^*$$
coincides with the element
$\mu_L$ in $\SS^2((\bar X_L)^*)\otimes\frakl^*$.
Thus, there is a $P_c$-equivariant dg-algebra homomorphism $\bfA_L\to\bfA_P$
such that $\bfA_P$ is flat as an $\bfA_L$-module.
It yields an almost smooth dg-stack homomorphism
$q:T^*\calX_P\to T^*\calX_L$
and the pullback functor
\begin{align}\label{spb}
Lq^*:D^\b(T^*\calX_L)\to D^\b(T^*\calX_P)
\quad,\quad
M\mapsto \bfA_P\otimes_{\,\bfA_L}\!M.
\end{align}

\smallskip

\begin{remark}\label{rem:induction}
Let $X_H=\Spec(A_H)$ be an $H$-equivariant affine scheme, for $H=P,G$, and
$\calX_H=[X_H\,/\,H]$ be the corresponding quotient stack.
Let $p:\calX_P\to\calX_G$ be a stack homomorphism.
We have a direct image functor
$p_*:Q\Coh(\calX_P)\to Q\Coh(\calX_G)$
which can be described in the following way.
Geometrically $p_*$ is the composed functor
$$\xymatrix{Q\Coh(\calX_P)\ar[r]^-{\ind_P^G}&Q\Coh([G\times_PX_P\,/\,G])\ar[r]& Q\Coh(\calX_G)}$$
where the first map is the geometric induction as in \cite[\S 5.2.16]{CG} and the second one is the direct image by 
the action map $G\times_PX_P\to X_G$.
Algebraically, for any $A_P\rtimes P$-module $M$ we have
$\ind_P^GM=(M\otimes\bbC[G])^P$ as an $A_G\rtimes G$-module,
and the latter is the induction of rational representations as in
\cite[\S I.3.3]{J}.
\end{remark}

\smallskip

\subsubsection{Twist}\label{sec:twist}
In order to compare the K-theoretic Hall algebra with quantum groups we must twist the multiplication.
To do that, we introduce a twisted version $Rp_\circ$ of the functor $Rp_*$.
Fix a total ordering of the vertices of $Q$. 
The vector space $\bar X_U^<$ of all representations $x\in\bar X_U$ such that
$x_h=0$ unless $h\in\bar Q_1^<$ admits a representation of the group
$P_c$ such that $P$ acts in the obvious way and 
$\bbC^\times$ as in \eqref{action}. Hence, the one dimensional bigraded vector space
\begin{align}\label{Omega}
\Omega_U
=\det(\bar X_U^<)^{-1}[\dim \bar X_U^<]
\end{align}
is a graded character of the group $P_c$.
We may abbreviate $\Omega=\Omega_U$.
We call it a \emph{twist}.

\smallskip

We define the \emph{twisted induction functor} from bigraded representations of $P$ to 
bigraded representations of $G$ to be
\begin{align}\label{twisted-ind}\ind_{P,\circ}^G=
\ind_P^G\big(\,\text{-}\,\otimes\, \Omega\big).
\end{align}
Replacing the induction by the twisted induction in \eqref{pdi} yields the functor 
\begin{align}\label{tpdi}
Rp_\circ:D^\b(T^*\calX_P)\to D^\b(T^*\calX_G)
\quad,\quad
M\mapsto \ind_{P,\circ}^GM,
\end{align}
Composing \eqref{tpdi} with the functor $M\mapsto \bfA_P\otimes_{\,\bfA_L}\!M$ from
$L_c$-equivariant dg-modules over $\bfA_L$ 
to $P_c$-equivariant dg-modules over $\bfA_P$, which is exact, we get
a graded triangulated functor
\begin{align}\label{indcoh}
\circ=Rp_\circ Lq^*:D^\b(T^*\calX_L)\to
D^\b(T^*\calX_G).
\end{align}

\smallskip

\begin{proposition}
The category $D^\b(T^*\calX)$, equipped with the functor $\circ$, is graded triangulated monoidal. 
\end{proposition}

\begin{proof}
The proof follows from base change.
\end{proof}

\smallskip

We call the graded triangulated monoidal category $(D^\b(T^*\calX)\,,\,\circ\,)$ the
\emph{derived Hall algebra} of the preprojective algebra $\Pi_Q$.

\smallskip

\subsubsection{Nilpotence}
Let $\calS\subset\calM$ be any closed substack which is
closed under extensions of $\bbC\bar Q$-modules. The functor \eqref{indcoh} yields a functor 
$\circ:D^\b(T^*\calX_L)_\calS\to D^\b(T^*\calX_G)_\calS$, hence
a graded monoidal triangulated category $(D^\b(T^*\calX)_\calS\,,\,\circ\,)$.
The nilpotent representations form a Serre subcategory.
Thus, setting $\calS=\Lambda$,
 the functor $\circ$ in \eqref{indcoh} yields a graded triangulated monoidal category
$(D^\b(T^*\calX)_\Lambda\,,\,\circ)$.

\smallskip

\subsection{K-theoretic Hall algebras and quantum  groups}

In this section we introduce the quantum groups and the K-theoretic Hall algebra (KHA).
In order to compare them, we introduce several versions of the KHA that we'll need later :
the KHA of a preprojective algebra,
its nilpotent, deformed and twisted versions, and the KHA of a quiver with and without potential.
Then, we compare
them in the case of a triple quiver.

\subsubsection{Quantum groups}\label{sec:QLG}

Let $Q$ be a Kac-Moody quiver.
Let $\frakg_\Q$ be the Kac-Moody algebra associated with $\bfQ$ and $\frakn_\bfQ$ its positive half.
Let $L\frakg_\Q$ and $L\frakn_\bfQ$ be their loop algebras, i.e., we set
$L\frakg_\Q=\frakg_\Q[s,s^{-1}]$ and
$L\frakn_\bfQ=\frakn_\bfQ[s,s^{-1}],$
with the obvious Lie bracket.
Let $a_{i,j}$ be the coefficients of the Cartan matrix of the quiver $Q$.
Let $\bfU^+_\K$ be the quantum enveloping algebra of $L\frakn_\bfQ$. 
It is the $\K$-algebra generated by elements 
$\e_{i,n}$ with $i\in Q_0$, $n\in\bbZ$ subject to the following defining relations :

\begin{itemize}[leftmargin=8mm]
\item[$\mathrm{(a)}$] for all $i$, $j$ we have the Drinfeld relation given by
$$(q^{a_{i,j}}z-w)\,\e_i(z)\,\e_j(w)=(z-q^{a_{i,j}}w)\,\e_j(w)\,\e_i(z),$$
\item[$\mathrm{(b)}$]  for $i\neq j$ we have the Serre relation given by
$$\Sym_z\sum_{k=0}^l(-1)^k\left[\begin{matrix}l\cr k\end{matrix}\right]_q
\e_i(z_1)\dots\e_i(z_k)\,\e_j(w)\,\e_i(z_{k+1})\dots\e_i(z_l)=0\quad,\quad l=1-a_{i,j}.$$
\end{itemize}
In these relations we used the generating series
$\e_i(z)=\sum_{n\in\bbZ}\e_{i,n}z^{-n}$, the averaging operator $\Sym_z$ in the variables $z_i$'s and the following commutator
\begin{align*}
[\,a\,,\,b\,]_z=ab-zba.
\end{align*}
We'll say that an $\R$-module $M$ is \emph{separated at 1} if the submodule 
${}^\infty M$ such that ${}^\infty M=\bigcap_{n>0}(q-1)^nM$ is $\{0\}$.
We define $\widetilde\bfU^+$ to be the $\R$-subalgebra of $\bfU^+_\K$ generated by the quantum divided powers
$(\e_{i,n})^{(r)}$ with $i\in Q_0$, $n\in\bbZ$ and $r\in\bbZ_{>0}$. 
Let $\bfU^+$ be the maximal quotient of $\widetilde\bfU^+$ which is separated at 1, i.e., we define 
$\bfU^+$ to be the quotient of $\widetilde\bfU^+$ by the two-sided ideal ${}^\infty\widetilde\bfU^+$.

\smallskip

\subsubsection{Quantum groups and Shuffle algebras}\label{sec:SH}
Let $Q$ be any Kac-Moody quiver with a total ordering of the set $Q_0$.
Fix $\alpha, \gamma\in\bbN^{Q_0}$ with sum $\beta$. 
Let $a_i,$ $b_i,$ $c_i$ be the $i$th entries of the dimension vectors $\alpha,\beta,\gamma$.
We have $b_i=a_i+c_i$ for all $i\in Q_0$.
Let $W_\beta$ be the Weyl group of $G_\beta$.
It is a product of symmetric groups $W_\beta=\prod_{i\in Q_0}S_{b_i}$.
We identify the ring $\R_{G_{\beta,c}}$ with the following ring of symmetric Laurent polynomials
\begin{align}\label{R}
\R_{G_{\beta,c}}=\R[\,x_{i,s}\,,\,x_{i,s}^{-1}\,;\,i\in Q_0\,,\,s\in[1,b_i]\,]^{W_\beta}.
\end{align}
We have 
\begin{align*}
\K_{G_{\beta,c}}=\K(\,x_{i,s}\,;\,i\in Q_0\,,\,s\in[1,b_i]\,)^{W_\beta}.
\end{align*}
We write
$$W_{\alpha,\gamma}=\prod_{i\in Q_0}
\{w\in S_{b_i}\,;\,w(r)<w(s)\ \text{if}\ 0<r<s\leqslant a_i\ \text{or}\ a_i<r<s\leqslant b_i\}.$$
The averaging operator 
$\Sym_{x}:\R_{G_{\alpha,c}}\otimes_\R\R_{G_{\gamma,c}}\to\R_{G_{\beta,c}}$
takes $f$ to the sum $\sum_{w\in W_{\alpha,\gamma}}w(f).$ 
Fix a total ordering of the vertices of $Q$. We consider the rational function 
$$\zeta^\bullet_{\alpha,\gamma}\in\K(\,x_{i,s}\,;\,i\in Q_0\,,\,s\in[1,b_i]\,)^{W_\alpha\times W_\gamma}$$ 
given by
\begin{align*}
\zeta^\bullet_{\alpha,\gamma}=
\prod_{\substack{i,j\in Q_0\\i<j}}
\prod_{\substack{0<r\leqslant a_i\\b_j\geqslant s>a_j}}\frac{q^{-a_{ij}}x_{i,r}-x_{j,s}}{x_{i,r}-q^{-a_{ij}}x_{j,s}}
\cdot\prod_{i\in Q_0}\prod_{\substack{0<r\leqslant a_i\\b_i\geqslant s>a_i}}
\frac{q^{-2}x_{i,r}-x_{i,s}}{x_{i,r}-x_{i,s}}.
\end{align*}
We equip the $\K$-vector space
\begin{align}\label{SH}
\SH_\K=\bigoplus_\beta\K_{G_{\beta,c}}
\end{align}
with the shuffle product with kernel $\zeta^\bullet_{\alpha,\gamma}$. This product is given by
\begin{align}\label{shuffle1}
f\bullet g=\Sym_{x}(\zeta^\bullet_{\alpha,\gamma}\cdot( f\otimes g))
\quad,\quad f\in\K_{G_{\alpha,c}}, g\in\K_{G_{\gamma,c}}.
\end{align}
Let $\SH_\K^\bullet$ be the corresponding $\R$-algebra.
The following is well-known when $Q$ is of finite type, see \cite{E00}, \cite{Gr07}.
See Appendix \ref{app:A} for more details if $Q$ is any quiver of Kac-Moody type, 
for a more general weight function.

\smallskip

\begin{proposition}\label{prop:Shuffle}
If $Q$ is a quiver of Kac-Moody type, 
there is an $\R$-algebra homomorphism
\begin{align*}
\psi:\widetilde\bfU^+\to \SH_\K^\bullet
\quad,\quad
(\e_{i,n})^{(r)}\mapsto(x_{i,1}x_{i,2}\cdots x_{i,r})^n
\quad,\quad i\in Q_0, n\in\bbZ, r\in\bbZ_{>0}.
\end{align*}
\qed
\end{proposition}

The proof of Theorem \ref{thm:A} below implies that the map $\psi$ factorizes through the 
quotient $\bfU^+$ of $\widetilde\bfU^+$.
\smallskip

\subsubsection{K-theoretic Hall algebras}\label{sec:KHA}
Let $Q$ be any quiver. 
The Grothendieck groups $\G_0(T^*\calX)$ and
$\G_0(T^*\calX)_\Lambda$
are $\R$-modules such that 
$q[M]=[M\langle -1\rangle]$ for each $M$.
We choose the twist $\Omega$ as in \eqref{Omega}.
The monoidal structure on $D^\b(T^*\calX)$ and
$D^\b(T^*\calX)_\Lambda$ yields an $\R$-linear multiplication $\circ$ 
on the Grothendieck groups $\G_0(T^*\calX)$ and
$\G_0(T^*\calX)_\Lambda$. The first one is called the
\emph{K-theoretic Hall algebra} of the preprojective algebra $\Pi_Q$, and the second one the
\emph{nilpotent K-theoretic Hall algebra}. 
Let us denote them by $\bfK(\Pi_Q)$ and $\bfN\bfK(\Pi_Q)$.
For each dimension vector $\beta$ we set also
$\KHA(\Pi_Q)_\beta=\G_0(T^*\calX_\beta)$ and $\NKHA(\Pi_Q)_\beta=\G_0(T^*\calX_\beta)_\Lambda.$
The obvious functor $D^\b(T^*\calX)_\Lambda\to D^\b(T^*\calX)$
yields an $\R$-algebra homomorphism 
\begin{align}\label{nilnonil}\NKHA(\Pi_Q)\to\KHA(\Pi_Q).\end{align} 
This map is invertible if the quiver $Q$ is of finite type, because in this case all representations are nilpotent.
As $\R$-modules, we have
\begin{align}\label{isom11}
\KHA(\Pi_Q)=\G_0(\calM)
\quad,\quad
\NKHA(\Pi_Q)=\G_0(\Lambda).
\end{align}

\smallskip

Now, assume that $Q$ is of Kac-Moody type
and that $Q_1=\bar Q_1^<$ for some total ordering of the vertices of $Q$. 
Since the quiver $Q$ has no edge loops,
for each vertex $i$ and each positive integer $r$, the graded dg-algebra $\bfA_{r\alpha_i}$
is isomorphic to the exterior algebra $\SS(\frakg_{r\alpha_i}^*[1]\langle 2\rangle)$ with zero differential. 
For each dimension vector $\beta$, let $\P_\beta$ be the set of dominant weights of $G_\beta$ and
$V(\lambda)$ the irreducible representation of highest weight $\beta$.
Let $\beta=r\alpha_i$.
We view $V(\lambda)$ as a dg-module over the crossed product $\bfA_{r\alpha_i}\rtimes G_{r\alpha_i,c}$
such that $\bfA_{r\alpha_i}\rtimes\bbC^\times$ acts trivially.
Let $\calO(\lambda)_{r\alpha_i}$ be the corresponding coherent sheaf over the dg-stack $T^*\calX_{r\alpha_i}$.
Its class in the Grothendieck group is an element
\begin{align}\label{O}[\calO(\lambda)_{r\alpha_i}]\in\NKHA(\Pi_Q)_{r\alpha_i}.\end{align}
Let $\omega_1,\dots,\omega_r$ be the fundamental weights in $\P_{r\alpha_i}$.
Our first result is the following.

\smallskip

\begin{theorem}\label{thm:A}
Let $Q$ be a Kac-Moody quiver, with
a normal weight function on $\bar Q$.
\begin{itemize}[leftmargin=8mm]
\item[$\mathrm{(a)}$] 
There is a surjective $\bbN^{Q_0}$-graded $\R$-algebra homomorphism 
$\phi:\bfU^+\to\NKHA(\Pi_Q)$ such that
$\phi((\e_{i,n})^{(r)})=[\calO(n\omega_r)_{r\alpha_i}]$
for all $i\in Q_0,$ $ n\in\bbZ$ and $ r\in\bbZ_{>0}.$
\item[$\mathrm{(b)}$] 
If $\bfQ$ is of finite or affine type but not of type $A_1^{(1)}$, then the map $\phi$ is injective.
\end{itemize}
\end{theorem}

\smallskip

The proof uses some generalizations of the K-theoretic Hall algebras.
The first one is the deformed K-theoretic Hall algebra 
given by considering more general torus actions instead of the 
$\bbC^\times$-action defined in \S\ref{sec:TQ1}.
The other one is the deformed K-theoretic Hall algebra associated with quivers with potential.

\smallskip

The theorem does not imply that the quantum group $\widetilde\bfU^+$ is free over $\R$, since we used the
quotient $\bfU^+$.
This is known to be true if the quiver $Q$ is of finite type, by the PBW theorem for $\widetilde\bfU^+$.
In this case we have indeed $\widetilde\bfU^+=\bfU^+$.

\smallskip

\smallskip

\subsubsection{Deformed K-theoretic Hall algebras}\label{sec:deformed}
Let $Q$ be any quiver.
Consider a torus $\T$ with a weight function $Q_1\to\X^*(\T)$ such that $h\mapsto q_h$.
The torus $\T$ acts on the representation space $X$ of $Q$ as  in \eqref{action} with $u\in\T$, i.e., we have
\begin{align*}u\cdot x_h=q_h(u)x_h
\quad,\quad u\in\T, x\in X_\beta, h\in Q_1.\end{align*}
This yields a $G_{\beta,\T}$-action on $X_\beta$ for each dimension vector $\beta$.
Now, consider the double quiver $\bar Q$ and choose an homogeneous weight function $h\mapsto q_h$
on $\bar Q$, see \S\ref{sec:TQ1}.
We get a $G_{\beta,\T}$-action on the space $\bar X_\beta$.
The groups $G_\beta$ and $\T$ act on 
$\frakg_\beta^*$ by conjugation and by the character  $q^2$ respectively.
The moment map $\mu_\beta:\bar X_\beta\to\frakg_\beta^*$ is  $G_{\beta,\T}$-invariant. 
Replacing everywhere the $\bbC^\times$-action by the $\T$-action
we define as above the stacks 
$$\calX_\T\quad,\quad\bar\calX_\T\quad,\quad\calM_\T\quad,\quad\Lambda_\T\quad,\quad\text{etc}$$
Then, we define the $\R_\T$-algebras $\NKHA(\Pi_Q)_\T$ and $\KHA(\Pi_Q)_\T$
as in \S\ref{sec:KHA}. We call them the \emph{deformed nilpotent K-theoretic Hall algebra} 
and the \emph{deformed K-theoretic Hall algebra} of $\Pi_Q$.

\smallskip

\begin{remark}
If $\T=\{1\}$ we write 
$$\calX_1=\calX_\T
\quad,\quad
\Lambda_1=\Lambda_\T
\quad,\quad
\NKHA(\Pi_Q)_1=\NKHA(\Pi_Q)_\T
\quad,\quad 
etc.$$
We'll often assume that the torus  $\T$ is non trivial.
If  $\T=\bbC^\times\times\A$ with $\A$ a subtorus, as in \S\ref{sec:TQ1}, then 
the weight function is a map $h\mapsto q_h,t_h$ and the $\T$-action on $X$ writes
\begin{align}\label{actionT}(u,v)\cdot x_h=q_h(u)t_h(v)x_h
\quad,\quad u\in\bbC^\times, v\in \A, x\in X, h\in Q_1.\end{align}
\end{remark}

\smallskip

\subsubsection{Nilpotent K-theoretic Hall algebra of $\Pi_Q$}\label{sec:generators}
Let $Q$ be any quiver.
Consider a torus $\T$ and an homogeneous weight function $\bar Q_1\to\X^*(\T)$
such that $h\mapsto q_h$ yielding a $G_{\beta,\T}$-action 
on $\bar X_\beta$ as in \S\ref{sec:deformed}.
The  nilpotent version $\NKHA(\Pi_Q)_\T$ is introduced for the following reasons :
\begin{itemize}[leftmargin=8mm]
\item[$\mathrm{(a)}$] 
$\NKHA(\Pi_Q)_\T$ admits some generators of dimension in $\bigsqcup_i\bbN\alpha_i$ by Proposition \ref{prop:generators},
\item[$\mathrm{(b)}$] 
$\NKHA(\Pi_Q)_{\beta,\T}$ is a torsion free $\R_{G_{\beta,\T}}$-module for each dimension $\beta$ by Proposition \ref{prop:TFT}, if the torus $\T$ is large enough,
\item[$\mathrm{(c)}$] 
$\NKHA(\Pi_Q)_\T$ is a free $\R_\T$-module  by Lemma \ref{lem:red},
\item[$\mathrm{(d)}$] 
$\NKHA(\Pi_Q)_\T\otimes_{\R_\T}\K_\T=\KHA(\Pi_Q)_\T\otimes_{\R_\T}\K_\T$
by  Lemma \ref{lem:red} if $\T=\bbC^\times\times\A$.
\end{itemize}
The goal of this section is to prove Claim (a).
We choose any torus $\T$, possibly trivial.
We abbreviate
$$\NKHA(\Pi_Q)_{\bullet\alpha_i,\T}=\bigoplus_{r\geqslant 0}\NKHA(\Pi_Q)_{r\alpha_i,\T}
\quad,\quad
\NKHA(\Pi_Q)_{\delta,\T}=\bigsqcup_{i\in Q_0}\NKHA(\Pi_Q)_{\bullet\alpha_i,\T}.$$
The direct image yields a map
\begin{align*}\NKHA(\Pi_Q)_{\delta,\T}\to\NKHA(\Pi_Q)_\T.\end{align*}

\begin{proposition}\label{prop:generators}
The $\R_\T$-algebra $\NKHA(\Pi_Q)_\T$ is generated by the subset
$\NKHA(\Pi_Q)_{\delta,\T}.$
\end{proposition}

\begin{proof} 
For each vertex $i\in Q_0$ and each representation $x\in\bar X_\beta$, we set
$x_i^{out}=\sum_hx_h,$
where the sum runs over the set of all arrows $h\in\bar Q_1$ with source $i$ and target $\neq i$.
Let $\varepsilon_i(x)$ be the dimension of the maximal subspace of $\dim\Ker(x_i^{out})$ preserved by the map
$x_h$ for all loop $h:i\to i$ in $\bar Q_1$.
We abbreviate
$$\bar X^{\geqslant l\alpha_i}_\beta=\{\,x\in\bar X_\beta\,;\,\varepsilon_i(x)\geqslant l\,\}
\quad,\quad l\in\bbN.$$
Consider the substacks of $\bar\calX_\beta$ given by
$$\bar\calX_{\beta,\T}^{\geqslant l\alpha_i}=[\,\bar X_\beta^{\geqslant l\alpha_i}\,/\,G_{\beta,\T}\,]
\quad,\quad
\bar\calX_{\beta,\T}^{<l\alpha_i}=\bar\calX_{\beta,\T}\,\setminus\,\bar\calX_{\beta,\T}^{\geqslant l\alpha_i}
\quad,\quad
\bar\calX_{\beta,\T}^{l\alpha_i}=\bar\calX_{\beta,\T}^{\geqslant l\alpha_i}\,\setminus\,\bar\calX_{\beta,\T}^{> l\alpha_i}.$$
For  $Z=\Lambda$ or $T^*\calX_{\beta,\T}$  we define similarly 
$Z^{\geqslant l\alpha_i}$, 
$Z^{<l\alpha_i}$ and 
$Z^{l\alpha_i}$.
The stack $\Lambda^{\geqslant l\alpha_i}$ is closed in $\Lambda$.
The dg-stack $T^*\calX_{\beta,\T}^{< l\alpha_i}$ is almost smooth, being open in the almost smooth dg-stack
$T^*\calX_{\beta,\T}$.
For each flag $F=(F_0\subset F_1\subset\cdots\subset F_r)$ as in \eqref{flag}
and each element $x\in\bar X(F)$,
we have $x_i^{out}(F_1)=0$ for any vertex $i$.
Hence the stack
 $\Lambda_\beta$ is contained in the union of all closed substacks $\bar\calX_\beta^{\geqslant \alpha_i}$
as $i$ runs over the set $Q_0$.
Let $\G_0(T^*\calX_{\beta,\T})^{\geqslant l\alpha_i}_\Lambda$ be the image of the obvious map
$$\G_0(T^*\calX_{\beta,\T})_{\Lambda^{\geqslant l\alpha_i}}\to\G_0(T^*\calX_{\beta,\T})_\Lambda.$$
Thus the Mayer-Vietoris exact sequence implies that, see \S\ref{sec:MayerVietoris},
\begin{align}\label{MV}
\G_0(T^*\calX_{\beta,\T})_\Lambda=\sum_{i\in Q_0}\G_0(T^*\calX_{\beta,\T})_\Lambda^{\geqslant \alpha_i}.
\end{align}
Let $A$ be the subalgebra of $\NKHA(\Pi_Q)_\T$ generated by the subset
$\NKHA(\Pi_Q)_{\delta,\T}$. 
Set
$$A_\beta=A\cap\NKHA(\Pi_Q)_{\beta,\T}.$$
Let $b_i$ be the $i$th coordinate of the dimension vector $\beta$.
If $l>b_i$ then we have
$$\G_0(T^*\calX_{\beta,\T})^{\geqslant l\alpha_i}_\Lambda=\{0\}.$$
Assume that for some positive integer $l$ we have
\begin{align}\label{HYP}\G_0(T^*\calX_{\beta,\T})^{>l\alpha_i}_\Lambda\subset A_\beta.\end{align}
By descending induction and \eqref{MV} we'll prove that
\begin{align}\label{EQ1}\G_0(T^*\calX_{\beta,\T})^{\geqslant l\alpha_i}_\Lambda\subset A_\beta.\end{align}

\smallskip

To do so, set $\gamma=l\alpha_i$.
Given a flag of $Q_0$-graded vector spaces $W\subset V$ of dimension $(\gamma,\beta)$, we define
the groups $G$, $P$, $L$, the stacks $\calX_{G,\T}$, $\calX_{P,\T}$, $\calX_{L,\T}$ and the
morphisms $q$, $p$ as in \S\ref{sec:DHAPQ}.
We abbreviate 
$$
\bar\calX_{L,\T}^{0\alpha_i}=\bar\calX_{\alpha,\T}^{0\alpha_i}\times \bar\calX_{\gamma,\T}
\quad,\quad
\bar\calX_{P,\T}^{l\alpha_i}=(p_\flat)^{-1}(\bar\calX_{G,\T}^{l\alpha_i}).$$
If $x\in X_P$ then $W\subset\Ker(x_i^{out})$, hence 
$p_\flat(\bar\calX_{P,\T})\subset\bar\calX_{G,\T}^{\geqslant l\alpha_i}$ and
the stack $\bar\calX_{P,\T}^{l\alpha_i}$ is open in $\bar\calX_{P,\T}$.
The map $p_\flat$ restricts to an isomorphism 
$p'_\flat:\bar\calX_{P,\T}^{l\alpha_i}\to\bar\calX_{G,\T}^{l\alpha_i}$.
We get the following commutative diagram of stacks
\begin{equation*}
\xymatrix{
\bar\calX_{L,\T}&\,\bar\calX_{L,\T}^{0\alpha_i}\ar[l]_-{j_1}\\
\bar\calX_{P,\T}\ar[u]^-{q_\flat}\ar[d]\ar[d]_{p_\flat}
&\,\bar\calX_{P,\T}^{l\alpha_i}\ar[l]_-{j_2}\ar[u]_-{q'_\flat}\ar[d]^-{p'_\flat} \\
\bar\calX_{G,\T}&\,\bar\calX_{G,\T}^{\leqslant l\alpha_i}\ar[l]_-{j_3}.}
\end{equation*}
The $j$'s are the obvious open embeddings.
The maps in the diagram preserve the subsets of nilpotent representations.
Taking the derived functors and the Grothendieck groups,
we get the following diagram
\begin{equation*}
\begin{split}
\xymatrix{
\G_0(T^*\calX_{L,\T})_\Lambda\ar@{->>}[r]^-{L(j_1)^*}\ar[d]_-{Lq^*}&
\G_0(T^*\calX_{L,\T}^{0\alpha_i})_{\Lambda^{0\alpha_i}}\ar[d]^-{L(q')^*}_-\wr\\
\G_0(T^*\calX_{P,\T})_\Lambda\ar@{->>}[r]^-{L(j_2)^*}\ar[d]_-{Rp_*}&
\G_0(T^*\calX_{P,\T}^{l\alpha_i})_{\Lambda^{l\alpha_i}}\ar[d]^-{R(p')_*}_-\wr\\
\G_0(T^*\calX_{G,\T})_{\Lambda^{\geqslant l\alpha_i}}\ar@{->>}[r]^-{L(j_3)^*}&
\G_0(T^*\calX_{G,\T}^{\leqslant l\alpha_i})_{\Lambda^{l\alpha_i}}.}
\end{split}
\end{equation*}
The lower square is commutative by base change.
The upper square is obviously commutative.
The map $R(p')_*$ is invertible. 
The map $L(q')^*$ is invertible because $(q')_\flat$ is a vector bundle stack 
$\Lambda_{P,\T}^{l\alpha_i}\to\Lambda_{L,\T}^{0\alpha_i}$ by \cite[lem~5.7]{SV18}. 
We deduce that for each element $x$ in $\G_0(T^*\calX_{G,\T})_{\Lambda^{\geqslant l\alpha_i}}$ there is some 
$y$ in $\G_0(T^*\calX_{L,\T})_\Lambda$ such that 
$L(j_3)^*(x)=L(j_3)^*Rp_*Lq^*(y)$.
Hence, we have
$$L(j_3)^*(\G_0(T^*\calX_{\beta,\T})^{\geqslant l\alpha_i}_\Lambda)\subseteq 
L(j_3)^*(\G_0(T^*\calX_{\alpha,\T})_\Lambda\circ\G_0(T^*\calX_{\gamma,\T})_\Lambda).$$
Since $l>0$, by increasing induction on $\beta$ we can assume that
$\G_0(T^*\calX_{\alpha,\T})_\Lambda\subset A_\alpha,$
hence 
$$L(j_3)^*(\G_0(T^*\calX_{\beta,\T})^{\geqslant l\alpha_i}_\Lambda)\subseteq 
L(j_3)^*(A_\beta).$$
The localization exact sequence yields the following exact sequence
\begin{align*}
\xymatrix{
0\ar[r]&\G_0(T^*\calX_{\beta,\T})^{> l\alpha_i}_\Lambda\ar[r]&
\G_0(T^*\calX_{\beta,\T})^{\geqslant l\alpha_i}_\Lambda\ar[r]^-{L(j_3)^*}&
\G_0(T^*\calX_{\beta,\T}^{\leqslant l\alpha_i})_{\Lambda^{\leqslant l\alpha_i}}.
}
\end{align*}
Thus the inclusion \eqref{EQ1} follows from \eqref{HYP}.
\end{proof}

\smallskip

\subsubsection{Deformed K-theoretic Hall algebra of a quiver without potential}\label{sec:DHA}
Let $Q$ be any quiver.
Fix a torus $\T$ and a weight function $Q_1\to\X^*(\T)$ such that $h\mapsto q_h$.
We define
$$\KHA(Q)_\T=\G_0(\calX_\T)\quad,\quad
\KHA(Q)_{\beta,\T}=\G_0(\calX_{\beta,\T}).$$ 
Since $X_\beta$ is an affine space, we have obvious $\R_\T$-module isomorphisms
$$\G_0(\calX_{\beta,\T})=\R_{G_{\beta,\T}}\otimes[\calO_{\calX_{\beta,\T}}]=
\R_{G_{\beta,\T}}.$$
Let $\SH_\T=\bigoplus_\beta\R_{G_{\beta,\T}}$.
We deduce that
\begin{align}\label{DKHAWP}
\KHA(Q)_\T=\SH_T
\end{align}
Let the groups $G$, $P$, $L$, $U$ be as in \S\ref{sec:DHAPQ}. 
There is an obvious proper map
$p:\calX_{P,\T}\to \calX_{G,\T}$ and almost smooth map
$q:\calX_{P,\T}\to \calX_{L,\T}$.
They give rise to the triangulated functors
$$Rp_*:D^\b(\calX_{P,\T})\to D^\b(\calX_{G,\T})
\quad,\quad
Lq^*:D^\b(\calX_{L,\T})\to D^\b(\calX_{P,\T}).$$
We define the functor
\begin{align}\label{*}*=Rp_*Lq^*:D^\b(\calX_{L,\T})\to D^\b(\calX_{G,\T}).\end{align}
A direct computation yields the following.

\smallskip

\begin{lemma}\label{lemma:shuffle}
Under the isomorphism \eqref{DKHAWP}
the multiplication $*$ on $\KHA(Q)_\T$
is identified with the shuffle product on $\SH_\T$ with kernel 
\begin{align}\label{kernel0}
\zeta_{\alpha,\gamma}^*=\prod_{h\in Q_1}
\prod_{\substack{b_i\geqslant r>a_i\\0<s\leqslant a_j}}
(1-x_{i,r}/q_hx_{j,s})
\cdot\prod_{i\in Q_0}\prod_{\substack{0<s\leqslant a_i\\b_i\geqslant r>a_i}}(1-x_{i,r}/x_{i,s})^{-1}.
\end{align}
\qed
\end{lemma}

We'll use a twisted version $\circ$ of the functor $*$ defined as follows, see \S\ref{sec:twist},
\begin{align}\label{tp}
\circ=Rp_\circ Lq^*
\quad,\quad
Rp_\circ=Rp_*(\,\text{-}\otimes \Omega).\end{align}
The twist $\Omega$ is an object in $D^\b(\calX_{P,\T})$ which is analogous to \eqref{Omega}.
We will define it later.
The functors $*$ and $\circ$ yield $\R_\T$-algebra structures on $\KHA(Q)_\T$ called the
\emph{K-theoretic Hall algebra of the quiver $Q$ without potential} (and its twisted version).

\smallskip

\subsubsection{Deformed K-theoretic Hall algebra of a quiver with  potential}\label{sec:DHAW}
Let $Q$ be any quiver.
Let $W$ be a \emph{potential} on $Q$, i.e., an element of the quotient $\bbC Q\,/\,[\bbC Q\,,\,\bbC Q]$. We define
the K-theoretic Hall algebra $\KHA(Q,W)_\T$ of the pair $(Q,W)$ as follows, see \cite{P19}.
For each dimension vector $\beta$, the trace of $W$ yields a map 
$w_\beta:X_\beta\to\bbC.$ 
We'll assume that this function is $G_{\beta,\T}$-invariant.
We define
$$\calW_{\beta,\T}=[\,W_\beta\,/\,G_{\beta,\T}\,]
\quad,\quad
W_\beta=\{0\}\times^R_{\bbC} X_\beta$$
The dg-stack $\calW_{\beta,\T}$ is the quotient of the derived zero fiber $W_\beta$ of 
the map $w_\beta$ by the $G_{\beta,\T}$-action. 
Let $D^\sg(\calW_{\beta,\T})$ be the category of singularities of the dg-stack $\calW_{\beta,\T}$. 
Since $X_\beta$ is an affine space, if 
$w_\beta\neq 0$ then the map $w_\beta$ is flat. 
In this case $\calW_{\beta,\T}$ is just the quotient stack of the scheme theoretic zero 
locus of $w_\beta$, and the derived category of singularities considered here is the graded derived 
category of singularities considered by Orlov.
We define
$$\KHA(Q,W)_{\beta,\T}=\G_0^\sg(\calW_{\beta,\T}).$$ 
By \cite[\S 3.1]{P19}, for each dimension vectors
$\alpha,$ $\gamma$ with sum $\beta$ we define as in \eqref{*} a triangulated functor
$$*:D^\sg(\calW_{\alpha,\T})\times D^\sg(\calW_{\gamma,\T})\to D^\sg(\calW_{\beta,\T}).$$
We may also define a twisted version $\circ$ of the product $*$ as in \eqref{tp}.
Taking the sum over all dimension vectors we get a graded triangulated monoidal category structure $\circ$ on
$$D^\sg(\calW_\T)=\bigoplus_{\beta\in\bbN^{Q_0}}D^\sg(\calW_{\beta,\T}).$$
We call the graded triangulated monoidal category $(D^\sg(\calW_\T)\,,\,\circ)$ the
\emph{deformed derived Hall algebra} of the quiver with potential $(Q,W)$. 
This monoidal structure yields an $\R_\T$-linear multiplication $\circ$ on the Grothendieck group,
hence an $\R_\T$-algebra 
$$\KHA(Q,W)_\T=\bigoplus_\beta\KHA(Q,W)_{\beta,\T}$$ 
called the \emph{deformed K-theoretic Hall algebra} of the quiver with potential $(Q,W)$.
If $\T=\bbC^\times$ we abbreviate 
$$\KHA(Q)=\KHA(Q)_\T
\quad,\quad
\KHA(Q,W)=\KHA(Q,W)_\T.$$

We can now compare the deformed K-theoretic Hall algebra with zero potential and the
deformed K-theoretic Hall algebra without potential.
For each dimension vector $\beta$ the second projection $W_\beta\to X_\beta$ is an 
l.c.i.~closed immersion. 
It yields an l.c.i.~closed immersion of dg-stacks $i_W:\calW_\T\to\calX_\T$.
The pushforward by the map $i_W$ is well-defined, see, e.g., \cite{T12}. 

\smallskip

\begin{lemma}\label{lem:pad} 
\hfill
\begin{itemize}[leftmargin=8mm]
\item[$\mathrm{(a)}$] 
The direct image yields an $\R_\T$-algebra homomorphism
$(i_W)_*:\KHA(Q,W)_\T\to\KHA(Q)_\T.$
\item[$\mathrm{(b)}$] 
The direct image yields an $\R_\T$-algebra isomomorphism
$(i_0)_*:\KHA(Q,0)_\T\to\KHA(Q)_\T.$
\end{itemize}
\end{lemma}

\begin{proof}
A proof of (a) is sketched in \cite[prop~9.4]{P19}, see also \cite[prop~3.6]{P21}.
If $W=0$ then the derived stack $\calW_\T$ is the derived zero locus $R\Z(\calX_\T\,,\,\calO_{\calX_\T}\,,\,0)$.
In this case, the map $(i_0)_*$ is invertible, so we have an $\R_\T$-algebra isomomorphism
$\KHA(Q,0)_\T=\KHA(Q)_\T.$
\end{proof}

\smallskip

\subsubsection{Deformed K-theoretic Hall algebra of a triple quiver}\label{sec:triple}
Let $Q$ be any quiver and $\widetilde Q$ the triple quiver of $Q$.
Let $\T=\bbC^\times\times\A$ with $\A$ a subtorus.
Fix a weight function $\widetilde Q_1\to\X^*(\T)$ such that $q_h=q^{-2}$ and $t_h=1$ 
for each loop $h\in\omega(Q_0)$.
Fix a total ordering of the vertices.
Assume that $\bar Q_1^<\subset Q_1$.
We renormalize  the twist $\Omega$ in \eqref{Omega} to take account of the $\A$-action and the edge loops
in $Q_1$, see below.
We equip the $\R_\T$-module $\KHA(\widetilde Q)_\T$ 
with the multiplication $\circ$ in \eqref{tp}.
Our first goal is to compare the deformed K-theoretic Hall algebra of the triple quiver without potential
$\KHA(\widetilde Q)_\T$, the deformed nilpotent K-theoretic Hall algebra $\NKHA(\Pi_Q)_\T$, 
and a shuffle algebra that we introduce now.
We equip the $\K_\T$-vector space
\begin{align}\label{SHT}
\SH_{\T,\K}=\bigoplus_\beta\K_{G_{\beta,\T}}
\end{align}
with the shuffle product $\diamond$ with kernel $\zeta^\diamond_{\alpha,\gamma}$ given by
\begin{align}\label{kernel}
\zeta^\diamond_{\alpha,\gamma}=
\prod_{h\in Q_1}
\prod_{\substack{0<r\leqslant a_i\\b_j\geqslant s>a_j}}\frac{q_{h'}x_{i,r}-t_{h'}^{-1}x_{j,s}}{x_{i,r}-q_ht_hx_{j,s}}
\cdot\prod_{i\in Q_0}\prod_{\substack{0<r\leqslant a_i\\b_i\geqslant s>a_i}}
\frac{q^{-2}x_{i,r}-x_{i,s}}{x_{i,r}-x_{i,s}}.
\end{align}

\begin{lemma}\label{lem:KHAQT1}
There is an $\R_\T$-algebra embedding
$\nu:\KHA(\widetilde Q)_\T\to\SH_{\T,\K}^\diamond.$ 
\end{lemma}

\begin{proof}
By \eqref{DKHAWP} we have an obvious $\R_\T$-module isomorphism
\begin{align}\label{isom2}\KHA(\widetilde Q)_\T=\bigoplus_\beta\R_{G_{\beta,\T}}.\end{align}
By Lemma \ref{lemma:shuffle}, 
the multiplication $*$ on $\KHA(\widetilde Q)_\T$ is the shuffle product with the kernel 
\begin{align}\label{kernel0}
\zeta_{\alpha,\gamma}^*=\prod_{h\in \widetilde Q_1}
\prod_{\substack{b_i\geqslant r>a_i\\0<s\leqslant a_j}}
(1-x_{i,r}/q_ht_hx_{j,s})
\cdot\prod_{i\in Q_0}\prod_{\substack{0<s\leqslant a_i\\b_i\geqslant r>a_i}}(1-x_{i,r}/x_{i,s})^{-1}.
\end{align}
The renormalized twist is
$$\Omega=\prod_{h\in Q_1}
\prod_{\substack{b_j\geqslant s>a_j\\0<r\leqslant a_i}}
(-q_{h'}x_{i,r}/q_ht_hx_{j,s}).$$
The twisted multiplication $\circ$ on $\KHA(\widetilde Q)_\T$ is a shuffle product with some kernel 
$\zeta^\circ_{\alpha,\gamma}$.
Conjugating the isomorphism \eqref{isom2}
by the rational function
$$\prod_{h\in Q_1}\prod_{\substack{0<r\leqslant b_i\\0<s\leqslant b_j}}(1-x_{i,r}/q_ht_hx_{j,s})^{-1}$$
yields an $\R_\T$-algebra embedding
\begin{align}\label{nu}
\nu:\KHA(\widetilde Q)_\T\to\SH_{\T,\K}^\diamond,\end{align} 
where the multiplication on the right hand side is the shuffle product with kernel
$$\zeta^\diamond_{\alpha,\gamma}=\zeta^\circ_{\alpha,\gamma}
\prod_{h\in Q_1}
\prod_{\substack{b_j\geqslant s>a_j\\0<r\leqslant a_i}}(1-x_{ir}/q_ht_hx_{js})^{-1}
\prod_{\substack{b_i\geqslant r>a_i\\0<s\leqslant a_j}}(1-x_{ir}/q_ht_hx_{js})^{-1}
$$
A direct computation shows that $\zeta^\diamond_{\alpha,\gamma}$ is given by the formula \eqref{kernel},
proving the lemma.
\end{proof}

\smallskip

We equip the quiver $\widetilde Q$ with the potential 
$$W=\sum_{h\in Q_1}\sum_{i\in Q_0}(x_hx_{h'}x_{\omega_i}-x_{h'}x_hx_{\omega_i}).$$
The trace of $W$ is the map
\begin{align}\label{w}
w_\beta:\widetilde X_\beta\to\bbC
\quad,\quad
x\mapsto\sum_{h\in Q_1}\sum_{i\in Q_0}\Tr([x_h,x_{h'}]x_{\omega_i})=
-\sum_{h\in Q_1}\sum_{i\in Q_0}\Tr([x_h,x_{\omega_i}]x_{h'}).\end{align}
The groups $G_\beta$ and $\T$ act on $\frakg_\beta$ by conjugation and through the character $q^{-2}$.
Set $\widetilde\calX_{\beta,\T}=[\widetilde X_\beta\,/\,G_{\beta,\T}]$.
The map $w_\beta$ factors to a map $\widetilde \calX_{\beta,\T}\to\bbC$.
 We also denote it by $w_\beta$.
 The derived zero locus of this map is  $\calW_{\beta,\T}$.
Via the obvious isomorphisms $\widetilde X_\beta=\bar X_\beta\times\frakg_\beta$ and 
$\frakg_\beta=\frakg_\beta^*$, the map $w_\beta$ is identified with
the section $\mu_\beta$ of the vector bundle $E_\beta\to\bar\calX_\beta$.
The dimensional reduction \eqref{dimred0} 
yields an equivalence of graded triangulated categories
\begin{align}\label{K1}D^\b(T^*\calX_{\beta,\T})=D^{\gr,\sg}(\calW_{\beta,\T}).\end{align}
Taking the Grothendieck groups as in \S\ref{sec:K-theory}
yields an $\R_{G_{\beta,\T}}$-module isomorphism
\begin{align}\label{K2}\KHA(\Pi_Q)_{\beta,\T}=\KHA(\widetilde Q,W)_{\beta,\T}.\end{align}
Taking the sum over all dimension vectors we get an an $\R_\T$-module isomorphism
\begin{align}\label{K3}\KHA(\Pi_Q)_\T=\KHA(\widetilde Q,W)_\T.\end{align} 

\smallskip

\begin{remark} By  \cite[(16)]{P21} the map \eqref{K2} is an $\R_\T$-agebra homomorphism.

\end{remark}

\smallskip

\subsection{Proof of Theorem \ref{thm:A}}\label{sec:ProofA}

We need more material to prove the theorem.

\subsubsection{Torsion freeness}\label{sec:TF}
Let $Q$ be any quiver.
Let $Q^+$ be the subquiver of $\widetilde Q$
with vertex set $Q_0$ and arrow set $Q_1^+=Q_1\sqcup \omega(Q_0)$.
Let $X_\beta^+$ be the space of all $\beta$-dimensional representations of $Q^+$.
We have $X_\beta^+=X_\beta\times\frakg_\beta$.
Consider a torus  $\T=\bbC^\times\times\A$ as in \S\ref{sec:TQ1}, an homogeneous 
$G_{\beta,\T}$-action on $\bar X_\beta$ as in \eqref{actionT}, and its restriction to $X_\beta$.
Let the groups $G_\beta$ and $\T$ act on 
$\frakg_\beta$ by conjugation and by multiplication by the character  $q^{-2}$ respectively.
The group $G_{\beta,\T}$ acts diagonally on $\widetilde X_\beta$ and $X_\beta^+$.
Let $Z$ be an arbitrary 
$G_{\beta,\T}$-invariant locally closed subset of $\frakg_\beta$.
We define 
$$I_Z=I_\beta\cap(X_\beta\times Z)
\quad,\quad
I_\beta=\{\,(x,\omega)\in X_\beta^+\,;\,[x,\omega]=0\,\}$$
with $[x,\omega]=0$ meaning $x_h\omega_i=\omega_jx_h$ for each arrow $h\in Q_1$.
Let $\calI_{\beta,\T}$ and $\calI_{Z,\T}$ denote the quotient stacks 
$$\calI_{\beta,\T}=[I_\beta\,/\,G_{\beta,\T}]
\quad,\quad
\calI_{Z,\T}=[I_Z\,/\,G_{\beta,\T}]
\quad,\quad
\calX^+_{\beta,\T}=[X_\beta^+\,/\,G_{\beta,\T}].$$
Let $\frakN_\beta$ be the nilpotent cone in $\frakg_\beta$.

\smallskip

\begin{lemma}\label{lem:red} Let $\T=\bbC^\times\times\A$ as above.
\hfill
\begin{itemize}[leftmargin=8mm]
\item[$\mathrm{(a)}$] 
$\G_0(\calI_{\beta,\T})=\G_0(\calM_{\beta,\T})$ as $\R_{G_{\beta,\T}}$-modules.
\item[$\mathrm{(b)}$] 
$\G_0(\Lambda_{\beta,\T})$ is a free $\R_\T$-module.
\item[$\mathrm{(c)}$] 
The direct image by the obvious closed embedding $\calI_{\frakN_\beta,\T}\subset\calI_{\beta,\T}$ 
yields an isomorphism
$$\G_0(\calI_{\frakN_\beta,\T})\otimes_{\R_\T}\K_\T\to\G_0(\calI_{\beta,\T})\otimes_{\R_\T}\K_\T.$$
\item[$\mathrm{(d)}$] 
If condition \eqref{large} below holds then the direct image by the obvious closed embedding 
$\Lambda_{\beta,\T}\subset\calM_{\beta,\T}$
yields an isomorphism
$$\G_0(\Lambda_{\beta,\T})\otimes_{\R_\T}\K_\T\to
\G_0(\calM_{\beta,\T})\otimes_{\R_\T}\K_\T.$$
\end{itemize}
\end{lemma}

\begin{proof}
The projection $\widetilde X_\beta\to X^+_\beta$ forgetting the arrows of the quiver $Q'_1$
yields a vector bundle
$V_\beta$ over $\calX^+_{\beta,\T}$ with fibers isomorphic to $X'_\beta$.
The function $w_\beta$ in  \eqref{w} is fiberwise linear, hence it can be viewed as a section of the dual bundle 
$V_\beta^*$. 
By \eqref{dimred0}, there is an equivalence of graded triangulated categories
\begin{align*}D^\b(R\Z(\calX^+_{\beta,\T}, V^*_\beta,w_\beta))=D^{\gr,\sg}(\calW_{\beta,\T}).\end{align*}
Composing it with the equivalence \eqref{K1}, we deduce that
$$D^\b(R\Z(\calX^+_{\beta,\T}, V^*_\beta,w_\beta))=D^\b(T^*\calX_{\beta,\T}).$$
The truncations of these dg-stacks are
$$\pi_0(R\Z(\calX^+_{\beta,\T}, V_\beta^*,w_\beta)))=\calI_{\beta,\T}
\quad,\quad
\pi_0(T^*\calX_{\beta,\T})=\calM_{\beta,\T}.$$
Taking the Grothendieck groups as in \eqref{K2},
using the identity \eqref{truncation}, yields the isomorphism in (a).

\smallskip

Let us concentrate on (b).
We'll use Nakajima's quiver varieties. 
To do that, let the group $G_\beta\times G_{\beta,\T}$ acts on $\bar X_\beta\times(\frakg_\beta)^2$ as follows
$$(f,g,u,v)\cdot(x,i,j)=((g,u,v)\cdot x\,,\,ugif^{-1}\,,\,ufjg^{-1})
\quad,\quad
f,g\in G_\beta, (u,v)\in\T=\bbC^\times\times\A.$$
We'll use the following subgroups of $G_\beta\times G_{\beta,\T}$
\begin{align*}
G_{\beta,\T}=\{f=1\}\quad,\quad
G_\beta=\{f=u=v=1\}\quad,\quad F_{\beta,\T}=\{g=1\}\quad,\quad F_\beta=\{g=u=v=1\}.
\end{align*}
The Nakajima quiver variety $\frakM_\beta$
associated with the quiver $\bfQ$ and the pair of dimension vectors $(\beta,\beta)$
is defined as
the geometric quotient by $G_\beta$ of the variety of all triples
$(x,i,j)$ such that $\mu_\beta(x)+ij=0$ which are \emph{stable}, i.e., 
the kernel $\Ker j$ does not contain any non-zero $Q_0$-graded 
subspace stable by the action of the representation $x$.
This is a smooth quasi-projective variety by \cite[lem.~3.10]{N98}.
Note that loc. cit. considers only Kac-Moody quivers, but the proof there holds also for any $Q$.
Let $L_\beta$ be the variety of all triples $(x,i,j)$ such that $x\in N_\beta$ and $i=0$.
The variety $\frakM_\beta$ contains a  projective
subvariety $\frakL_\beta$ 
which is the geometric quotient
of the set of all stable triples $(x,i,j)$ in $L_\beta$.
See, e.g., \cite[\S 4.1.1]{SV18}.
We have open embeddings
\begin{align*}
&\frakL_\beta'=\{(x,i,j)\in L_\beta\ \,;\,j\in G_\beta\}\,/\,G_\beta
\subset\,\frakL_\beta\,\subset\,[L_\beta\,/\,G_\beta].
\end{align*}
Note that
$$\Lambda_{\beta,\T}=[\{(x,i,j)\in L_\beta\,;\,j=0\}\,/\,G_{\beta,\T}].$$
The group $F_\beta$ acts trivially on $\Lambda_{\beta,\T}$.
We consider the following diagram
\begin{align*}
\xymatrix{
\G_0(\Lambda_{\beta,\T})\ar[r]^-a&
\G_0(\Lambda_{\beta,\T}\times BF_\beta)\ar[r]^-b&
\G_0([\,L_\beta\,/\,G_\beta\times F_{\beta,\T}\,])\ar[d]^-c\\
&
\G_0([\,\frakL_\beta'\,/\,F_{\beta,\T}\,])\ar[lu]_-e&
\G_0([\,\frakL_\beta\,/\,F_{\beta,\T}\,])\ar[l]_-d
}
\end{align*}
The map $a$ is the inclusion 
$$a=id\otimes 1 : \G_0(\Lambda_{\beta,\T})\to\G_0(\Lambda_{\beta,\T}\times BF_\beta)=
\G_0(\Lambda_{\beta,\T})\otimes \R_{F_\beta}.$$
The map $b$ is the pullback by the obvious vector bundle
$$[\,L_\beta\,/\,G_\beta\times F_{\beta,\T}\,]\to \Lambda_{\beta,\T}\times BF_\beta.$$
It is invertible.
The maps $c$, $d$ are the open restrictions.
The assignment $(x,i,j)\mapsto (x,0,0)$ yields a stack isomorphism
$[\frakL_\beta'\,/\,F_{\beta,\T}]\to\Lambda_{\beta,\T}$, hence an isomorphism $e$ of the Grothendieck groups.
All maps are $\R_\T$-linear.
The composed map $edcba$ is the identity of $\G_0(\Lambda_{\beta,\T})$.
So, the identity factors through the $\R_\T$-module
$\G_0([\frakL_\beta\,/\,F_{\beta,\T}])$. The latter is a free $\R_\T$-module by \cite[thm~7.3.5]{N00},
hence $\G_0(\Lambda_{\beta,\T})$ is also free over $\R_\T$. 
Note that loc. cit. considers only Kac-Moody quivers, but the proof there extends to any 
quiver $Q$ as in \cite[prop.~4.3]{SV18}.

\smallskip

To prove (c) it is enough to check that
$$\G_1(\calI_{\beta,\T}\,\setminus\,\calI_{\frakN_\beta,\T})\otimes_{\R_\T}\K_\T=
\G_0(\calI_{\beta,\T}\,\setminus\,\calI_{\frakN_\beta,\T})\otimes_{\R_\T}\K_\T=0.$$
The torus $\T$ acts on $\frakg_\beta$ via the character $q^{-2}$.
Thus, for each $\bbC$-point $(x,\omega)$ of $I_\beta\,\setminus\, I_{\frakN_\beta}$, 
there is a $G_\beta$-invariant regular function 
on $X_\beta^+$ which does not vanish at $(x,\omega)$ 
and has nonzero weight for the action of the subgroup $\bbC^\times\subset \T$.
Hence the stabilizer of $(x,\omega)$ in $G_{\beta,\T}$ 
is contained in the subgroup $G_\beta\times\{1\}\times\A$. We deduce that the
Grothendieck group
$\G_i([G_{\beta,\T}\cdot(x,\omega)\,/\,G_{\beta,\T}])$ is killed by  the multiplication by the element $q-1$ 
in $\R_{G_{\beta,\T}}$ for $i=0,1$.
The localization exact sequence and the Noetherian induction imply that
the $\R_{G_{\beta,\T}}$-module $\G_i(\calI_{\beta,\T}\,\setminus\,\calI_{\frakN_\beta,\T})$
is killed by some positive power of $q-1$.

\smallskip

To prove (d) it is enough to check that
$$\G_1(\calM_{\beta,\T}\,\setminus\,\Lambda_{\beta,\T})\otimes_{\R_\T}\K_\T=
\G_0(\calM_{\beta,\T}\,\setminus\,\Lambda_{\beta,\T})\otimes_{\R_\T}\K_\T=0.$$
Recall the definition of the nilpotent variety $N_\beta$ in \S\ref{sec:nilp}.
By the Hilbert criterion and Condition \eqref{large}
for each $\bbC$-point $x$ of $\calM_{\beta,\T}\,\setminus\,\Lambda_{\beta,\T}$, 
there is a $G_\beta$-invariant regular function 
on $M_\beta$ which does not vanish at $x$ 
and has nonzero weight for the action of the subgroup $\theta$ of $\T$.
The rest of the proof is as for part (c), using the localization exact sequence and Noetherian induction.
\end{proof}

\smallskip

Fix a cocharacter $\theta\in\X_*(\T)$.
Composing $\theta$ with the inclusion $\T\subset G_{\beta,\T}$,
we may view it as a cocharacter
of $G_{\beta,\T}$.
Let $S_\theta\subset\R_{G_{\beta,\T}}$ be the complement of the kernel of the comorphism
$\theta^*:\R_{G_{\beta,\T}}\to\R$.
We consider the following condition on the action of  $\theta$ on $\bar X$
\begin{align}\label{large}
\theta(v)\cdot x_h=v^{n_h}\,x_h\quad,\quad
\theta(v)\cdot x_{h'}=v^{n_{h'}}\,x_{h'}\quad,\quad
n_{h'}=-n_h< 0\quad,\quad
h\in Q_1, v\in\bbC^\times.
\end{align}
Under this condition the cocharacter $\theta$ acts trivially on the factor $\frakg_\beta$ of
$X_\beta^+=X_\beta\times\frakg_\beta$.
Note that \eqref{homogeneous} implies that for \eqref{large} to hold the torus $\A$ must be non trivial, i.e.,
we must have $\dim(\T)\geqslant 2$.

\smallskip

\begin{lemma}\label{lem:2.2}  
Assume that \eqref{large} holds. 
\hfill
\begin{itemize}[leftmargin=8mm]
\item[$\mathrm{(a)}$]
$\G_0(\calI_{O,\T})$ is torsion-free over $S_\theta$ for each
each $G_\beta$-orbit $O\subset \frakN_\beta$.
\item[$\mathrm{(b)}$] 
The partition $\calI_{\frakN_\beta,\T}=\bigsqcup_O\calI_{O,\T}$ 
yields a filtration on the $\R_{G_{\beta,\T}}$-module
$\G_0(\calI_{\frakN_\beta,\T})$.
The associated graded is isomorphic to $\bigoplus_O\G_0(\calI_{O,\T})$.
\item[$\mathrm{(c)}$] 
$\G_0(\calI_{\frakN_\beta,\T})$ is torsion-free over $S_\theta$.
\end{itemize}
\end{lemma}

\smallskip

\begin{proof}
For any $\bbC$-point $(x,\omega)$ of $I_\beta$ and any integer $k\geqslant 0$,  
the representation $x$ preserves the $Q_0$-graded vector space 
$V_k=\Im(\omega^k)$, yielding a flag of representations of the quiver $Q$
$$0=V_s\subset\cdots\subset V_1\subset V_0=V.$$
Define $V'_k=V_k\,/\,V_{k+1}$.
The map $\omega$ yields a chain of surjective morphisms of representations
$$V'_0\twoheadrightarrow V'_1\twoheadrightarrow\cdots\twoheadrightarrow V'_s.$$
For each integer $k=0,\dots s,$  let $x^\omega_k$ be the representation induced by $x$ on 
the  $Q_0$-graded vector space $V^\omega_k=\Ker(V'_k\to V'_{k+1}).$
Fix a $G_\beta$-orbit $O\subset \frakN_\beta$ and a $\bbC$-point $\omega\in O$. 
Let $G_{\omega,\T}$ be the stabilizer of $\omega$ in $G_{\beta,\T}$.
We have 
$$\calI_{O,\T}=[\,I_\beta\cap(X_\beta\times\{\omega\})\,/\,G_{\omega,\T}].$$
The dimension vector of the $Q_0$-graded vector space $V^\omega_k$ 
does not depend on the choice of $\omega$ in $O$.
Let us denote it by $\beta(O,k)$.
Set 
$$X_O=\prod_{k>0}X_{\beta(O,k)}
\quad,\quad
G_{O,\T}=(\prod_{k>0}G_{\beta(O,k)})\times \T
\quad,\quad
\calX_{O,\T}=[X_O\,/\,G_{O,\T}]
.$$
The group $G_{O,\T}$ is the reductive part of $G_{\omega,\T}$.
In particular, there is an obvious embedding $G_{O,\T}\subset G_{\omega,\T}$.
The assignment
$$\calI_{O,\T}\to\calX_{O,\T}\quad,\quad (x,\omega)\mapsto(x^\omega_k\,;\,k=0,\dots s)$$ 
is a composition of vector bundle stacks, 
because the module category of the quiver $Q$ is quasi-hereditary and
the representation $x$ is reconstructed inductively from the $x^\omega_k$'s
via the short exact sequences 
$$0\to V_{k+1}\to V_k\to V'_k\to 0
\quad,\quad
0\to V^\omega_k\to V'_k\to V'_{k+1}\to 0.
$$
Hence,
the  pullback gives group isomorphisms
\begin{align}\label{isom1}
\G_i(\calI_{O,\T})=\G_i(\calX_{O,\T})
\quad,\quad
\G_i(\calI_{O,\T})^\top=\G_i(\calX_{O,\T})^\top
\quad,\quad
\forall i=0,1.
\end{align}
In particular, since $X_O$ is a vector space we have an isomorphism
$$\G_0(\calI_{O,\T})=\G_0(\calX_{O,\T})=\R_{G_{O,\T}}.$$
The $\R_{G_{\beta,\T}}$-action on $\G_0(\calI_{O,\T})$ factors through
the restriction $\R_{G_{\beta,\T}}\to\R_{G_{\omega,\T}}$.
Since $G_{O,\T}$ is the reductive part of $G_{\omega,\T}$, there is a canonical isomorphism
$\R_{G_{\omega,\T}}=\R_{G_{O,\T}}$.
As an $\R_{G_{\beta,\T}}$-module, the Grothendieck group $\G_0(\calI_{O,\T})$ is the pull-back
of the regular $\R_{G_{O,\T}}$-module $\R_{G_{O,\T}}$ by the restriction 
$\res_O:\R_{G_{\beta,\T}}\to\R_{G_{\omega,\T}}=\R_{G_{O,\T}}$.
The condition \eqref{large} implies that $\theta$ acts trivially on $\omega$, i.e., 
we have $\theta(\bbC^\times)\subset G_{\omega,\T}$.
Therefore, the comorphism
$\theta^*:\R_{G_{\beta,\T}}\to\R$ factors through
the map $\res_O$, i.e., we have
$$\Ker(\res_O)\subset\Ker(\theta^*).$$
We deduce that the map $\res_O$ is injective on the multiplicative set $S_\theta$.
The claim (a) follows.
Part (b) is a consequence of \eqref{isom1} and the following general fact which is recalled in Appendix \ref{app:C}.

\begin{lemma}\label{lem:FILT}
Let $G$ be a complex linear group and $X$ a complex $G$-variety. 
Let $X=\bigsqcup_{O\in\Pi} X_O$ be a finite partition by $G$-equivariant locally closed subsets indexed
by a poset such that $\bigsqcup_{O'\leqslant O} X_{O'}$ is closed in $X$ for each $O$.
Set $\calX_O=[X_O\,/\,G]$ and $\calX=[X\,/\,G]$.
Assume that the cycle map $\G_0(\calX_O)\to\G_0(\calX_O)^\top$ is an isomorphism and
$\G_1(\calX_O)^\top=\{0\}$. 
Then, the cycle map $\G_0(\calX)\to\G_0(\calX)^\top$ is an isomorphism and
the partition of $X$ yields a filtration on the $\R_G$-module 
$\G_0(\calX)$ whose associated graded is isomorphic to
the direct sum $\bigoplus_O\G_0(\calX_O)$.
\qed
\end{lemma}

Part (c) of the lemma follows from (a) and (b).
\end{proof}

\smallskip

\begin{proposition} \label{prop:TFT} 
The $\R_{G_{\beta,\T}}$-module $\G_0(\Lambda_{\beta,\T})$ is torsion-free if \eqref{large} holds.
\end{proposition}

\begin{proof}
The Segal concentration theorem implies that whenever a connected linear complex algebraic group $G$ 
acts on a complex scheme $X$ of finite type, for any closed subtorus embedding $\theta:H\to G$ with comorphism
$\theta^*:\R_G\to\R_H$, the direct image by the inclusion of the $H$-fixed points locus $X^H\subset X$ yields group
isomorphisms 
$$\G_i([GX^H\,/\,G])[S_\theta^{-1}]\to\G_i([X\,/\,G])[S_\theta^{-1}]
\quad,\quad i\geqslant 0,$$
where
$S_\theta$ is the multiplicative subset of $\R_G$ given by $S_\theta=\R_G\setminus \Ker(\theta^*)$,
see, e.g., \cite[\S 4]{T88}.
By \eqref{large}, the fixed point locus of the $\theta$-action on $N_\beta$ is  $\{0\}$. 
Thus, the direct image by the inclusion
$BG_{\beta,\T}\to\Lambda_{\beta,\T}$ is an isomorphism
$$\R_{G_{\beta,\T}}[S^{-1}_\theta]=\G_0(\Lambda_{\beta,\T})[S^{-1}_\theta].$$
On the other hand, Lemma \ref{lem:red} implies that 
$$\G_0(\Lambda_{\beta,\T})\subset\G_0(\calI_{\frakN_\beta,\T})\otimes_{\R_\T}\K_\T.$$
Since $\G_0(\calI_{\frakN_\beta,\T})$ is torsion-free over $S_\theta$
by Lemma \ref{lem:2.2}, we deduce that
$\G_0(\Lambda_{\beta,\T})$ is also torsion-free over $S_\theta$.
We deduce that there is the following embedding of $\R_{G_{\beta,\T}}$-modules
$$\G_0(\Lambda_{\beta,\T})\subset\G_0(\Lambda_{\beta,\T})[S^{-1}_\theta]=
\R_{G_{\beta,\T}}[S^{-1}_\theta]\subset\K_{G_{\beta,\T}}.$$
Thus $\G_0(\Lambda_{\beta,\T})$ is torsion-free over $\R_{G_{\beta,\T}}$.
\end{proof}

\smallskip

For a future use, we consider a particular case where Condition \eqref{large} holds.
Let $\T=\bbC^\times\times\A$, with $\bbC^\times\subset\A$ and a weight function $h\mapsto q_h,t_h$.
Assume that the restriction of $t_h$ to $\bbC^\times$ is given by
$t_h=t^{n_h}$ for all $h\in \bar Q_1$,
where $t$ is the weight one character of $\bbC^\times$ and the integer $n_h$ is as in \eqref{large}.
Then, we may choose $\theta$ to be the cocharacter $v\mapsto (1,v).$ 
We may also make the following additional assumption which permits to use \eqref{large}
with $A=\{1\}.$
\begin{align}\label{COND}
\begin{split}
&\emph{Assume that there is a central cocharacter $z$ of $G_\beta$ such that $z(u)x_h=t_h(u)x_h$} \\
&\emph{for all $h\in \bar Q_1$, where $t_h$ is as above.}
\end{split}
\end{align}
Note that if $\bfQ$ is of Kac-Moody type then \eqref{COND} holds up to changing the orientation.
Under this condition, we have the following analogue of Proposition \ref{prop:TFT}.

\smallskip

\begin{corollary}\label{cor:TF}
The $\R_{G_{\beta,c}}$-module $\G_0(\Lambda_\beta)$ is torsion free
if Condition \eqref{COND} holds.
\end{corollary}

\begin{proof}
Let the cocharacter $z$ be as in \eqref{COND}.
Consider the group automorphism
\begin{align*}
&a:G_{\beta,\T}\to G_{\beta,\T}\quad,\quad (g,u,v)\mapsto (gz(v)^{-1},u,v).
\end{align*}
The element $a(g,u,v)$ acts on $\bar X_\beta$ as $(g,u,1)$ by \eqref{actionT}, \eqref{COND}.
The pullback by $a$ identifies the $\R_{G_{\beta,\T}}$-modules $\G_0(\Lambda_{\beta,\T})$ and 
$\G_0(\Lambda_\beta)[t,t^{-1}]$.
In particular, there are compatible inclusions
$$\R_{G_{\beta,c}}\subset\R_{G_{\beta,\T}}
\quad,\quad
\G_0(\Lambda_\beta)\subset\G_0(\Lambda_{\beta,\T}).$$
Hence the torsion freeness of $\G_0(\Lambda_\beta)$ over $\R_{G_{\beta,c}}$ follows from 
the torsion freeness of $\G_0(\Lambda_{\beta,\T})$ over $\R_{G_{\beta,\T}}$ which is proved in
Proposition \ref{prop:TFT}.\end{proof}

\smallskip

\subsubsection{Proof of Theorem $\ref{thm:A}$}\label{sec:thmA}

Let $Q$ be any Kac-Moody quiver with $Q_1=\bar Q_1^<$.
Recall that, in this case, condition \eqref{COND} holds.
Fix $\T=\bbC^\times$ and fix a normal weight function on $\bar Q$, i.e., 
condition \eqref{normal} holds.
In this section we'll prove the theorem under Condition \eqref{COND}.
The general case is proved in \S\ref{sec:GC}.

\medskip

\emph{Step $1$ :} We first prove the following.

\begin{lemma}\label{lem:phi}
There is an $\R$-algebra homomorphism
\begin{align}\label{phi}
\phi:\bfU^+\to\NKHA(\Pi_Q)
\quad,\quad
(\e_{i,n})^{(r)}\mapsto[\calO(n\omega_r)_{r\alpha_i}]
\quad,\quad i\in Q_0, n\in\bbZ, r\in\bbZ_{>0}.
\end{align}
\end{lemma}

\begin{proof}
To prove the lemma, it is enough to check that 
\begin{itemize}[leftmargin=8mm]
\item[$\mathrm{(a)}$]
$[\calO(n\omega_1)_{\alpha_i}]$ satisfies the Drinfeld and Serre relations in $\NKHA(\Pi_Q)$,
\item[$\mathrm{(b)}$]
$[\calO(n\omega_r)_{r\alpha_i}]=[\calO(n\omega_1)_{\alpha_i}]^{(r)}$ in $\NKHA(\Pi_Q)$,
\item[$\mathrm{(c)}$]
the map $\widetilde\bfU^+\to\NKHA(\Pi_Q)$ given by (a) and (b) factorizes through $\bfU^+$.
\end{itemize}
Part (c) follows from Lemma \ref{lem:red}(b).
Part (b) is a direct computation which is done in Lemma \ref{lem:A1} below.
Now we concentrate on Part (a).
A way to check it is suggested in \cite{Gr94}, without proof.
We'll give a different proof.
To do so, we consider the $\K$-algebra $\SH_\K^\diamond$.
It is the $\K$-vector space $\SH_\K$
equipped with the shuffle product $\diamond$ associated with the kernel \eqref{kernel}, with $\A=\{1\}$.
We'll prove the following.

\smallskip

\begin{lemma}\label{lem:induction}
There is an $\R$-algebra embedding 
$$\rho:\NKHA(\Pi_Q)\to \SH_\K^\diamond
\quad,\quad
[\calO(\lambda)_{r\alpha_i}]\mapsto[V(\lambda)]
\quad,\quad
\lambda\in \P_{r\alpha_i}, r\in\bbZ_{>0}.$$
\end{lemma}

\begin{proof}
We consider the following chain of $\R_{G_{\beta,c}}$-linear maps
\begin{align*}
\xymatrix{
\G_0(\bfA_\beta\rtimes G_{\beta,c})\ar[r]\ar@{.>}[rd]_-{\pi_\beta}&
\G_0(H^0(\bfA_\beta)\rtimes G_{\beta,c})\ar[r]&
\G_0(\SS(\bar X_\beta^*\langle 1\rangle)\rtimes G_{\beta,c})\ar[d]\\
&\K_{G_{\beta,c}}&
\ar[l]\G_0(\SS(\bar X_\beta^*\langle 1\rangle)\rtimes G_{\beta,c})\otimes_{\R_{G_{\beta,c}}}\K_{G_{\beta,c}}.
}
\end{align*}
The first map is taking the cohomology, the second one the pull-back by the surjective ring homomorphism
$\SS(\bar X_\beta^*\langle 1\rangle)\to H^0(\bfA_\beta)$, the third one is base change, and the last one
the inverse of the pull-back by the counit $\SS(\bar X_\beta^*\langle 1\rangle)\to\bbC$.
By the Segal concentration theorem, the following map is invertible
$$\pi_\beta\otimes_{\R_{G_{\beta,c}}}\K_{G_{\beta,c}}:
\G_0(\bfA_\beta\rtimes G_{\beta,c})\otimes_{\R_{G_{\beta,c}}}\K_{G_{\beta,c}}\to\K_{G_{\beta,c}}.
$$
Taking the sum over all dimension vectors $\beta$, we get an $\R$-linear map
$$\pi:\NKHA(\Pi_Q)\to \SH_\K.$$
It is injective by Corollary \ref{cor:TF} because Condition \eqref{COND} holds.
Recall that $[M]$ is the class in 
$\G_0(T^*\calX)_\Lambda$ of a dg-module $M$ in $D^\b(\bfA_\beta\rtimes G_{\beta,c})_\Lambda$. 
Fix $\alpha$, $\gamma$ with sum $\beta=\alpha+\gamma$. 
Let $a_i,$ $b_i$, $c_i$ be the $i$th entries of the dimension vectors $\alpha,$ $\beta$, $\gamma$.
We must prove that for each dg-modules $M\in D^\b(\bfA_\alpha\rtimes G_{\alpha,c})_\Lambda$ and
$N\in D^\b(\bfA_\gamma\rtimes G_{\gamma,c})_\Lambda$ we have
\begin{align}\label{shuffle}
\pi[M\circ N]=
\Sym_{x}\big(\zeta^\circ_{\alpha,\gamma}\otimes \pi[M]\otimes \pi[N]\big).
\end{align}
Let the groups $L,$ $P,$ $U$ and $G$ be as in \S\ref{sec:DHAPQ}.
The tensor product $M\otimes N$ is a dg-module over $\bfA_L\rtimes L_c$.
We have 
$$M\circ N=\ind_{P,\,!}^G(\bfA_P\otimes_{\,\bfA_L}\! (M\otimes N)).$$
As an $L_c$-equivariant bigraded module over the bigraded algebra $\bfA_L$, we have
$$\bfA_P=\SS((\bar X_U)^*\oplus\fraku^*[1]\langle 2\rangle)\otimes\bfA_L,$$
where the group $L_c$ acts diagonally.
Note that
$$[\SS((\bar X_U)^*)]\otimes[\SS((\bar X_U)^*[1])]=1.$$
Thus, the Weyl character formula yields the identity \eqref{shuffle} with
the kernel $\zeta^\circ_{\alpha,\gamma}$ equal to the character of the following class
\begin{align*}
[\SS((\bar X_{U})^*[1])]^{-1}
\otimes[\SS(\fraku^*[1]\langle 2\rangle)]
\otimes[\Omega]
\otimes[\SS(\fraku^*[1])]^{-1}.
\end{align*}
Further, we have the following identities in $R_{T_c}$
\begin{align*}
[\SS((\bar X_U)^*[1])]
&=\prod_{h\in \bar Q_1}\prod_{\substack{0<s\leqslant a_i\\b_j\geqslant t>a_j}}(1-x_{i,s}/q_hx_{j,t}),\\
[\SS(\fraku^*[1]\langle d\rangle)]&=
\prod_{i\in Q_0}\prod_{\substack{0<s\leqslant a_i\\b_i\geqslant t>a_i}}(1-x_{i,s}/q^dx_{i,t}),\\
[\Omega]&=\prod_{h\in \bar Q^<_1}\prod_{\substack{0<s\leqslant a_i\\b_j\geqslant t>a_j}}(-x_{i,s}/q_hx_{j,t}).
\end{align*}
We deduce that the map $\pi$ is an $\R$-algebra embedding from $\NKHA(\Pi_Q)$
to $\SH_\K$ equipped with the shuffle product with kernel
\begin{align}\label{kernel3}
\begin{split}
&
\prod_{h\in \bar Q_1^<}
\prod_{\substack{0<r\leqslant a_i\\b_j\geqslant s> a_j}}\frac{1}{1-q_hx_{j,s}/x_{i,r}}
\cdot
\prod_{h\in \bar Q_1^>}
\prod_{\substack{0<r\leqslant a_i\\b_j\geqslant s> a_j}}\frac{1}{1-x_{i,r}/q_hx_{j,s}}
\cdot
\prod_{i\in Q_0}\prod_{\substack{0<r\leqslant a_i\\b_i\geqslant s>a_i}}
\frac{q^{-2}x_{i,r}-x_{i,s}}{x_{i,r}-x_{i,s}}.
\end{split}
\end{align}
Let $\rho$ be the composition of $\pi$ and the
conjugation  by the rational function
\begin{align*}
\prod_{h\in \bar Q^>_1}\prod_{\substack{0<r\leqslant b_i\\0<s\leqslant b_j}}(1-x_{i,r}/q_hx_{j,s}).
\end{align*}
The morphism $\rho$ satisfies the conditions of the lemma.

\end{proof}

\smallskip

Now, assume that the weight function satisfies Condition \eqref{normal}.
Then, we have $\zeta^\diamond_{\alpha,\gamma}=\zeta^\bullet_{\alpha,\gamma}$.
Thus, we have defined an $\R$-algebra embedding
\begin{align}\label{rho}
\rho:\NKHA(\Pi_Q)\to \SH_\K^\bullet
\quad,\quad
[\calO(\lambda)_{r\alpha_i}]\mapsto [V(\lambda)]
\quad,\quad
\lambda\in \P_{r\alpha_i}, r\in\bbZ_{>0}.
\end{align}
On the other hand, by Proposition \ref{prop:Shuffle} there is an $\R$-algebra homomorphism
\begin{align}\label{psi}
\psi:\widetilde\bfU^+\to \SH_\K^\bullet
\quad,\quad
(\e_{i,n})^{(r)}\mapsto(x_{i,1}x_{i,2}\cdots x_{i,r})^n
\quad,\quad i\in Q_0, n\in\bbZ, r\in\bbZ_{>0}.
\end{align}
The map $\psi$ factors through the embedding $\rho$ in \eqref{rho},
yielding the map $\phi$ in \eqref{phi}.

\end{proof}

\medskip

\emph{Step $2$ :}
We prove that the map $\phi:\bfU^+\to\NKHA(\Pi_Q)$ is surjective.
By Proposition \ref{prop:generators}, it is enough to check the following statement.

\smallskip

\begin{lemma}\label{lem:A1}
Let $\bfQ$ be a quiver of type $A_1$.
\begin{itemize}[leftmargin=8mm]
\item[$\mathrm{(a)}$]
The map $\phi:\bfU^+\to\NKHA(\Pi_Q)$ is surjective.
\item[$\mathrm{(b)}$]
We have $[\calO(n\omega_r)_{r\alpha_i}]=[\calO(n\omega_1)_{\alpha_i}]^{(r)}$ in $\NKHA(\Pi_Q)$.
\end{itemize}
\end{lemma}

\begin{proof}
Set $Q_0=\{i\}$. 
Since the graded dg-algebra $\bfA_{r\alpha_i}$
is isomorphic to the exterior algebra of $\frakg_{r}^*\langle 2\rangle$ with the zero differential,
the $\calA$-module $\NKHA(\Pi_Q)_{r\alpha_i}$ is spanned by the classes $[\calO(\lambda)_{r\alpha_i}]$
as $\lambda$ runs over the set of all dominant weights of $G_{r\alpha_i}$. 
By \eqref{rho}, we deduce that 
$$\rho(\NKHA(\Pi_Q)_{r\alpha_i})=\R_{G_{r\alpha_i,c}}.$$
We'll abbreviate $x_n=x_{i,n}$ and $\e_n=\e_{i,n}$ for each integer $n$.
Set $\lambda=(\lambda_1\geqslant\lambda_2\geqslant\dots\geqslant \lambda_r)$ and
$x^\lambda=x_1^{\lambda_1}x_2^{\lambda_2}\cdots x_r^{\lambda_r}.$
Let $i_n$ be the multiplicity of $n$ in $\lambda$.
We have
\begin{align}\label{power}
\begin{split}
 \lambda&=(n^{i_n}(n-1)^{i_{n-1}}(n-2)^{i_{n-2}}\dots),\\
\psi\big((\e_n)^{(i_n)}(\e_{n-1})^{(i_{n-1})}\cdots\big)&=
\Sym_x\Big(x^\lambda\prod_{s<t}
\frac{q^{-2}x_s-x_t}{x_s-x_t}\Big)\,\Big/\,[i_1]_q!\cdots[i_n]_q!.
\end{split}
\end{align}
This is the Hall-Littlewood polynomial $P_\lambda(x\,;\,q^{-2})$ up to the product by a power of $q$.
We deduce that $\psi(\bfU^+)=\rho(\NKHA(\Pi_Q))$, proving (a).
The identity in Part (b) follows immediately from the definition of the multiplication in $\NKHA(\Pi_Q)$.
Taking the image by the map $\rho$, it corresponds to the fact that
$P_{(n^r)}(x\,;\,q^{-2})=(x_1\dots x_r)^n.$
\end{proof}

\medskip

\emph{Step $3$ :}
Let $Q$ be any quiver of finite or affine type but not of type $A_1^{(1)}$.
We'll prove that the $R$-algebra homomorphism $\phi:\bfU^+\to\NKHA(\Pi_Q)$ constructed in Step 1
is injective.
To do this, we consider the topological ring $\bbQ[\![\hbar]\!]$.
It is an $\R$-algebra under the assignment $q\mapsto\exp(\hbar)$.
Consider the following topological tensor products
$$\widehat\bfU^+=\bfU^+\,\widehat\otimes_\R\,\bbQ[\![\hbar]\!]=\varprojlim_n\bfU^+/(q-1)^n\bfU^+
\quad,\quad
\widehat{\NKHA}(\Pi_Q)=\NKHA(\Pi_Q)\,\widehat\otimes_\R\,\bbQ[\![\hbar]\!].$$
Since the $\R$-algebra
$\bfU^+$ is separated at 1 by definition, the obvious map yields an $\R$-algebra embedding
$\bfU^+\subset\widehat\bfU^+$.
Thus it is enough to check that the $\bbQ[\![\hbar]\!]$-algebra homomorphism
$$\widehat\phi:\widehat\bfU^+\to\widehat{\NKHA}(\Pi_Q)$$
given by $\phi$ under base change is injective. Set
$$ \bfU^+_1=\bfU^+/(q-1)\bfU^+
\quad,\quad
\NKHA(\Pi_Q)_1=\NKHA(\Pi_Q)\,/\,(q-1)\NKHA(\Pi_Q).$$
By base change, the map $\phi$ yields a $\bbQ$-algebra homomorphism
$$\phi_1:\bfU^+_1\to\NKHA(\Pi_Q)_1.$$
It is enough to prove that the map $\phi_1$ is injective.
Indeed, if $\phi_1$ is injective then
$$\Ker(\widehat\phi)\subset\hbar\widehat\bfU^+.$$
By Lemma \ref{lem:red}, the $\bbQ[\![\hbar]\!]$-module
$\widehat{\NKHA}(\Pi_Q)$ is torsion free. 
We deduce that
$$\Ker(\widehat\phi)\subset\hbar\Ker(\widehat\phi).$$
Hence, we have
$\Ker(\widehat\phi)\subset\bigcap_{n>0}\hbar^n\widehat\bfU^+$.
Since $\bfU^+$ is separated at 1, we deduce that
the map $\widehat\phi$ is injective.

\begin{lemma} The map $\phi_1$ is injective.
\end{lemma}

\begin{proof}
The map $\phi_1:\bfU^+_1\to\G_0(\Lambda)_1$ is surjective by Step 2.
To prove the injectivity it is enough to prove the graded dimension of $\bfU^+_1$ is less than the
graded dimension of $\G_0(\Lambda)_1$,
for some filtrations compatible with the map $\phi_1$ to be defined.

The Lie algebra $\frakg_Q$ in \S\ref{sec:QLG} is isomorphic to the universal central extension of the current
algebra
$\frakg\otimes\bbQ[t^{\pm 1}]$ of a simple Lie algebra $\frakg$ over $\bbQ$, and
$\frakn_Q$ is the standard Borel subalgebra contained in $\frakg\otimes\bbQ[t]$.
Let $\Delta_Q\subset\bbN^{Q_0}$ be the set of roots of $\frakn_Q$.
Recall that 
$$\Delta_Q=(\Delta_++\bbN\delta)\sqcup(\Delta_-+\bbZ_{>0}\delta)\sqcup(\bbZ_{>0}\delta).$$

For any affine algebra $S$ the universal central extension of the current Lie algebra
$\frakg\otimes S$ is the sum of $\frakg\otimes S$ and its center, which is the quotient of the space of 
Kahler 1-forms of the ring $S$ by the exact forms. 
Let $(\bullet,\bullet)$ be a non zero invariant scalar product on $\frakg$.
The Lie bracket is such that
$[x\otimes f\,,\,y\otimes g]=[x,y]\otimes fg+(x,y)fdg$.

Set $S=\bbQ[s^{\pm 1}, t^{\pm 1}]$ and $\Sigma=\bbQ[\sigma, t^{\pm 1}]$.
Let $L\frakg_Q$ and $L^+\frakg_Q$ be the universal central extensions of 
$\frakg\otimes S$ and $\frakg\otimes\Sigma$.
The Lie algebra $L\frakg_Q$ is $\bbZ^{Q_0}\times\bbZ$-graded with the subspace
$\frakg_\beta\otimes s^at^b$ in degree $(\beta+b\delta,a)$, and the central elements
$\c_a=s^at^{-1}dt$ with $a\in\bbZ$ and $\c_{a,b}=s^{a-1}t^bds$ with $a,b\in\bbZ$ such that
$b\neq 0$ unless $a=0$ in degrees
$(0,2a)$ and $(b\delta,2a)$ respectively, see \cite{M90}.
A similar computation shows that
the Lie algebra $L^+\frakg_Q$ is $\bbZ^{Q_0}\times\bbN$-graded with the subspace
$\frakg_\beta\otimes \sigma^at^b$ in degree $(\beta+b\delta,a)$, and the central elements
$\c_\alpha=\sigma^\alpha t^{-1}dt$ with $\alpha\geqslant 0$ and
$\c_{\alpha,b}=\sigma^{\alpha-1} t^bd\sigma$ with $\alpha> 0$, $b\in \bbZ_{\neq 0}$  in degrees $(0,2\alpha)$ and
$(b\delta,2\alpha)$ respectively.
In both cases, 
the first component of the grading is called the horizontal grading, the second one the vertical grading.
Let $L\frakn_Q$ and $L^+\frakn_Q$
be the positive half of $L\frakg_Q$ and $L^+\frakg_Q$, i.e., the sum of all weight subspaces whose horizontal
degree belongs to $\Delta_Q$.
Let $\{\e_\beta^i\}$ be a basis of the root subspace $\frakg_\beta$ of $\frakg$.
A basis of $L\frakn_Q$ consists of elements
$\e_\beta^i\otimes s^a$ and $\c_{a,b}$ with $a\in\bbZ$, $b>0$.
A basis of $L^+\frakn_Q$ consists of elements
$\e_\beta^i\otimes \sigma^\alpha$ and $\c_{\alpha,b}$ with $\alpha, b>0$.
We deduce that
\begin{align}\label{form4}
\begin{split}
\sum_{\beta\in\Delta_Q}\sum_{\alpha\in\bbN}\dim(L^+\frakn_Q)_{\beta,\alpha}\,q^\alpha z^\beta&=
\sum_{\beta\in\Delta_+}\sum_{\alpha,b\in\bbN}q^{2\alpha}z^{\beta+b\delta}+
\sum_{\beta\in\Delta_-}\sum_{\alpha,b\in\bbN}q^{2\alpha}z^{\beta+(b+1)\delta}+\\
&\quad\sum_{\alpha,b\in\bbN}(q^2+\rk(\frakg))q^{2\alpha}z^{(b+1)\delta}.
\end{split}
\end{align}

Since $Q\neq A^{(1)}_1$,  the $\bbQ$-algebra $\bfU^+_1$ is 
isomorphic to the enveloping algebra of $L\frakn_Q$ by \cite[prop~1.5,\,1.6]{E03}.
The algebra embedding of $S$ into the topological ring $\widehat\Sigma=\bbQ[t^{\pm 1}][[\sigma]]$
such that $s,t^{\pm 1}\mapsto e^\sigma,t^{\pm 1}$ yields an embedding of $L^+\frakg_Q$
into the completed tensor product
$L^+\frakg_Q\widehat\otimes_\Sigma\widehat\Sigma$ which is compatible with the Lie brackets and the
horizontal degrees. Let $\gr_\bullet (\bfU^+_1)$ be the associated graded of the $J$-adic filtration
of $\bfU^+_1$ associated with the two-sided ideal $J$ generated by the elements
$\e_\beta^i\otimes s^a(s-1)$ and $\c_{a+1,b}-\c_{a,b}$ with $a\in\bbZ$ and $b>0$.
We deduce that there is a graded algebra isomorphism
$\gr_\bullet (\bfU^+_1)=U(L^+\frakn_Q)$ which takes
the subspace $\gr_\alpha (\bfU^+_1)_\beta$ isomorphically to $U(L^+\frakn_Q)_{\beta,2\alpha}$.
From the formula \eqref{form4}, we deduce that
\begin{align}\label{dimY}
\sum_{\beta\in\bbN^{Q_0}}\sum_{\alpha\in\bbN}\dim \gr_\alpha(\bfU^+_1)_\beta\,q^{2\alpha}z^\beta=
\text{Exp}\Big(\sum_{\beta\in\Delta_Q}\sum_{\alpha\in\bbN} p_\beta(q^2)q^{2\alpha}z^\beta\Big),
\end{align}
where Exp is the plethystic exponential attached to the formal variables $q$ and $z$,
and sum $p_\beta(q)\in\bbN[q]$ is the Kac polynomial associated with the
root $\beta$ of $\frakn_Q$.
We abbreviate $\Lambda_{\beta,1}=[N_\beta\,/\,G_\beta]$.
By \cite[thm~5.4]{SV18}, this implies that
\begin{align*}
\dim\gr_\alpha(\bfU^+_1)_\beta=\dim H_{2\dim\Lambda_{\beta,1}-2\alpha}(\Lambda_{\beta,1})
\quad,\quad\forall(\beta,\alpha)\in\bbN^{Q_0}\times\bbN.
\end{align*}

\smallskip

Let $X$ be an equidimensional  quasi-projective $G$-scheme.
Let $A_\alpha^G(X)_\bbQ$ be the rational equivariant Chow group of dimension $a$ as defined in \cite{EG98}.
Consider the quotient stack $\calX=[X\,/\,G]$, its dimension $n=\dim\calX$, and 
its rational Chow group $A_\alpha(\calX)_\bbQ=A^G_{\alpha+\dim G}(X)_\bbQ$. 
Let $I_G$ be the augmentation ideal of $\R_G$, i.e., the ideal of rank zero virtual representations.
Let $\widehat{\G_0(\calX)}$ be the $I_G$-adic completion of the $\R_G$-module $\G_0(\calX)$.
Edidin-Graham  \cite{EG00} defined a Riemann-Roch isomorphism
\begin{align}\label{RR}\tau_\calX:\widehat{\G_0(\calX)}\to\prod_{\alpha\in\bbN}A_{n-\alpha}(\calX)_\bbQ.\end{align}

\smallskip

We apply the Riemann-Roch isomorphism to the stack $\calX=\Lambda_{\beta,1}$.
Let 
$$\gr_\alpha\G_0(\Lambda_{\beta,1})=
(I_{G_\beta})^\alpha\G_0(\Lambda_{\beta,1})\,/\,(I_{G_\beta})^{\alpha+1}\G_0(\Lambda_{\beta,1}).$$
We get the following equality
\begin{align*}
\dim\gr_\alpha\G_0(\Lambda_{\beta,1})=
\dim A_{\dim\Lambda_{\beta,1}-\alpha}(\Lambda_{\beta,1}).
\end{align*}
We define
$$\G_0(\Lambda_\beta)_1=\G_0(\Lambda_\beta)\,/\,(q-1)
\quad,\quad
\gr_\alpha\G_0(\Lambda_\beta)_1=
(I_{G_{\beta}})^\alpha\G_0(\Lambda_\beta)_1\,/\,(I_{G_\beta})^{\alpha+1}\G_0(\Lambda_\beta)_1.$$
Note that $\G_0(\Lambda_\beta)$ is free as an $\R$-module by Lemma \ref{lem:red}, and
the specialization map $\G_0(\Lambda_\beta)_1\to\G_0(\Lambda_{\beta,1})$ 
is surjective by Proposition \ref{prop:genNKT} below.
Finally, we have $H_{2\alpha+1}(\Lambda_{\beta,1})=0$  and
the cycle map $A_{\alpha}(\Lambda_{\beta,1})_\bbQ\to H_{2\alpha}(\Lambda_{\beta,1})$ 
is surjective for each $\alpha$ by \cite[thm~A, prop~5.1]{SV18}.
We deduce that
$$\dim\gr_\alpha\G_0(\Lambda_\beta)_1\geqslant\dim\gr_\alpha\G_0(\Lambda_{\beta,1})\geqslant
\dim H_{2\dim\Lambda_{\beta,1}-2\alpha}(\Lambda_{\beta,1})=\dim\gr_\alpha(\bfU^+_1)_\beta.$$
This proves the lemma, as explained above.
\end{proof}

\smallskip

This finishes the proof of Theorem $\ref{thm:A}$.

\smallskip

\subsection{The case of a general quiver}\label{sec:GQ}
\subsubsection{A deformed version of Theorem $\ref{thm:A}$}\label{sec:GC}
Let $Q$ be any Kac-Moody quiver with $Q_1=\bar Q_1^<$ for some
total ordering of the vertices.
Fix $\T=\bbC^\times\times\A$ and fix a normal weight function on $\bar Q$, i.e., such that 
\eqref{normal} and \eqref{homogeneous} hold. We further assume that \eqref{large} holds.
This yields a KHA-realization of a multiparameter version of quantum groups which generalizes
Theorem $\ref{thm:A}$.
We equip the space $\SH_{\T,\K}$ in \eqref{SHT} with the $\K_\T$-algebra structure $\SH_{\T,\K}^\diamond$ 
whose multiplication $\diamond$ is the shuffle product with the kernel 
$\zeta^\diamond_{\alpha,\gamma}$ in \eqref{kernel}.
We have the following deformed version of Lemma \ref{lem:induction}.

\smallskip

\begin{lemma}\label{lem:inductionT}
Assume that Condition \eqref{large} holds.
There is an $\R_\T$-algebra embedding 
\begin{align}\label{rhoT}
\rho:\NKHA(\Pi_Q)_\T\to \SH_{\T,\K}^\diamond
\quad,\quad
[\calO(\lambda)_{r\alpha_i}]\mapsto[V(\lambda)]
\quad,\quad
\lambda\in \P_{r\alpha_i}, r\in\bbZ_{>0}.
\end{align}
\end{lemma}

\begin{proof}
Since the condition \eqref{large} holds, by Proposition \ref{prop:TFT}, the $\R_{G_{\beta,\T}}$-module 
$\NKHA(\Pi_Q)_{\beta,\T}$ is torsion free. 
Using the Proposition \ref{prop:TFT} instead of Corollary \ref{cor:TF} as in the proof of 
Lemma \ref{lem:induction}
yields the result.
\end{proof}

\smallskip

We now consider a particular setting. 
Choose a torus $\T$ satisfying \eqref{large} as follows :
 $\T=\bbC^\times\times\bbC^\times$ with
a normal weight function $h\mapsto q_h,t_h$ on the double quiver $\bar Q$ 
such that $t_h=t_{h'}^{-1}=t$ for all $h\in Q_1$,
where $t\in\X^*(\T)$ is second projection. 
Fix the twist $\Omega$ as in \S\ref{sec:triple}.
Then, we define a matrix $M=(m_{i,j}\,;\,i,j\in Q_0)$ such that we have
$$m_{i,j}=-m_{j,i}=1\ \text{if}\ i<j\ \text{and}\ a_{i,j}\neq 0\quad,\quad
m_{i,j}=0\ \text{else}.$$
Let $\bfU^+_{\T,\K}$ be the $\K_\T$-algebra generated by elements 
$\e_{i,s}$ with $i\in Q_0$ and $s\in\bbZ$
subject to the following defining relations :
\begin{itemize}[leftmargin=8mm]
\item[$\mathrm{(a)}$] if $a_{i,j}=0$, then we have the Drinfeld relation given by
$$[\,\e_i(z)\,,\,\e_j(w)\,]=0,$$
\item[$\mathrm{(b)}$] if $a_{i,j}\neq 0$, then we have the Drinfeld relation given by
$$(q^{a_{i,j}}z-t^{m_{i,j}}w)\,\e_i(z)\,\e_j(w)=(z-q^{a_{i,j}}t^{m_{i,j}}w)\,\e_j(w)\,\e_i(z),$$
\item[$\mathrm{(c)}$]  for $i\neq j$ and $l=1-a_{i,j}$ we have the Serre relation given by
$$\Sym_z\sum_{k=0}^l(-1)^k\left[\begin{matrix}l\cr k\end{matrix}\right]_q
\e_i(z_1)\dots\e_i(z_k)\,\e_j(w)\,\e_i(z_{k+1})\dots\e_i(z_l)=0.$$
\end{itemize}
We define $\widetilde\bfU^+_\T$ 
to be the $\R_\T$-subalgebra of $\bfU^+_{\T,\K}$ generated by the quantum divided powers
$(\e_{i,n})^{(r)}$ with $i\in Q_0$, $n\in\bbZ$ and $r\in\bbZ_{>0}$. 
The following lemma is proved in Appendix \ref{app:A}.

\smallskip

\begin{lemma}\label{lem:UST}
There is an $\R_\T$-algebra homomorphism
$$\psi:\widetilde\bfU^+_\T\to \SH^\diamond_{\T,\K}
\quad,\quad
(\e_{i,n})^{(r)}\mapsto(x_{i,1}x_{i,2}\cdots x_{i,r})^n
\quad,\quad
 i\in Q_0, n\in\bbZ, r\in\bbZ_{>0}.
$$
\qed
\end{lemma}

The map $\psi$ factors through the embedding $\rho$ in \eqref{rhoT},
yielding an $\R_\T$-algebra homomorphism $\widetilde\bfU^+_\T\to\NKHA(\Pi_Q)_\T$.

\smallskip

\begin{remark}
We can specialize the ring $\R_\T$ to $\R$ by specializing the parameter $t$ to 1.
We get an $\R$-algebra homomorphism
$$\widetilde\bfU^+_\T\,/\,(t-1)\to\NKHA(\Pi_Q)_\T\,/\,(t-1).$$
Composing this map with the obvious $\R$-algebra homomorphisms
$$\widetilde\bfU^+\to\widetilde\bfU^+_\T\,/\,(t-1)\quad,\quad\NKHA(\Pi_Q)_\T\,/\,(t-1)\to\NKHA(\Pi_Q)$$
we get an $\R$-algebra homomorphism
\begin{align}\label{phitilde}
\widetilde\bfU^+\to\NKHA(\Pi_Q)\quad,\quad
(\e_{i,n})^{(r)}\mapsto[\calO(n\omega_r)_{r\alpha_i}]
\quad,\quad i\in Q_0, n\in\bbZ, r\in\bbZ_{>0}.\end{align}
The map \eqref{phitilde} factorizes through an $\R$-algebra homomorphism 
$\phi:\bfU^+\to\NKHA(\Pi_Q)$ by Lemma \ref{lem:red}(b).
The map $\phi$ is invertible according to Steps 2,3 in \S \ref{sec:thmA},
hence we recover Theorem $\ref{thm:A}$ from this deformed version.
\end{remark}

\smallskip

\subsubsection{The case of a general quiver}\label{sec:GQ}

Let us indicate briefly how to modify Theorem \ref{thm:A} 
and the arguments above for a general quiver $Q$, possibly with edge loops.
We consider a torus $\T=\bbC^\times\times\A$ with a weight function as in \S\ref{sec:deformed}. 
Fix a total ordering on the set $Q_0$ and assume that $\bar Q_1^<\subset Q_1$ and the twist is as in 
\S\ref{sec:triple}.
We equip the space $\SH_{\T,\K}$ with a $\K_\T$-algebra structure $\SH_{\T,\K}^\diamond$ whose multiplication
$\diamond$ is defined as in \eqref{rhoT}.
The proof of Lemma \ref{lem:inductionT} extends to the case of the quivers with edge loops and
Propositions  \ref{prop:generators} and \ref{prop:TFT} yield the following

\begin{corollary}\label{cor:loops}\hfill
\begin{itemize}[leftmargin=8mm]

\item[$\mathrm{(a)}$]
The $\R_{G_{\beta,\T}}$-module $\NKHA(\Pi_Q)_{\beta,\T}$ is 
torsion free if \eqref{large} holds.

\item[$\mathrm{(b)}$]
There is an $\R_\T$-algebra embedding 
$\rho:\NKHA(\Pi_Q)_\T\to \SH_{\T,\K}^\diamond$ if \eqref{large} holds.

\item[$\mathrm{(c)}$]
The $\R_\T$-algebra $\NKHA(\Pi_Q)_\T$ is generated by the subset $\NKHA(\Pi_Q)_{\delta,\T}$.
\end{itemize}
\qed
\end{corollary}

Hopefully, part (b)  may permit to compute a presentation of the $\R_\T$-algebra $\NKHA(\Pi_Q)_\T$.
To recover a presentation of the algebras $\NKHA(\Pi_Q)$ or $\NKHA(\Pi_Q)_1$ from a presentation of
$\NKHA(\Pi_Q)_\T$, we need the following result.
Assume that $\T=\bbC^\times\times\bbC^\times$ with a weight function such that $t_h=t_{h'}^{-1}=t$ for all $h\in Q_1$ and   \eqref{large} holds.

\smallskip

\begin{proposition}\label{prop:genNKT}
The following specialization maps are surjective
$$\NKHA(\Pi_Q)_\T\,/\,(t-1)\to\NKHA(\Pi_Q)\to\NKHA(\Pi_Q)_1.$$
\end{proposition}

\begin{proof}
We'll prove that the first map is surjective.
The surjectivity of the second one is proved in a similar way.
It is enough to prove that the specialization map
$$\NKHA(\Pi_Q)_{\bullet\alpha_i,\T}\,/\,(t-1)\to\NKHA(\Pi_Q)_{\bullet\alpha_i}
\quad,\quad i\in Q_0,$$ 
is surjective. 
We may assume that $Q_0=\{i\}$ and $Q_1=\{h_1,h_2,\dots,h_g\}$.
If $g=0$, then the claim follows from the proof of Lemma \ref{lem:A1}.
We'll assume that $g>0$.
Set $\beta=r\alpha_i$. We have $G_\beta=GL_r$.
Let $\calP$ be the set of all partitions of $r$ if $g=1$, and the set of all compositions of $r$ if $g>1$.
For each $\lambda\in\calP$,
let $P_\lambda$ be the standard parabolic subgroup of $G_\beta$ of block-type $\lambda$, and
$F_\lambda$ the partial flag of vector subspaces in $\bbC^r$ stabilized by $P_\lambda$.
Set
$$
N_\lambda=G_\beta(\bar X(F_\lambda)\cap M_\beta).$$ 
By \cite[\S3.5]{SV18}, there is a $G_{\beta,\T}$-equivariant dense open subset 
$N^\circ_\lambda$ of $N_\lambda$ such that
\begin{itemize}[leftmargin=8mm]
\item $N_\beta=\bigsqcup_{\lambda\in \calP}N_\lambda^\circ$
is a partition into locally closed subsets,
\item $N_\lambda\subseteq\bigcup_{\mu\leqslant\lambda}N_\mu^\circ$ 
where $\leqslant$ is the dominance order.
\end{itemize}
Consider the substack $\Lambda^\circ_{\lambda,\T}=[N^\circ_\lambda\,/\,G_{\beta,\T}]$ of $\Lambda_{\beta,\T}$. 
Set $\Lambda_{\leqslant \lambda,\T}=\bigcup_{\mu\leqslant\lambda}\Lambda^\circ_{\mu,\T}$.
The localization exact sequence yields an exact sequence
$$\xymatrix{\G_0(\Lambda_{<\lambda,\T})\ar[r]&\G_0(\Lambda_{\leqslant\lambda,\T})\ar[r]^-{res}
&\G_0(\Lambda_{\lambda,\T}^\circ)\ar[r]& 0.}$$
By the four lemma, to prove the proposition it is enough to check that the specialization map
$\G_0(\Lambda_{\lambda,\T}^\circ)\to\G_0(\Lambda_{\lambda}^\circ)$
is surjective.
This follows from the following lemma.

\smallskip

\begin{lemma}
$\G_0(\Lambda_{\lambda,\T}^\circ)$ is generated by the fundamental class
$[\calO_{\Lambda_{\lambda,\T}^\circ}]$ as an $\R_{G,\T}$-module.
\end{lemma}

\smallskip

\begin{proof}
First, assume that $g=1$.
We have $X_\beta=\frakg^*_\beta$ and $X'_\beta=\frakg_\beta$.
Further, the subset
$N^\circ_\lambda\subset\frakg_\beta\times\frakg^*_\beta$
is the conormal bundle to the Richardson orbit 
$O_\lambda$ in $\frakg_\beta$
associated with the parabolic subgroup $P_\lambda$, i.e., the nilpotent $G_\beta$-orbit 
with Jordan type dual to the partition $\lambda$.
See \cite[\S 3.6.1]{SV18} for details.
Let $\fraku_\lambda$ be the Lie algebra of the unipotent radical of $P_\lambda$.
The open restriction and the Thom isomorphism yield a surjective map
$$\R_{P_\lambda\times T}\otimes[\calO_{\fraku_\lambda}]=\G_0([\fraku_\lambda\,/\,P_\lambda\times T])
\twoheadrightarrow\G_0([\fraku_\lambda\cap O_\lambda\,/\,P_\lambda\times T])=
\G_0(\Lambda^\circ_{\lambda,\T}).$$
This map takes the class $[\calO_{\fraku_\lambda}]$ to $[\calO_{\Lambda^\circ_{\lambda,\T}}]$,
proving the claim.

\smallskip

Next, assume that $g>1$. 
Then the claim follows from the proof of \cite[lem~5.13]{SV18}.
Indeed, the variety $N^\circ_\lambda$ is isomorphic to
an open subset of a vector bundle over the fixed points locus
of a cocharacter acting on a Nakajima quiver variety. See \cite[\S 4.1.2]{SV18} for details.
Then, the K-theoretic analogue of the Kirwan surjectivity 
statement proved in \cite[\S A.4]{SV18} implies the lemma.
\end{proof}

\end{proof}

\medskip

\section{K-theoretic Hall algebras and super quantum groups}\label{sec:super}

Let $\T=\bbC^\times\times\bbC^\times$ and  $q, t\in\X^*(\T)$ be the characters of weight $(1,0)$, $(0,1)$. 

\subsection{Super quantum groups}

Let $m,n\in\bbN$ with $(n,m)\neq(0,0)$ and $mn\neq 1,2$.
A \emph{parity sequence} of type $(m,n)$
is a tuple $\bfs=(s_1,s_2,\dots,s_{m+n})$ of $\pm 1$ with $m$ entries equal to 1.
Set $Q_0=\bbZ\,/\,(m+n)\bbZ$.
Let $\{\alpha_i\,;\,i\in Q_0\}$ be the basis of Dirac functions of the lattice $\bbZ^{Q_0}$.
Consider the bilinear form $(\,\text{-}:\text{-}\,)$ on $\bbZ^{Q_0}$ such that 
$$
(\alpha_i\,:\,\alpha_j)=(s_i+s_{i+1})\delta_{i,j}-s_i\delta_{i,j+1}-s_j\delta_{i+1,j}.
$$
We define matrices $A=(a_{i,j}\,;\,i,j\in Q_0)$ and $M=(m_{i,j}\,;\,i,j\in Q_0)$ by
\begin{align*}
a_{i,j}&=(\alpha_i\,:\,\alpha_j),\\
m_{i,i+1}&=-m_{i+1,i}=-s_{i+1},\\
m_{i,j}&=0\ \text{if}\ i-j\neq\pm 1.
\end{align*}
A vertex $i\in Q_0$ is odd if and only if $s_i=-s_{i+1}$, i.e., 
if and only if $a_{i,i}=0$. 
Else, the vertex $i$ is even and $a_{i,i}=\pm 2$.
We write $|i|=0$ if $i$ is even and $|i|=1$ if $i$ is odd.
Let $Q_0=Q_0^{ev}\sqcup Q_0^{odd}$ be the partition into even and odd vertices.

\smallskip

The positive part of the
\emph{quantum toroidal algebra} of $\frakg\frakl(m|n)$ with parity sequence $\bfs$ is the 
$\bbZ/2\bbZ\times\bbZ^{Q_0}$-graded $\K_\T$-algebra $\pmb\calU^+_\K$ generated by
$\e_{i,s}$ with $i\in Q_0$ and $s\in\bbZ$, 
subject to the defining relations $\mathrm{(a)}$ to $\mathrm{(d)}$ below :
\begin{itemize}[leftmargin=8mm]

\item[$\mathrm{(a)}$] if $a_{i,j}\neq 0$, then we have the Drinfeld relation
$$(q^{a_{i,j}}z-t^{m_{i,j}}w)\,\e_i(z)\,\e_j(w)=(-1)^{|i|\, |j|}(z-q^{a_{i,j}}t^{m_{i,j}}w)\,\e_j(w)\,\e_i(z),$$

\item[$\mathrm{(b)}$] if $a_{i,j}=0$, then we have the Drinfeld relation
$$[\,\e_i(z)\,,\,\e_j(w)\,]=0,$$

\item[$\mathrm{(c)}$]  if $i$ even and $j=i-1$ or $i+1$, then we have the cubic Serre relation
$$\Sym_z[\![\,\e_i(z_1)\,,\,[\![\,\e_i(z_2)\,,\,\e_j(w)\,]\!]]\!]=0,$$

\item[$\mathrm{(d)}$]  if $i$ odd, $i-1$, $i+1$ even, then we have the quartic Serre relation
$$\Sym_z[\![\,\e_i(z_1)\,,\,[\![\,\e_{i+1}(u)\,,\,[\![\e_i(z_2)\,,\,\e_{i-1}(v)\,]\!]]\!]]\!]=0.$$

\end{itemize}
See \cite{BM19} for more details. 
The \emph{parity} and the \emph{weight} of the generators are 
$$|\e_{i,s}|=|i|\in\bbZ/2\bbZ
\quad,\quad
\wt(\e_{i,s})=\alpha_i\in\bbZ^{Q_0}.$$
We used the generating series
$\e_i(z)=\sum_{r\in\bbZ}\e_{i,r}z^{-r}$ and the following commutators
\begin{align}\label{supercom}
[\![\,a\,,\,b\,]\!]=ab-(-1)^{|a|\,|b|}q^{(\wt(a)\,:\,\wt(b))}ba
\quad,\quad
[\,a\,,\,b\,]=ab-(-1)^{|a|\,|b|}ba.
\end{align}
We define $\pmb\calU^+_\T$ to be the $\R_\T$-subalgebra of $\pmb\calU^+_\K$
generated by the quantum divided powers
$(\e_{i,r})^{(l)}$ with $i\in Q_0$, $r\in\bbZ$ and $l\in\bbZ_{>0}$. 

\smallskip

\begin{remark}
Relations (c), (d) also hold for all other parities, but, then,  they follow from the relations above.
The quartic Serre relation is equivalent to
$$\Sym_z[\![[\![\,\e_i(z_1)\,,\,\e_{i+1}(u)]\!]\,,\,[\![\e_i(z_2)\,,\,\e_{i-1}(v)\,]\!]]\!]=0.$$
\end{remark}

\smallskip

\subsection{K-theoretic Hall algebras and super quantum groups}\label{sec:super2}
Let $Q$ be the quiver with set of vertices $Q_0$ and set of arrows
$Q_1=Q_1^+\sqcup Q_1^0\sqcup Q_1^-$ with
$$Q_1^+=\{x_i:i\to i+1\,;\,i\in Q_0\}
\quad,\quad
Q_1^-=\{y_i:i+1\to i\,;\,i\in Q_0\}
\quad,\quad
Q_1^0=\{\omega_i:i\to i\,;\,i\in Q_0^{ev}\}
.$$
Fix a total ordering of the vertices of $Q$.
We consider the weight function in $\X^*(\T)$ such that 
$$
q_h=q_{h'}=q^{-a_{i,j}}
\quad,\quad
t_h=t_{h'}^{-1}=t^{m_{i,j}}
\quad,\quad
\forall i\leqslant j.
$$
We define the potential 
\begin{align*}
W&=\sum_{i\in Q_0^{ev}} (\omega_iy_ix_i-\omega_ix_{i-1}y_{i-1})+
\sum_{i\in Q_0^{odd}} y_ix_ix_{i-1}y_{i-1}.
\end{align*}
The potential $W$ is $\T$-invariant because
$$
q_{x_i}=q^{s_{i+1}}
\quad,\quad
q_{y_i}=q^{s_{i+1}}
\quad,\quad
t_{x_i}=t^{-s_{i+1}}
\quad,\quad
t_{y_i}=t^{s_{i+1}}
\quad,\quad
q_{\omega_i}=q^{-2s_i}
\quad,\quad
t_{\omega_i}=1.
$$
The deformed K-theoretic Hall algebra without potential
$\KHA(Q)_\T$ is defined as in \S\ref{sec:DHA}.
The normalization of the twist $\Omega$ of the multiplication $\circ$ in \eqref{tp} 
is fixed in the proof of Theorem \ref{thm:B} below.
The deformed K-theoretic Hall algebra with potential $\KHA(Q,W)_\T$ is defined as in 
\S\ref{sec:DHAW}.
For each dimension vector $\beta\in\bbN^{Q_0}$ we have
$$\KHA(Q,W)_{\beta,\T}=\G_0^\sg(\calW_{\beta,\T})
\quad,\quad
\KHA(Q)_{\beta,\T}=\G_0(\calX_{\beta,\T})$$
where $\calW_{\beta,\T}$ and $\calX_{\beta,\T}$ are the following dg-stacks
\begin{align}\label{stacks}\calW_{\beta,\T}=[\,\{0\}\times^R_{\bbC} X_\beta\,/\,G_{\beta,\T}\,]
\quad,\quad
\calX_{\beta,\T}=[X_\beta\,/\,G_{\beta,\T}\,].\end{align}
The derived fiber product is relative to the function $w_\beta:X_\beta\to\bbC$ 
given by the trace of the potential $W$.
For each vertex $i\in Q_0$, each positive integer $l$ and each dominant weight $\lambda\in \P_{r\alpha_i}$, let  
$\calO(\lambda)_{\calX_{l\alpha_i,\T}}$ and $\calO(\lambda)_{\calW_{l\alpha_i,\T}}$ be
the obvious coherent sheaves over the stacks $\calX_{l\alpha_i,\T}$ 
and $\calW_{l\alpha_i,\T}$
associated with the representation $V(\lambda)$ of $G_{l\alpha_i}$.
More precisely, recall that the dg-stack
$\calW_{l\alpha_i,\T}$ is a ringed space with underlying topological space
$\calX_{l\alpha_i,\T}=[X_{l\alpha_i}\,/\,G_{l\alpha_i,\T}].$
We consider the odd and even cases separately :
\begin{itemize}[leftmargin=8mm]

\item 
If $i$ is odd, then we have $\calX_{l\alpha_i,\T}=BG_{l\alpha_i,\T}$ and $w_{l\alpha_i}=0$.
Hence 
$$\calW_{l\alpha_i,\T}=[R\Spec(\bbC[t]/t^2)\,/\,G_{l\alpha_i,\T}]
\quad,\quad\deg(t)=-1,$$
with the zero differential.
Then $\calO(\lambda)_{\calW_{l\alpha_i,\T}}$ is the sheaf of 
sections of the bundle $[V(\lambda)\,/\,G_{l\alpha_i,\T}]$ over $BG_{l\alpha_i,\T}$ with the zero $t$-action.

\item
If $i$ is even, then we have
$\calX_{l\alpha_i,\T}=[\frakg_{l\alpha_i}\,/\,G_{l\alpha_i,\T}]$ and
$w_{l\alpha_i}=0$. Hence 
$$\calW_{l\alpha_i,\T}=[R\Spec(\SS(\frakg^*_{l\alpha_i})\otimes\bbC[t]/t^2)\,/\,G_{l\alpha_i,\T}]
\quad,\quad\deg(t)=-1,$$
with the zero differential.
Then
$\calO(\lambda)_{\calW_{l\alpha_i,\T}}$ is the sheaf of sections of the bundle
$[V(\lambda)\times\frakg_{l\alpha_i}\,/\,G_{l\alpha_i,\T}]$
over $[\frakg_{l\alpha_i}\,/\,G_{l\alpha_i,\T}]$ with the zero $t$-action
\end{itemize}
The following theorem is an analogue in the super case of Theorem \ref{thm:A}(a).

\smallskip

\begin{theorem}\label{thm:B}
There is an $\bbN^{Q_0}$-graded $\R_\T$-algebra homomorphism
$\psi:\pmb\calU^+_\T\to \KHA(Q,W)_\T$ such that
$\psi((\e_{i,r})^{(l)})=[\calO(r\omega_l)_{\calW_{l\alpha_i,\T}}]$
for all $i\in Q_0,$ $ r,l\in\bbZ$ with $l>0.$
\end{theorem}

Before to prove the theorem, we first consider the K-theoretic
Hall algebra associated with the quiver $Q$ without potential
as in \S\ref{sec:DHA}.

\begin{lemma}\label{lem:super}
There is an $\bbN^{Q_0}$-graded $\R_\T$-algebra homomorphism
$\psi_0:\pmb\calU^+_\T\to \KHA(Q)_\T$ such that
$\psi_0((\e_{i,r})^{(l)})=[\calO(r\omega_l)_{\calX_{l\alpha_i,\T}}]$
for all $i\in Q_0,$ $ r,l\in\bbZ$ with $l>0.$
\end{lemma}

\begin{proof}
As an $\R_\T$-module, we have
$\KHA(Q)_\T=\bigoplus_{\beta\in\bbN^{Q_0}}\R_{G_{\beta,\T}}$.
The multiplication $*$ is the shuffle product associated with kernel \eqref{kernel0} such that
\begin{align*}
\zeta_{\alpha,\gamma}^*=\prod_{h\in Q_1}
\prod_{\substack{b_i\geqslant r>a_i\\0<s\leqslant a_j}}
(1-x_{i,r}/q_ht_hx_{j,s})
\cdot\prod_{i\in Q_0}\prod_{\substack{0<r\leqslant a_i\\b_i\geqslant s>a_i}}(1-x_{i,r}/x_{i,s})^{-1}.
\end{align*}
We normalize the twist $\Omega$ in \eqref{tp} such that
the twisted multiplication $\circ$  is the shuffle product associated with the following kernel
\begin{align*}
\zeta_{\alpha,\gamma}^\circ&=\prod_{h\in Q_1^<}
\prod_{\substack{b_i\geqslant r>a_i\\0<s\leqslant a_j}}
(-1)^{|i|\,\cdot\,|j|}(1-x_{i,r}/q_ht_hx_{j,s})
\cdot \prod_{h\in Q_1^>}\prod_{\substack{b_i\geqslant r>a_i\\0<s\leqslant a_j}}
(1-q_ht_hx_{j,s}/x_{i,r})\\
&\quad
\cdot\prod_{i\in Q_0^{ev}}\prod_{\substack{0<r\leqslant a_i\\b_i\geqslant s>a_i}}
\frac{x_{i,r}-q^{2s_i}x_{i,s}}{x_{i,r}-x_{i,s}}
\cdot\prod_{i\in Q_0^{odd}}\prod_{\substack{0<r\leqslant a_i\\b_i\geqslant s>a_i}}
\frac{1}{x_{i,r}-x_{i,s}}.
\end{align*}
Up to the conjugation by some rational function, this kernel can be written as follows
\begin{align*}
\zeta^\diamond_{\alpha,\gamma}
&=
\prod_{h\in Q^<_1}
\prod_{\substack{b_i\geqslant r>a_i\\0<s\leqslant a_j}}
(-1)^{|i|\,\cdot\,|j|}
\frac{t_{h'}x_{i,r}-q_{h}x_{j,s}}{q_{h'}t_{h'}x_{i,r}-x_{j,s}}
\cdot\prod_{i\in Q_0^{ev}}\prod_{\substack{0<r\leqslant a_i\\b_i\geqslant s>a_i}}
\frac{x_{i,r}-q^{2s_i}x_{i,s}}{x_{i,r}-x_{i,s}}
\cdot\prod_{i\in Q_0^{odd}}\prod_{\substack{0<r\leqslant a_i\\b_i\geqslant s>a_i}}
\frac{1}{x_{i,r}-x_{i,s}},\\
&=\prod_{\substack{h\in Q_1\\i\neq j}}
\prod_{\substack{b_i\geqslant r>a_i\\0<s\leqslant a_j}}\zeta_{i,j}(x_{i,r},x_{j,s})
\cdot\prod_{i\in Q_0}\prod_{\substack{b_i\geqslant r> a_i\\0<s\leqslant a_i}}
\zeta_{i,i}(x_{i,r},x_{i,s}),
\end{align*}
where the rational function $\zeta_{i,j}(u)$ is given by
\begin{align*}
\zeta_{i,j}(u,v)&=(-1)^{|i|\,\cdot\,|j|}
\frac{q^{a_{i,j}}u-t^{m_{i,j}}v}{u-q^{a_{i,j}}t^{m_{i,j}}v}\ \text{if}\  i<j,\\
\zeta_{i,j}(u,v)&=
1\ \text{if}\  i>j,\\
\zeta_{i,i}(u,v)&=\frac{q^{2s_i}u-v}{u-v}\ \text{if}\ i\in Q_0^{ev},\\
\zeta_{i,i}(u,v)&=\frac{1}{u-v}\ \text{if}\ i\in Q_0^{odd}.
\end{align*}
Let $\diamond$ be the corresponding shuffle product.
If $\beta=\alpha_i$ we abbreviate $x_i=x_{i,1}$ in $\R_{G_{\beta,\T}}$.
We define the following generating series 
$$e_i(z)=\sum_{r\in\bbZ}(x_i)^r\,z^{-r}\in\R_{G_{\beta, T}}[\![z^{-1}, z]\!].$$
The Drinfeld relations are easily checked.
Fix two vertices $i,$ $j$ in $Q_0$. 
If $i=j$ are odd, then a direct computation yields
\begin{align}\label{R1}e_i(z)\diamond e_i(w)=-e_i(w)\diamond e_i(z).\end{align}
If $i=j$ are even, then we have
\begin{align}\label{R4}(w-q^{2s_i}z)\,e_i(z)\diamond e_i(w)=(q^{2s_i}w-z)\,e_i(w)\diamond e_i(z).\end{align}
If $i< j$ and $a_{i,j}\neq 0$, then we have 
\begin{align}\label{R5}
(t^{m_{i,j}}w-q^{a_{i,j}}z)\,e_i(z)\diamond e_j(w)=(-1)^{|i|\,\cdot\,|j|}
(q^{a_{i,j}}t^{m_{i,j}}w-z)\,e_j(w)\diamond e_i(z).\end{align}
If $i< j$ and $a_{i,j}=0$, then we have
\begin{align}\label{R2}e_i(z)\diamond e_j(w)=(-1)^{|i|\,\cdot\,|j|}e_j(w)\diamond e_i(z).\end{align}
Now, we check the Serre relations. We abbreviate $e_i=e_{i,0}$ for all $i$.
It is enough to prove the following, see Appendix \ref{app:A} :
\begin{itemize}[leftmargin=8mm]

\item  if $i$ even and $j-i=\pm 1$, then 
\begin{align}\label{S1}[\![\,e_i\,,\,[\![\,e_i\,,\,e_j\,]\!]]\!]=0,\end{align}

\item if $i$ odd and $i-1$, $i+1$ even, then 
\begin{align}\label{S2}
[\![\,e_i\,,\,[\![\,e_{i+1}\,,\,[\![e_i\,,\,e_{i-1}\,]\!]]\!]]\!]=0.\end{align}

\end{itemize}
Relation \eqref{S1} is proved as in  Appendix \ref{app:A}, because $i$ is even.
Let us prove \eqref{S2}. 
Assume that $i$ is odd and $i-1$, $i+1$ even. 
For all $i_1,$ $i_2,$ $i_3,$ $i_4$ we  abbreviate 
$$
e_{i_1\,,\,i_2\,,\,i_3\,,\,i_4}=e_{i_1}\diamond e_{i_2}\diamond e_{i_3}\diamond e_{i_4}.$$
From \eqref{R1} and \eqref{R2}, we get
$$e_i\diamond e_i=0
\quad,\quad
e_{i-1}\diamond e_{i+1}=e_{i+1}\diamond e_{i-1}.$$
We deduce that
\begin{align*}
[\![\,e_i\,,\,[\![\,e_{i+1}\,,\,[\![e_i\,,\,e_{i-1}\,]\!]]\!]]\!]=
e_{i,i+1,i,i-1}-[2]_qe_{i,i+1,i-1,i}+e_{i,i-1,i,i+1}+e_{i+1,i,i-1,i}+e_{i-1,i,i+1,i}.
\end{align*}
Assume that the order of the vertices of $Q$ is such that $i-1<i<i+1$. Then, we have
$$\zeta_{i+1,i}=\zeta_{i,i-1}=\zeta_{i+1,i-1}=\zeta_{i-1,i+1}=1
\quad,\quad
\zeta_{i-1,i}(u,v)=\zeta_{i,i+1}(v,u).
$$
We abbreviate  $a=x_{i-1}$, $b_1=t^{-s_i}x_{i,1}$, $b_2=t^{-s_i}x_{i,2}$ and $c=x_{i+1}$. 
We must check that 
$$\Sym_{b}\Big(P(a,b_1,b_2,c)\,/\,(b_1-b_2)\Big)=0$$
where $P$ is the following rational function
\begin{align*}
P(a,b_1,b_2,c) &=
\frac{q a -b_2}{a-q b_2}+\frac{q c -b_1}{c-q b_1}
-[2]_q\frac{(q a-b_2)(q c -b_1)}{(a-q b_2)(c-q b_1)}\\
&\ +\frac{(q a-b_2)(q c -b_1)(q c-b_2)}{(a-q b_2)(c-q b_1)(c-q b_2)}
+\frac{(q a-b_1)(q a -b_2)(q c-b_1)}{(a-q b_1)(a-q b_2)(c-q b_1)}.
\end{align*}
This is done by a direct computation.
For other orderings the computation is very similar.
We may assume that $i<i-1<i+1$. Then, after the change of variables above, we must check that 
$$\Sym_{b}\Big(Q(a,b_1,b_2,c)\,/\,(b_1-b_2)\Big)=0$$
where $Q$ is the following rational function
\begin{align*}
Q(a,b_1,b_2,c) &=
\frac{a-q b_1}{q a-b_1}+\frac{q c -b_1}{c-q b_1}
-[2]_q\frac{(a-qb_1)(q c -b_1)}{(qa-b_1)(c-q b_1)}\\
&\ +\frac{(a-qb_1)(a -qb_2)(q c-b_1)}{(qa-b_1)(qa-b_2)(c-q b_1)}
+\frac{(a-qb_1)(q c -b_1)(q c-b_2)}{(qa-b_1)(a-q b_1)(c-q b_2)}.
\end{align*}
\end{proof}

\smallskip

We can now prove the theorem.

\begin{proof}[Proof of Theorem $\ref{thm:B}$]
By Lemma \ref{lem:pad}, there is an $\R_\T$-algebra homomorphism
\begin{align*}\rho:\KHA(Q,W)_\T\to\KHA(Q)_\T.\end{align*}
We'll prove that the map
$\psi_0:\pmb\calU^+_\T\to \KHA(Q)_\T$ factorizes through $\rho$ into a map
$$\psi:\pmb\calU^+_\T\to \KHA(Q,W)_\T,$$ 
i.e., we have a commutative diagram
\begin{align*}
\xymatrix{\KHA(Q,W)_\T\ar[r]^-\rho&\KHA(Q)_\T\\
\pmb\calU^+_\T\ar@{.>}[u]^-\psi\ar[ur]_-{\psi_0}&}
\end{align*}
Recall that for each dimension vector $\beta\in\bbN^{Q_0}$ we have
$$\KHA(Q,W)_{\beta,\T}=\G_0^\sg(\calW_{\beta,\T})
\quad,\quad
\KHA(Q)_{\beta,\T}=\G_0(\calX_{\beta,\T})$$
where $\calW_{\beta,\T}$ and $\calX_{\beta,\T}$ are the dg-stacks
in \eqref{stacks}.
There is an obvious l.c.i.~closed immersion of dg-stacks $i:\calW_\T\to\calX_\T$. 
The map $\rho$ is the pushforward by the map $i$. 
We have $\rho\psi(\e_{i,r})=\psi_0(\e_{i,r}).$
To prove that the assignment 
$\e_{i,r}\mapsto[\calO(r\omega_1)_{\calW_{\alpha_i,\T}}]$
for all $i\in Q_0$ and $r\in\bbZ$
gives an algebra homomorphism
$\psi:\pmb\calU^+_\T\to \KHA(Q,W)_\T$, we must check that the elements
$\psi(\e_{i,r})$ satisfy the defining relations of $\pmb\calU^+_\T$.
Since the map $\rho$ is an algebra homomorphism and these relations hold in $\KHA(Q)_\T$
by Lemma \ref{lem:super}, it is enough to check that the restriction $\rho_\beta$ of the map $\rho$ 
to $\KHA(Q,W)_{\beta,\T}$ is injective in the following cases
\begin{itemize}[leftmargin=8mm]
\item[$\mathrm{(a)}$] 
$\beta=\alpha_i+\alpha_j$, 
\item[$\mathrm{(b)}$]
$\beta=2\alpha_i+\alpha_j$, $j=i-1$ or $i+1$,  and $i$ even, 
\item[$\mathrm{(c)}$]
$\beta=\alpha_{i-1}+2\alpha_i+\alpha_{i+1}$, $i$ odd, and $i-1$, $i+1$ even.
\end{itemize}
We can also assume that the vertices are not all even, because, else, the claim follows from the proof of
Theorem \ref{thm:A}. 
The map $\rho_\beta$ is an isomorphism if $w_\beta=0$ by Lemma \ref{lem:pad}.
Now, we consider each case separately.

\smallskip

In (b) we may assume that $j$ is odd. Set $\beta=2\alpha_i+\alpha_j$.
Let $X_\beta$ be the space of 
$\beta$-dimensional representations of the following quiver
$$\begin{tikzcd}
i \arrow[in=145,out=215,loop, looseness=4, "\omega"] \arrow[r,bend left,"x"] & 
j  \arrow[l,bend left,"y"]
\end{tikzcd}$$
Let $\calX_{\beta,\T}=[X_\beta\,/\,G_{\beta,\T}]$ be the corresponding moduli stack,
and $\calW_{\beta,\T}$ be the zero fiber of the map
$$w_\beta:\calX_{\beta,\T}\to\bbC\quad,\quad(x,y,\omega)\mapsto x\omega y.$$
We have $G_{\beta,T}=GL_2\times GL_1\times\T$. The $G_{\beta,\T}$-action on $X_{\beta,\T}$ is such that
$$(a,b,q,t)\cdot(x,y,\omega)=(a^{-1}bqt^{-1}x,ab^{-1}qty,q^{-2}a\omega a^{-1}).$$
We must check that the pushforward yields an injective map
\begin{align}\label{PF}
\G_0^\sg(\calW_{\beta,\T})\to\G_0(\calX_{\beta,\T})\end{align}
The K-theoretic version of the dimensional reduction \eqref{Kdimred0} yields an isomorphism
\begin{align}\label{dimred}\G_0([\{(x,y,\omega)\,;\,yx=0=\omega\}\, /\,G_{\beta,\T}])=\G_0^\sg(\calW_{\beta,\T}).
\end{align}
It is enough to check that the left hand side of \eqref{PF} is torsion free over $\R_{G_{\beta,\T}}$. 
Indeed, since the $G_{\beta,\T}$-fixed points loci of $X_\beta$ and $\{(x,y,\omega)\,;\,yx=0=\omega\}$
are both equal to $\{0\}$, the Segal concentration theorem implies that the map \eqref{PF} is generically invertible.
To do so, we must compute the $R_{G_{\beta,\T}}$-module
\begin{align*}
\G_0([\{(x,y,\omega)\,;\,yx=0=\omega\}\, /\,G_{\beta,\T}])&=\G_0([\{(x,y)\,;\,yx=0\}\, /\,G_{\beta,\T}]).\end{align*}
In $\bbA^4=\Spec\bbC[x_1,x_2,y_1,y_2]$, we have
$$\{yx=0\}=\{x=0\}\cup\{y=0\}
\quad,\quad y=\begin{pmatrix}y_1\\y_2\end{pmatrix}\quad,\quad x=(x_1,x_2)$$ 
The Mayer-Vietoris exact sequence in equivariant K-theory for the closed covering above
yields an exact sequence
$$\G_0([\{0\}/G_{\beta,\T}])\to \G_0([\{y=0\}\, /\,G_{\beta,\T}])\oplus\G_0([\{x=0\}\, /\,G_{\beta,\T}])\to
\G_0([\{yx=0\}\, /\,G_{\beta,\T}])\to 0,$$
See \S\ref{sec:MayerVietoris} for details.
Let $\alpha, \beta\in R_{G_{\beta,\T}}$ be the alternating sums of the exterior powers of the representations
$\bbC x_1\oplus\bbC x_2$ and $\bbC y_1\oplus\bbC y_2$ of the group $G_{\beta,\T}$ respectively.
The exact sequence above is
$$\xymatrix{
R_{G_{\beta,\T}}\ar[r]^-{d_1}&R_{G_{\beta,\T}}\oplus R_{G_{\beta,\T}}\ar[r]^-{d_2}&\G_0([\{yx=0\}\, /\,G_{\beta,\T}])
\ar[r]&0}$$
where the map $d_1$ takes the unit in $R_{G_{\beta,\T}}$ 
to the element $(\alpha,\beta)$ in $R_{G_{\beta,\T}}^2$.
The sequence $(\alpha,\beta)$ is regular in $R_{G_{\beta,\T}}$.
So the cokernel of $d_1$ is isomorphic to the ideal of $R_{G_{\beta,\T}}$ generated by $\alpha$ and $\beta$
as an $R_{G_{\beta,\T}}$-module, because the map $d_1$ is
the first half of the Koszul complex associated with the regular sequence $(\alpha,\beta)$.
This is a torsion free $R_{G_{\beta,\T}}$-module.

\smallskip

In (a) we have $w_\beta=0$ unless $i$ is odd and $j$ is even with $i-j=\pm 1$ (or conversely).
We are reduced to check that the pushforward \eqref{PF}
is injective in the same setting as in (b) above except that the dimension vector is now $\beta=\alpha_i+\alpha_j$.
The reason for this is the same as in the case (b).

\smallskip

Finally we consider the case (c). 
Set
$\beta=\alpha_{j}+2\alpha_i+\alpha_{k}$.
Let $X_\beta$ be the space of 
$\beta$-dimensional representations of the following quiver
$$
\begin{tikzcd}
j\arrow[in=145,out=215,loop, looseness=4, "\omega_j"] \arrow[r,bend left,"x_j"] & 
i\arrow[r,bend left,"x_k"]  \arrow[l,bend left,"y_j"]& 
k \arrow[in=325,out=35,loop,  looseness=4, "\omega_k"]\arrow[l,bend left,"y_k"]
\end{tikzcd}
$$
Let $\calX_{\beta,\T}=[X_\beta\,/\,G_{\beta,\T}]$ be the corresponding moduli stack,
and $\calW_{\beta,\T}$ be the zero fiber of the map
$$w_\beta:\calX_{\beta,\T}\to\bbC\quad,\quad
z=(x_j,x_k,y_j,y_k,\omega_j,\omega_k)\mapsto 
\omega_jy_jx_j-\omega_kx_ky_k+
x_kx_jy_jy_k.$$
The $\T$-action on $X_\beta$ is such that
$$(q,t)\cdot z=
(qt^{-1}x_j,q^{-1}tx_k,qty_j,q^{-1}t^{-1}y_k,q^{-2}\omega_j,q^2\omega_k).$$
We must check that the pushforward
\begin{align}\label{PF2}\G_0^\sg(\calW_{\beta,\T})\to\G_0(\calX_{\beta,\T})\end{align}
is injective. The proof is similar to the proof in the case (b). Indeed, let
$Y_\beta$ be the closed subset of $X_\beta$ given by
$$Y_\beta=\{\omega_j=y_k=y_jx_j=\omega_kx_k-x_kx_jy_j=0\}$$
and let
$\calY_{\beta,\T}=[Y_\beta\,/\,G_{\beta,\T}]$ be the corresponding moduli stack.
The K-theoretic version of the dimensional reduction \eqref{Kdimred0}
along the coordinates $\omega_j$, $y_k$ yields an 
$\R_{G_{\beta,\T}}$-module isomorphism
$$\G_0(\calY_{\beta,\T})=\G_0^\sg(\calW_{\beta,\T}).$$
Using the Mayer-Vietoris exact sequence in \S\ref{sec:MayerVietoris}, one proves the following

\begin{claim} 
The $\R_{G_{\beta,\T}}$-module 
$\G_0(\calY_{\beta,\T})$
is torsion free. 
\qed
\end{claim}

Thus the injectivity follows from the Segal concentration theorem because
the $G_{\beta,\T}$-fixed points loci of $X_\beta$ and $Y_\beta$
are both equal to $\{0\}$.
\end{proof}

\smallskip

The following conjecture is an analogue in the super case of Theorem \ref{thm:A}(b).

\smallskip

\begin{conjecture}\label{conj:A}
The map $\psi:\pmb\calU^+_\T\to \KHA(Q,W)_\T$ is injective.
\qed
\end{conjecture}

\smallskip

\begin{remark} 
A similar shuffle realization of affine super quantum groups appears in \cite[\S 4]{T19}.
\end{remark}

\medskip

\begin{appendices}

\section{Proof of Lemma $\ref{lem:FILT}$}\label{app:C}
We fix a total order on the set $\Pi$ such that the subset
$X_{\leqslant O}=\bigcup_{O'\leqslant O}\calX_{O'}$ is closed in $X$ for each $O$. 
We have
$\G_1(\calX_O)^\top=\{0\}$
and the cycle map $\c_O:\G_0(\calX_O)\to\G_0(\calX_O)^\top$ is invertible.
Thus, the localization exact sequence 
in topological and algebraic equivariant K-theory yields the following commutative diagram
\begin{align*}
\xymatrix{
\G_0(\calX_{<O})\ar[r]^-{b_O}\ar[d]_-{\c_{<O}}&
\G_0(\calX_{\leqslant O})\ar[r]\ar[d]_-{\c_{\leqslant O}}&
\G_0(\calX_{O})\ar[r]\ar@{=}[d]_-{\c_O}&\{0\}\\
\G_0(\calX_{<O})^\top\ar[r]&
\G_0(\calX_{\leqslant O})^\top\ar[r]&
\G_0(\calX_{O,\T})^\top\ar[d]&\\
\{0\}\ar[u]&\G_1(\calX_{\leqslant O})^\top\ar[l]&\ar[l]
\G_1(\calX_{<O})^\top.
}
\end{align*}
The upper row and the lower rectangle are exact.
Now, assume that $\c_{<O}$ is invertible and $\G_1(\calX_{<O})^\top=\{0\}$.
Then $\G_1(\calX_{\leqslant O})^\top=\{0\}$.
Further, the map $b_O$ is injective and $\c_{\leqslant O}$ is invertible by the five lemma.
An induction implies that for each orbit $O$ the following holds
\hfill
\begin{itemize}[leftmargin=8mm]
\item the maps $\c_{<O}$ and $\c_{\leqslant O}$ are invertible,
\item $\G_1(\calX_{<O})^\top=\G_1(\calX_{\leqslant O})^\top=\{0\}.$
\end{itemize}
We deduce that the map $b_O$ is injective for each $O$, proving the lemma.

\section{Proof of Lemma $\ref{lem:UST}$}\label{app:A}

We must check that the defining relations of $\bfU^+_{\T,\K}$ hold in $\SH^\diamond_{\T,\K}$.
Since Condition \eqref{normal} holds, the kernel $\zeta^\diamond_{\alpha,\gamma}$ in \eqref{kernel}
takes the following form
\begin{align*}
\zeta^\diamond_{\alpha,\gamma}=
\prod_{\substack{i,j\in Q_0\\i<j}}
\prod_{\substack{0<r\leqslant a_i\\b_j\geqslant s>a_j}}\frac{q^{-a_{ij}}x_{i,r}-tx_{j,s}}{x_{i,r}-q^{-a_{ij}}tx_{j,s}}
\cdot\prod_{i\in Q_0}\prod_{\substack{0<r\leqslant a_i\\b_i\geqslant s>a_i}}
\frac{q^{-2}x_{i,r}-x_{i,s}}{x_{i,r}-x_{i,s}}.
\end{align*}
Relations (a) and (b) are straightforward. Let us concentrate on (c).
The proof of the relation (d) is similar way and is left to the reader.
By the same argument as in \cite{Gr94}, see also \cite[\S 10.4]{N00}, it is enough to prove the following relation
$$\sum_{h+k=l}(-1)^h\,
\psi(\e_{i,0})^{\diamond (h)}\diamond\psi(\e_{j,0})\diamond\psi(\e_{i,0})^{\diamond (k)}=0
\quad,\quad
i\neq j
\quad,\quad l=1-a_{i,j},$$
where $a^{\diamond(h)}$ is the $q$-divided power of $a$ relative to the multiplication $\diamond$.
We may assume that $i<j$. We'll abbreviate $x_r=x_{i,r}$, $y=tx_{j,1}$.
We have 
$$
\psi(\e_{i,0})^{\diamond (h)}=\frac{1}{[h]_q!}\Sym_x(\Delta_h)
\quad,\quad 
\Delta_h=\prod_{0<r<s\leqslant h}\frac{q^{-2}x_r-x_s}{x_r-x_s}.
$$
We deduce that
$$\psi(\e_{i,0})^{\diamond (h)}\diamond\psi(\e_{j,0})=\frac{1}{[h]_q!}
\Sym_x\Big(\Delta_h\cdot
\prod_{r=1}^h\frac{q^{-a_{ij}}x_r-y}{x_r-q^{-a_{ij}}y}\Big)\,\Big.
$$
Hence, we have
$$\psi(\e_{i,0})^{\diamond (h)}\diamond\psi(\e_{j,0})\diamond\psi(\e_{i,0})^{\diamond (k)}=
\frac{1}{[h]_q![k]_q!}
\Sym_x\Big(\Delta_{k+h}\cdot
\prod_{r=1}^h\frac{q^{-a_{ij}}x_r-y}{x_r-q^{-a_{ij}}y}\Big)\,\Big.
$$
We define
$$F(x_1,\dots,x_l,n)=
\sum_{h=0}^l(-1)^hq^{nh}\left[\begin{matrix}l\cr h\end{matrix}\right]_q\,
\Sym_x\Big(\Delta_l\cdot\prod_{r=1}^h(q^{l-1-n}x_r-1)\cdot\prod_{s=h+1}^l(x_s-q^{l-1-n})\Big).
$$
We must prove that $F(x_1,\dots,x_l,0)$ vanishes.
The Pascal identity for the $q$-Gaussian binomial coefficient implies that
\begin{align*}
F(x_1,\dots,x_l,n)
&=\Sym_x\left(\sum_{h=0}^{l-1}(-1)^hq^{(n-1)h}\left[\begin{matrix}l-1\cr h\end{matrix}\right]_q\cdot
\Delta_l\cdot
\prod_{r=1}^h(q^{l-1-n}x_r-1)\cdot
\prod_{s=h+1}^l(x_s-q^{l-1-n})
\right.-\\
&\quad \sum_{h=0}^{l-1}(-1)^hq^{n+l-1+(n-1)h}\left[\begin{matrix}l-1\cr h\end{matrix}\right]_q\cdot
\left.\Delta_l\cdot
\prod_{r=1}^{h+1}(q^{l-1-n}x_r-1)\cdot
\prod_{s=h+2}^l(x_s-q^{l-1-n})
\right).
\end{align*}
We first check that $F(x_1,\dots,x_l,n)$ does not depend on $x_1,\dots,x_l$ by induction on $l$.
Set $F(l-1,n-1)=F(x_1,\dots,x_{l-1},n-1)$.
The induction hypothesis implies that
\begin{align*}
F(x_1,\dots,x_l,n)
&=F(l-1,n-1)\sum_{k=1}^l\Bigg(
\prod_{\substack{0<r\leqslant l\\r\neq k}}\frac{q^{-2}x_r-x_k}{x_r-x_k}(x_k-q^{l-1-n})-\\
&\quad q^{n+l-1}\prod_{\substack{0<s\leqslant l\\s\neq k}}\frac{q^{-2}x_k-x_s}{x_k-x_s}
(q^{l-1-n}x_k-1)\Bigg).
\end{align*}
We define
$$A=\prod_{r=1}^l\frac{q^{-2}x_r-z}{x_r-z}(z-q^{l-1-n})\frac{dz}{z}
\quad,\quad
B=-q^{n+l-1}\prod_{s=1}^l\frac{q^{-2}z-x_s}{z-x_s}(q^{l-1-n}z-1)\frac{dz}{z}.$$
We have
\begin{align*}
(1-q^{-2})\,F(x_1,\dots,x_l,n)
&=F(l-1,n-1)\sum_{k=1}^l\big(\res_{x_k}A+\res_{x_k}B\big)\\
&=-F(l-1,n-1)\big(\res_0A+\res_\infty A+\res_0B+\res_\infty B\big)\\
&=F(l-1,n-1)(q^{-n-l-1}+q^{n+l-1}-q^{n-l-1}-q^{-n+l-1})\\
&=q^{-1}F(l-1,n-1)(q^n-q^{-n})(q^{l}-q^{-l}).
\end{align*}
This proves our claim, and that 
$$(1-q^{-2})\,F(l,n)=q^{-1}F(l-1,n-1)(q^n-q^{-n})(q^{l}-q^{-l}).$$
Hence $F(l,0)=0$, proving the lemma.

\medskip

\section{Cohomological Hall algebra}
We define the rational Borel-Moore homology $H_\bullet(\-\,,\bbQ)$ of a stack as in \cite{KV19}. 
It is the dual of the cohomology with compact support $H^\bullet_c(\-\,,\bbQ)$.
By the Borel-Moore homology of a
dg-stack we'll mean the Borel-Moore homology of its truncation.
Similarly, the fundamental class of a dg-stack in Borel-Moore 
homology will mean the fundamental class of its truncation.
The results in the previous sections have an analogue by replacing everywhere the rational Grothendieck group 
$\G_0(\text{-})\otimes\bbQ$ by the rational
Borel-Moore homology $H_\bullet(\text{-},\bbQ)$. 
Let us explain this briefly.
For each linear algebraic group $G$ we set $H_G^\bullet=H^\bullet(BG,\bbQ)$.
If $G=\bbC^\times$ then  $H_G^\bullet=\bbQ[\hbar]$ where $\hbar=\c_{1,G}$ 
is the first Chern class of the linear character $q$.
If $G=\T=\bbC^\times\times\bbC^\times$ then
$H_\G^\bullet=\bbQ[\hbar,\varepsilon]$ where $\hbar$, $\varepsilon$ 
are the first Chern classes of the characters 
$q,$ $t$ of weight $(1,0)$ and $(0,1)$.
We abbreviate $\A=H^\bullet_{\bbC^\times}$.

\smallskip

\subsection{The Yangian of $\frakg_Q$}

Given an homogeneous weight function $\bar Q_1\to\X^*(\T)$, we consider 
the dg-moduli stack $T^*\calX$ of representations of the preprojective algebra $\Pi_Q$.
The $\R$-algebra structure on $H_\bullet(T^*\calX,\bbQ)$ and
$H_\bullet(T^*\calX,\bbQ)_\Lambda$ is as in \cite{SV18}, see also \cite{KV19}, \cite{Y18}. 
The first algebra is called the
\emph{cohomological Hall algebra} of $\Pi_Q$, and the second one the
\emph{nilpotent cohomological Hall algebra} of $\Pi_Q$. 
Let us denote them by $\bfH(\Pi_Q)$ and $\bfN\bfH(\Pi_Q)$.

\smallskip

Given a $\T$-invariant potential potential $W$ on a quiver $Q$, the deformed
cohomological Hall algebras $\bfH(Q)_\T$ and $\bfH(Q,W)_\T$ of the quiver $Q$, with and without potential,
are defined in \cite{KS} in the particular case where $\T=\{1\}$.
They are studied further in several works, in particular in 
\cite{D16} in the triple quiver case.

\smallskip

Now, let the quiver $Q$ be of Kac-Moody type. We define
$\bfY^+$ to be the positive part of the Yangian of $\frakg_Q$.
It is an $\R$-algebra generated by elements
$\e_{i,r}$ with $i\in Q_0$ and $ r\in\bbN$ modulo some relations which are analoguous to the relations (a) to (c) 
in \S\ref{sec:QLG}.
Let $[T^*\calX_{\alpha_i}]$ be the fundamental class in 
$\bfN\bfH(\Pi_Q)_{\alpha_i}=H_\bullet(T^*\calX_{\alpha_i}\,,\,\bbQ).$

\smallskip

The following theorem is analogous to  Theorem \ref{thm:A}.

\smallskip

\begin{Theorem}\label{thm:C}
Fix a normal weight function on $\bar Q$ in $\X^*(\bbC^\times)$.
\begin{itemize}[leftmargin=8mm]
\item[$\mathrm{(a)}$] 
There is a surjective $\bbN^{Q_0}$-graded $\R$-algebra homomorphism 
$\phi:\bfY^+\to\bfN\bfH(\Pi_Q)$ such that
$\phi(\e_{i,r})=(\c_{1,G_{\alpha_i}})^r\cup[T^*\calX_{\alpha_i}]$
for all $i\in Q_0$ and $ r\in\bbN.$
\item[$\mathrm{(b)}$] 
If $\bfQ$ is of finite or affine type but not of type $A^{(1)}$, then the map $\phi$ is injective.
\qed
\end{itemize}
\end{Theorem}

\smallskip

The proof of Theorem \ref{thm:A} relies on two types of arguments : 
the geometric lemmas in \S\ref{sec:ProofA}, and the combinatorial Lemma $\ref{lem:UST}$
proved in Appendix \ref{app:A}.
The proof of Theorem \ref{thm:C} is similar : the geometric lemma are proved in \cite{SV18}, \cite{D16},
and the analogue of the computation in Appendix \ref{app:A} is done in \cite[App.~A]{Y18}.

\smallskip

\subsection{The Yangian of $\widehat{\fraks\frakl}(m|n)$}\label{sec:SY}

Let the quiver $Q$, the potential $W$ 
and the $\T$-action on the representation space of $Q$ be as in \S\ref{sec:super}.
Let $\pmb\calY^+_\T$ be the positive part of the
affine Yangian of type $\fraks\frakl(m|n)$ with $(m,n)\neq(0,0)$  and $mn\neq 1,2$.
It is the $H_\T^\bullet$-algebra
generated by 
$\e_{i,r}$ with $i\in Q_0$ and $r\in\bbN$, subject to the defining relations $\mathrm{(a)}$ to $\mathrm{(d)}$ below,
see \cite{U19}.

\begin{itemize}[leftmargin=8mm]

\item[$\mathrm{(a)}$] if $a_{i,j}\neq 0$, then we have the Drinfeld relation
$$[\e_{i,r}\,,\,\e_{j,s+1}]-[\e_{i,r+1}\,,\,\e_{j,s}]=
-a_{i,j}\hbar\,\{\e_{i,r}\,,\,\e_{j,s}\}
+m_{i,j}\varepsilon\,[\e_{i,r}\,,\,\e_{j,s}],$$

\item[$\mathrm{(b)}$] if $a_{i,j}=0$, then we have the Drinfeld relation
$$[\,\e_{i,r}\,,\,\e_{j,s}\,]=0,$$

\item[$\mathrm{(c)}$]  if $i$ even and $j-i=\pm 1$, then we have the cubic Serre relation
$$\Sym_r[\,\e_{i,r_1}\,,\,[\,\e_{i,r_2}\,,\,\e_{j,s}\,]]=0,$$

\item[$\mathrm{(d)}$]  if $i$ odd and $i-1$, $i+1$ even, then we have the quartic Serre relation
$$\Sym_r[\,\e_{i,r_1}\,,\,[\,\e_{i+1,s_1}\,,\,[\e_{i,r_2}\,,\,\e_{i-1,s_2}\,]]]=0.$$
\end{itemize}
We used the following commutators
$$[\,a\,,\,b\,]=ab-(-1)^{|a|\,|b|}ba
\quad,\quad
\{\,a\,,\,b\,\}=ab+(-1)^{|a|\,|b|}ba.$$

\smallskip

Let $\gr(\-\,)$ denote the associated graded for the \emph{topological filtration}, see below for more details. 
By \cite[prop.5.3, cor.5.6]{P19} there is a Riemann-Roch map
$$\gr\bfK(Q,W)_\T\to\bfH(Q,W)_\T
\quad,\quad
\gr\bfK(Q)_\T\to\bfH(Q)_\T.$$
It is a ring homomorphism. 
We have a decomposition
$$\bfH(Q)_\T=\bigoplus_\beta \bfH(Q)_{\beta,\T}
\quad,\quad
\bfH(Q)_{\beta,\T}=H_\bullet(\calX_{\beta,\T}\,,\bbQ).$$
Let $[\calX_{\beta,\T}]$ be the fundamental class in $\bfH(Q)_{\beta,\T}$.
Let $i:\pi_0(\calW_{\beta,\T})\to\calX_{\beta,\T}$ be the inclusion of the zero fiber (underived) of the map 
$w_\beta:\calX_{\beta,\T}\to\bbC$.
Let $\psi_{w_\beta},\varphi_{w_\beta}[-1]: D^\b_c(\calX_{\beta,\T})\to D^\b_c(\calX_{\beta,\T})$ be
the nearby-cycle and vanishing-cycle functors. 
By definition of the cohomological Hall algebra of the quiver with potential $(Q,W)$, see, e.g., \cite{D17a},
we have the decomposition
$$\bfH(Q,W)_\T=\bigoplus_\beta \bfH(Q,W)_{\beta,\T}
\quad,\quad
\bfH(Q,W)_{\beta,\T}=H_\bullet(\calX_{\beta,\T}\,,\,\varphi_{w_\beta}\bbQ[-1]).$$
There are canonical morphisms of functors 
$$\psi_{w_\beta}[-1]\to\varphi_{w_\beta}[-1]\to i^*,$$
see \cite[(8.6.7)]{KS90}.
It yields a map
$$
H^\bullet_c(\calX_{\beta,\T}\,,\,\varphi_{w_\beta}\bbQ[-1])\to 
H^\bullet_c(\pi_0(\calW_{\beta,\T})\,,\,\bbQ).$$
Recall that we defined
$$H_\bullet(\calW_{\beta,\T}\,,\,\bbQ)=H_\bullet(\pi_0(\calW_{\beta,\T})\,,\,\bbQ).$$
Taking the dual, we get a map
\begin{align*}
b:H_\bullet(\calW_{\beta,\T}\,,\,\bbQ)\to 
H_\bullet(\calX_{\beta,\T}\,,\,\varphi_{w_\beta}\bbQ[-1])=\bfH(Q,W)_{\beta,\T}.
\end{align*}
Let $[\calW_{\beta,\T}]$ denote both the fundamental class of $\pi_0(\calW_{\beta,\T})$ in 
$H_\bullet(\calW_{\beta,\T}\,,\,\bbQ)$ and its image in $\bfH(Q,W)_{\beta,\T}$ by the map $b$.
The following theorem is analogous to Lemma \ref{lem:super}
and Theorem \ref{thm:B}.

\smallskip

\begin{Theorem}\label{thm:D}\hfill
\begin{itemize}[leftmargin=8mm]
\item[$\mathrm{(a)}$]
There is an $\bbN^{Q_0}$-graded $H_\T^\bullet$-algebra homomorphism
$\psi_0:\pmb\calY^+_\T\to \bfH(Q)_\T$ such that
$\psi_0(\e_{i,r})=(\c_{1,G_{\alpha_i}})^r\cup[\calX_{\alpha_i,\T}]$
for all $i\in Q_0$ and $ r\in\bbN$.
\item[$\mathrm{(b)}$]
There is an $\bbN^{Q_0}$-graded $H_\T^\bullet$-algebra homomorphism
$\psi:\pmb\calY^+_\T\to \bfH(Q,W)_\T$
such that
$\psi(\e_{i,r})=(\c_{1,G_{\alpha_i}})^r\cup[\calW_{\alpha_i,\T}]$
for all $i\in Q_0$ and $ r\in\bbN$.
\end{itemize}
\end{Theorem}

\begin{proof}
The proof of (a) is similar to the proof of Lemma \ref{lem:super}.
Let us concentrate on (b).

\smallskip

Given a linear algebraic group $G$, let
$I_G$ be the augmentation ideal of $\R_G$.
Let $X$ be a quasi-projective $G$-scheme and $\calX=[X\,/\,G]$ the quotient stack.
The $I_G$-adic filtration of the $\R_G$-module $\G_0(\calX)$ is called the topological filtration,
see, e.g., \cite[\S 6]{EG00}.
Let $\widehat{\G_0(\calX)}$ be the completion with respect to this filtration.
Assume that $\calX$ is pure of dimension $n$.
We define 
$$\widehat{H_\bullet(\calX,\bbQ)}=\prod_{a\geqslant 0}H_{2n-a}(\calX,\bbQ).$$
Composing the Riemann-Roch isomorphism in \eqref{RR} with the cycle map we get a map
\begin{align*}
\widehat\tau_\calX:\widehat{\G_0(\calX)}\to\widehat{H_\bullet(\calX,\bbQ)}\end{align*}
which satisfies the usual properties, see \cite[thm.~3.1]{EG00} and \cite[thm.~18.3]{F84}.
Taking the associated graded with respect to the topological filtration,
we get a map
\begin{align}\label{RR2}
\tau_\calX:\gr{\G_0(\calX)}\to H_\bullet(\calX,\bbQ)\end{align}
which is covariant for proper morphisms, contravariant for l.c.i.~morphisms and which takes the class of
$\calO_\calY$ to the fundamental class $[\calY]$ for any pure dimensional closed substack $\calY$ of $\calX$. 
The singularity K-theory group fits into an exact sequence
$$\xymatrix{\K_0(\calX)\ar[r]&\G_0(\calX)\ar[r]^-c&\G_0^\sg(\calX)\to 0}.$$
The topological filtration on $\G_0(\calX)$ induces a filtration on the quotient $\G_0^\sg(\calX)$.
Let  $\gr\,\G_0^\sg(\calX)$ be the associated graded. 

\smallskip

Coming back to the setting of the theorem, we define
$$\gr\,\bfK(Q,W)_\T=\bigoplus_\beta\gr\,\G_0^\sg(\calW_{\beta,\T}).$$
The map \eqref{RR2} yields a map
$$\tau_{\calW_{\beta,\T}}:\gr{\G_0(\calW_{\beta,\T})}\to H_\bullet(\calW_{\beta,\T},\bbQ).$$
By \cite[cor~5.6]{P19} there is an algebra homomorphism
$\tau:\gr\,\bfK(Q,W)_\T\to\bfH(Q,W)_\T$ such that the following square commutes
$$\xymatrix{
\gr{\G_0(\calW_{\beta,\T})}\ar[d]_-{\gr(c)}\ar[r]^{\tau_{\calW_{\beta,\T}}}&H_\bullet(\calW_{\beta,\T},\bbQ)\ar[d]^-b\\
\gr\,\bfK(Q,W)_\T\ar[r]^\tau&\bfH(Q,W)_\T
}$$
There is a filtration on the algebra 
$\pmb\calU^+_\T$ with associated graded $\gr\,\pmb\calU^+_\T$ and an algebra 
homomorphism $\pmb\calY^+_\T\to \gr\,\pmb\calU^+_\T$ such that the map
$\psi:\pmb\calU^+_\T\to \KHA(Q,W)_\T$ in Theorem \ref{thm:B} is compatible with the filtrations.
Composing $\tau$ with the associated graded
$\gr\psi:\gr\,\pmb\calU^+_\T\to \gr\KHA(Q,W)_\T$ we get a map
$$\psi:\pmb\calY^+_\T\to \bfH(Q,W)_\T$$
such that 
$$\psi(\e_{i,0})=\tau\big([\calO_{\calW_{\alpha_i,\T}}]\big)
=[\calW_{\alpha_i,\T}]\quad,\quad\forall i\in Q_0.$$
The map $\psi$ is the required algebra homomorphism.
\end{proof}

\smallskip

We also have the following analogue of  Conjecture \ref{conj:A}.

\smallskip

\begin{Conjecture}\label{conj:B}
There is an $H_\T^\bullet$-algebra homomorphism
$\rho:\bfH(Q,W)_\T\to\bfH(Q)_\T.$
The $H_\T^\bullet$-algebra homorphism $\psi_0:\pmb\calY^+_\T\to \bfH(Q)_\T$ 
factorizes through the map $\rho$ into 
the $H_\T^\bullet$-algebra homorphism $\psi:\pmb\calY^+_\T\to \bfH(Q,W)_\T$ and the latter is injective.
\qed
\end{Conjecture}

\medskip

\section{Mayer-Vietoris in equivariant K-theory}\label{sec:MayerVietoris}

Let $X$ be a $T$-equivariant quasi-projective variety with a covering $X=A\cup B$ where
$A$, $B$ are closed $T$-equivariant subsets of $X$. Let $Y=A\cap B$.
The following lemma is a consequence of the Mayer-Vietoris long exact sequence
for the K-theory of $T$-equivariant coherent sheaves. A well-known proof which goes back to
Quillen and Thomason uses the K-theory spectrum
of $T$-equivariant coherent sheaves. Let us indicate a simpler
proof using only the localization exact sequence in
equivariant K-theory. It is enough for our purpose.

\smallskip

\begin{Lemma} There is an exact sequence
$$\G_0([Y\,/\,T])\to\G_0([A\,/\,T])\oplus\G_0([B\,/\,T])\to\G_0([X\,/\,T])\to 0.$$
\end{Lemma}

\begin{proof} Set $U=X\setminus A$ and $V=X\setminus B$.
We have $U=B\setminus Y$ and $V=A\setminus Y$.
We consider the following commutative diagram 
$$\xymatrix{
V\ar@{^{(}->}[r]^-{j_2}\ar@{=}[d]&A\ar[d]_-{i_3}&\ar[l]_-{i^1}Y\ar[d]^-{i_2}\\
V\ar@{^{(}->}[r]^-{j_1}&X&\ar[l]_-{i^4}B
}$$
The localization long exact sequence yields the following commutative diagram
with exact rows
\begin{align*}
\xymatrix{
\G_1([V\,/\,T])\ar[r]^-\partial
\ar@{=}[d]&\G_0([B\,/\,T])\ar[r]^{i^4_*}&\G_0([X\,/\,T])\ar[r]^{j_1^*}&\G_0([V\,/\,T])\ar[r]\ar@{=}[d]&0\\
\G_1([V\,/\,T])\ar[r]^-\delta&\G_0([Y\,/\,T])\ar[r]^{i^1_*}\ar[u]^-{i^2_*}&\G_0([A\,/\,T])\ar[r]^{j_2^*}\ar[u]^-{i^3_*}&
\G_0([V\,/\,T])\ar[r]&0
}
\end{align*}
We consider the maps 
$$\alpha:\G_0([Y\,/\,T])\to\G_0([A\,/\,T])\oplus\G_0([B\,/\,T])\quad,\quad 
\beta:\G_0([A\,/\,T])\oplus\G_0([B\,/\,T])\to\G_0([X\,/\,T])$$ given by
$$\alpha=\begin{pmatrix}\ \ i^1_*\\-i^2_*\end{pmatrix}\quad,\quad
\beta=(i^3_*\ i^4_*).$$
The map $\beta$ is surjective because for each element $x\in\G_0([X/T])$ there is an element
$a\in\G_0([A/T])$ such that $j_2^*(a)=j_1^*(x)$. Hence $j_1^*(x-i^3_*(a))=0$. Thus, there is an element
$b\in \G_0([B/T])$ such that $i_*^4(b)=x-i^3_*(a)$. We also have $\beta\circ\alpha=0$. 
Thus, it is enough to check that $\Ker(\beta)\subset\Im(\alpha)$. To do that, fix an element $(a,b)$ in the kernel.
Then, we have $i_*^3(a)=-i_*^4(b)$. Thus $j_2^*(a)=j^*_1i_*^3(a)=-j_1^*i_*^4(b)=0$. Hence
$a=i_*^1(y)$ for some element $y\in \G_1([Y\,/\,T])$. We deduce that
$i_*^4(b+i_*^2(y))=i_*^4(b)-i_*^3(a)=0$. So $b+i_*^2(y)=\partial(v)=i_*^2\delta(v)$ for some $v\in\G_1([V\,/\,T])$.
Thus $b=i_*^2(-y+\delta(v))$. Hence
$$(a,b)=(i_*^1(y), i_*^2(-y+\delta(v)))=\alpha(y-\delta(z)).$$
\end{proof}
\end{appendices}

\end{document}